\documentclass[14t,a4paper]{amsart}
\usepackage{amsmath,amsthm,amssymb,stmaryrd,latexsym,a4wide,epic,eepic,bbm,mathrsfs,amsfonts,cases,yfonts,indentfirst,bbold}
\usepackage{dsfont,color,mathdots}

\usepackage{amsmath,amssymb}

\usepackage{amsfonts}
\usepackage{tipa}

\usepackage{CJK, CJKnumb}
\usepackage{color}      % 支持彩色
\usepackage{indentfirst}        %首行缩进宏包
\usepackage{latexsym, bm}        % 处理数学公式中和黑斜体的宏包

\usepackage{graphicx}
\usepackage{cases}
\usepackage{fancyhdr}
\usepackage{pifont}
\newtheorem{theorem}{Theorem}
\newtheorem{defi}[theorem]{Definition}
\newtheorem{lemma}[theorem]{Lemma}
\newtheorem{coro}[theorem]{Corollary}
\newtheorem{proposition}[theorem]{Proposition}
\newtheorem{example}[theorem]{Example}

\newtheorem{remark}[theorem]{Remark}
\usepackage[all]{xy}
\usepackage[active]{srcltx}
\usepackage{enumerate}
\numberwithin{equation}{section}

\usepackage{tikz-cd}

\begin{document}

\title[Quantum Supersymmetries (II):  Loewy Filtrations and Quantum de Rham Cohomology]{Quantum Supersymmetries (II): Loewy Filtrations and\\ Quantum de Rham Cohomology over\\ Quantum Grassmann Superalgebra}

\author[G. Feng]{Ge Feng}
\address{College of Science, University of Shanghai for Science and Technology, Shanghai 200093, China}
\email{fengge@usst.edu.cn}

\author[N.H. Hu]{Naihong Hu}%$^*$}
\address{School of Mathematical Sciences, MOE Key Laboratory of Mathematics and Engineering Applications \& Shanghai Key Laboratory of PMMP, East China Normal University, Shanghai 200241, China}
\email{nhhu@math.ecnu.edu.cn}

\author[M. Rosso]{Marc Rosso}
\address{UFR Mathematics de Universit\'e Paris Cit\'e, CNRS, IMJ-PRG, B\^atiment Sophie Germain – 8 place Aur\'elie Nemours, 75013 Paris}
\email{marc.rosso@imj-prg.fr}

\thanks{%$^*$Corresponding author. 
The paper is supported by the NNSFC (Grant No. 12171155) and in part by Science and Technology Commission of Shanghai Municipality (No. 22DZ2229014).}

\begin{abstract}
We explore the indecomposable submodule structure of quantum Grassmann super-algebra $\Omega_q(m|n)$ and its truncated objects $\Omega_q(m|n,\textbf{r})$ in the case when $q=\varepsilon$ is an $\ell$-th root of unity. A net-like weave-lifting method is developed to show the indecomposability of all the homogeneous super subspaces $\Omega_q^{(s)}(m|n,\textbf{r})$ and $\Omega_q^{(s)}(m|n)$ as $\mathcal U_q(\mathfrak{gl}(m|n))$-modules by defining ``energy grade" to depict their ``$\ell$-adic" phenomenon. Their Loewy filtrations are described, the Loewy layers and dimensions are determined by combinatorial identities. The quantum super de Rham cochain short complex $(\mathcal D_q(m|n)^{(\bullet)},d^\bullet)$ is constructed and proved to be acyclic (Poincar\'e Lemma), where $\mathcal D_q(m|n)=\Omega_q(m|n)\otimes \sqcap_q(m|n)$ and  $\sqcap_q(m|n)$ is the quantum exterior super-algebra, over which we define the $q$-differentials. %such that the product structure of $\sqcap_q(m|n)$, the quantum exterior super-algebra, is well-matched everywhere. 
However, the truncated quantum de Rham cochain subcomplexes $(\mathcal D_q(m|n,\textbf{r})^{(\bullet)},d^\bullet)$
we mainly consider are no longer acyclic and %calculating 
the resulting quantum super de Rham cohomologies $H^s_{DR}(\mathcal D_q(m|n, \mathbf r)^{(\bullet)})$ are highly nontrivial.
\end{abstract}
\keywords{Quantum Grassmann superalgebra, Energy grade, Indecomposability, Loewy filtration, Rigidity, Quantum super de Rham cochain complex, Quantum de Rham cohomology}

\subjclass{Primary 17A70, 17B10, 17B37; Secondary 81R60, 81T70, 81T75}
\maketitle

\tableofcontents

\section{Introduction}

\noindent{\bf 1.1.}
This paper is a continuation of \cite{FHRZ}, where the authors introduced the notion of quantum (resp., dual) Grassmann superalgebra $\Omega_q(m|n)$ (resp., $\Lambda_q(m|n)$) by defining the quantum $(m|n)$-superspace $A_q^{m|n}$ and its Manin dual superspace $(A_q^{m|n})^!$. Furthermore, the quantum general linear superalgebra $U_q(\mathfrak{gl}(m|n))$ is realized as a certain quantum differential operators superalgebra defined over  $A_q^{m|n}$ and  $(A_q^{m|n})^!$, so the quantum (resp., restricted) Grassmann superalgebra $\Omega_q(m|n)$ (resp., $\Omega_q(m|n; \mathbf 1)$) and dual quantum (resp., restricted) Grassmann superalgebra $\Lambda_q(m|n)$ (resp., $\Lambda_q(m|n; \mathbf 1)$) are made into $U_q(\mathfrak{gl}(m|n))$-module (also $u_q(\mathfrak{gl}(n))$-module) superalgebras (resp. in the root of unity cases), respectively, which provides a construction of a quantized version of the super de Rham complex of infinite length due to Manin (\cite{M2}) and Deligne-Morgan (\cite{DM}).

\smallskip
\noindent{\bf 1.2} \ 
The purpose of this paper is two-fold: One is to adopt a socle filtration analysis to study the ``modular" representation theory of $U_q(\mathfrak{gl}(m|n))$ or $u_q(\mathfrak{gl}(m|n))$ in the root of unity case. Another is to construct the quantum super de Rham short cochain complex $\mathcal D_q(m|n)^{(\bullet)}$  of length $m{+}n$ and its subcomplex $\mathcal D_q(m|n,\mathbf r)^{(\bullet)}$,  the truncated quantum super de Rham short subcomplex, and develop an inherent technique adapted to the quantum super case to calculate their cohomologies (as such calculations in our supercase are highly nontrivial and certainly interesting, and require to find new insights for dealing with the supercase). 

\smallskip
\noindent{\bf 1.3} \ For the first goal, we will explore the submodules constituents of homogenous
$\mathcal U_q$-modules $\Omega^{(s)}_q(m|n)$ and $\Omega_q^{(s)}(m|n;\mathbf r)$. To do that, first of all, 
we introduce the concept of ``{\it energy grade}"--- a grade modulo $\ell$ (here $\ell=\textbf{char}(q)$) over the quantum Grassman super-algebra $\Omega_q(m|n)$ which depicts the ``$\ell$-adic" phenomenon (see Lemmas 9, 10, 11 \& 12, together with their proofs)  shaping such a stratification of submodules (see Theorems 14, 16 \& 17). This only appears in the root of unity cases.
We then describe the Loewy filtrations (Theorem 23) for all indecomposable submodules $\Omega_q^{(s)}(m|n)$ and $\Omega_q^{(s)}(m|n; \mathbf r)$, as well as the Loewy layers and their dimensions determined by combinatorial identities (Corollary 24), 
and prove the rigidity of these modules (Theorem 26). Particularly, a net-like weave-lifting method  is developed to show the indecomposability of all the homogeneous super subspaces $\Omega_q^{(s)}(m|n,\textbf{r})$ and $\Omega_q^{(s)}(m|n)$ as $\mathcal U_q$-modules, where $\mathcal U_q=U_q(\mathfrak{gl}(m|n))$, or $u_q(\mathfrak{gl}(m|n))$ when $\text{char}\,(q)=\ell\ge 3$.  (See Corollary 22, and some net-like inclusion relationships for submodules exhibited in Figure 1 of Example 4.6)

\smallskip
\noindent{\bf 1.4} \ 
For the second goal, a challenging task is to find out the way to define/construct accurately the quantum differentials (Definition 30) that lead to the definition of the quantum super de Rham cochain short complex (Theorem 31) as expected in the quantum super-case. To our knowledge, we haven't seen any classical analogue in the cohomology theory of Lie superalgebras (see \cite{N, N1, SZ, Zh}, etc.). Both the quantum Grassmann superalgebra $\Omega_q(m|n)$ and the quantum exterior superalgebra $\sqcap_q(m|n)$ being $\mathcal U_q$-module super-algebras, taking their tensor product, one gets a new $\mathcal U_q$-module $\mathcal D_q(m|n):=\Omega_q(m|n)\otimes \sqcap_q(m|n)$. We thus get the expected quantum super de Rham cochain short complex $(\mathcal D_q(m|n)^{(\bullet)}, d^\bullet)$, which turns out to be acyclic (Actually, this is Poincar\'e Lemma of quantum super de Rham complex). However, remarkably, the quantum truncated super de Rham cohomologies $H^s_{DR}(\mathcal D_q(m|n, \mathbf r)^{(\bullet)})$ are highly nontrivial (see Theorem 36). The argumentation of the non-truncated case depends technically on using a ``modular'' trick (see the proof of Theorem 37) after studying the truncated case first. As for the proof of Main Theorem 36, we need to introduce the concept of ``super-weights" in Definition 33, and have to distinguish the ``critical super-weights" and ``non-critical super-weights" (see Definition 35). The complex restricted to those non-critical super-weights is always exact, meanwhile only those with critical super-weights have nonzero contributions to each $s$-th cohomology group, for $0\le s\le m+n$. Besides, we give in Lemma 34 a crucial criterion to distinguish the specific constituents of quantum differential forms of degree $s$ in $\mathcal D_\lambda^{(s)}$, for each super-weight $\lambda$, when considering each component of quantum de Rham cohomologies in the supercase. A remarkable point about the discovery of two kinds of quantum super de Rham long/short complexes reveals a
definite fact that the quantization deformation of the classical differential partial operators is of chirality, i.e., we have two different definitions for quantum differential operators (see
those defined in the proof in Theorem 30 \& Theorem 20 \cite{FHRZ}).
 
\section{Preliminaries}
\subsection{Notations}
Throughout the paper, we work on an algebraically closed field $\mathbb{K}$ of characteristic $0$. Denote by $\mathbb{Z}_+$, $\mathbb{N}$ the set of nonnegative integers, positive integers, respectively.

For any $m,n\in \mathbb{Z}_+$, we consider the set $I=\{1,2,\cdots,m+n\}$ with convention that $I=\emptyset$ if $m=n=0$. If we set $I_0:=\{1,2,\cdots,m\}$, $I_1:=\{m+1,m+2,\cdots,m+n\}$, then $I=I_0\cup I_1$. Denote by $J$ the set $\{1,2,\cdots,m+n-1\}$ with convention that $J=\emptyset$ if $m+n< 2$.

For any $i,j\in I$, let $E_{ij}$ be the matrix unit of size $(m+n)\times (m+n)$ with $1$ in the $(i,j)$ position, and $0$ in others. For a $\mathbb{Z}_2$-graded vector space $V:=V_{\bar{0}}\oplus V_{\bar{1}}$, denoted by $\bar{v}$ the parity of the homogeneous element $v\in V$.

\subsection{General linear Lie superalgebras}
The general linear Lie superalgebra $\mathfrak{gl}(m|n)$ has a standard basis $E_{ij}$, $ i,j\in I$. Under this basis, we get the graded structure  
$$
\mathfrak{gl}(m|n)=\mathfrak{gl}(m|n)_0\bigoplus \mathfrak{gl}(m|n)_1,
$$
where
$\mathfrak{gl}(m|n)_{\bar{0}}:=\text{Span}\,\{E_{ij}\mid i,j\in I_0\ \textit{or}\ i,j\in I_1\}$, $\mathfrak{gl}(m|n)_{\bar{1}}:=\text{Span}\,\{E_{ij}\mid i\in I_0,j\in I_1\ \textit{or}\ i\in I_1,j\in I_0\}$.

If $m=0$ $(\text{resp.}, \,n=0)$, then $\mathfrak{gl}(m|n)$ is exactly the general linear Lie algebra $\mathfrak{gl}(n)$ $(\text{resp.},\, \mathfrak{gl}(m))$.

The standard Cartan subalgebra $\mathfrak{h}$ of $\mathfrak{g}$ consists of all diagonal matrices in $\mathfrak{g}$ (where $\mathfrak g=\mathfrak{gl}(m|n)$), that is, the $\Bbbk$-span of $E_{ii},i\in I$. For each $i\in I$, we
denote by $\epsilon_i$ the dual of $E_{ii}$, which forms a basis of $\mathfrak{h}^\ast$.
The root system of $\mathfrak{g}$ with respect to $\mathfrak{h}$ is $\Delta:=\Delta_{\overline{0}}\cup\Delta_{\overline{1}}$, where
\begin{displaymath}
\Delta_{\overline{0}}:= \{\epsilon_i-\epsilon_j\mid\ i,j\in I_{\overline{0}} \ \textrm{or} \ i,j\in I_{\overline{1}}\}\quad\text{and}\quad
\Delta_{\overline{1}}:= \{\pm(\epsilon_i-\epsilon_j)\mid\  i\in I_{\overline{0}}, j\in I_{\overline{1}}\},
\end{displaymath}
 and its standard fundamental (simple root) system is given by the set
 \begin{displaymath}
 \Pi:=\{\alpha_i:=\epsilon_i-\epsilon_{i+1}\mid 1\leq i< m{+}n\}.
 \end{displaymath}
 Let $\Lambda:=\mathbb{Z}\epsilon_1\bigoplus\mathbb{Z}\epsilon_2\bigoplus\cdots\bigoplus\mathbb{Z}\epsilon_{m+n}$. For each $i\in I$
 denote by $\omega_i$ the fundamental weight $\epsilon_1+\epsilon_2+\cdots+\epsilon_{i}$.
From \cite{Mu}, we see that there is a symmetric bilinear form on $\Lambda$:
\begin{displaymath}
(\epsilon_i,\epsilon_j)=\left\{\begin{array}{ll} \delta_{ij} \quad & \mathrm{if} \ 1\leq i\leq m,\\-\delta_{ij} \quad & \mathrm{if} \ m+1\leq i \leq m{+}n,
\end{array}\right.
\end{displaymath}
which is induced from the supertrace $\mathfrak{str}$ on $\mathfrak{g}$ defined as:
$\mathfrak{str}(g):=\mathrm{tr}(a)-\mathrm{tr}(d)$ if
\begin{displaymath}
g=\left(\begin{array}{cc}
a & b \\
c & d
\end{array}\right)
\end{displaymath}
with $a, b, c$ and $d$ being $m\times m,m\times n, n\times m$ and $n\times n$ matrices, respectively.
\subsection{Quantum general linear superalgebras $U_q(\mathfrak{gl}(m|n))$}
Let $\mathbb{K}^*:=\mathbb{K}\setminus \{0\}$, $1\neq q\in \mathbb{K}^*$. The quantum general linear superalgebra $U_q(\mathfrak{gl}(m|n))$ is defined as the $\mathbb{K}$-superalgebra with generators $e_j, f_j\ (j\in J), K_i,K_i^{-1}\ (i\in I)$ and relations:
\begin{equation*}
\begin{split}
\text{\rm (R1)} &\qquad K_iK_j=K_jK_i,\qquad K_iK_i^{-1}=1=K_i^{-1}K_i;\\
\text{\rm (R2)} &\qquad K_ie_j=q_i^{\delta_{ij}-\delta_{i,j+1}}e_jK_i,\qquad K_if_j=q_i^{\delta_{i,j+1}-\delta_{ij}}f_jK_i;\\
\text{\rm (R3)} &\qquad e_if_j-(-1)^{p(e_i)p(f_j)}f_je_i=\delta_{ij}\frac{\mathcal{K}_i-\mathcal{K}_i^{-1}}{q_i-q_i^{-1}};\\
\text{\rm (R4)} &\qquad e_ie_j=e_je_i,\qquad f_if_j=f_jf_i,\qquad |i-j|>1;\\
\text{\rm (R5)} &\qquad e_i^2e_j-(q+q^{-1})e_ie_je_i+e_je_i^2=0,\qquad |i-j|=1, \ i\neq m;\\
                &\qquad f_i^2f_j-(q+q^{-1})f_if_jf_i+f_jf_i^2=0,\qquad |i-j|=1,\ i\neq m;\\
\text{\rm (R6)} &\qquad e_m^2=f_m^2=0;
\\
\text{\rm (R7)} &\qquad e_{m-1}e_me_{m+1}e_m+e_me_{m-1}e_me_{m+1}+e_{m+1}e_me_{m-1}e_m\\
                &\qquad +e_me_{m+1}e_me_{m-1}-(q+q^{-1})e_me_{m-1}e_{m+1}e_m=0;\\
                &\qquad f_{m-1}f_mf_{m+1}f_m+f_mf_{m-1}f_mf_{m+1}+f_{m+1}f_mf_{m-1}f_m\\
                &\qquad +f_mf_{m+1}f_mf_{m-1}-(q+q^{-1})f_mf_{m-1}f_{m+1}f_m=0;
\end{split}
\end{equation*}
where $\mathcal{K}_i=K_iK_{i+1}^{-1}$ for $i\in J$, $p(e_i)=p(f_i)=\delta_{im}$, and
$$q_i=\begin{cases}
      q,&\qquad i\in I_0,\\
      q^{-1},&\qquad i\in I_1.
      \end{cases}$$
Moreover, the Hopf superalgebra structure $(\Delta, \epsilon, S)$ on $U_q(\mathfrak{gl}(m|n))$ is given as follows: for any $i\in I,\, j\in J$:
\begin{equation}
\begin{split}
& \Delta(e_j)=e_j\otimes \mathcal{K}_j+1\otimes e_j,\qquad \epsilon(e_j)=0,\qquad S(e_j)=-e_j\mathcal{K}_j^{-1},\\
& \Delta(f_j)=f_j\otimes 1+\mathcal{K}_j^{-1}\otimes f_j,\qquad \epsilon(f_j)=0,\qquad S(f_j)=-\mathcal{K}_jf_j,\\
& \Delta(K_i^{\pm1})=K_i^{\pm1}\otimes K_i^{\pm1},\qquad \epsilon(K_i^{\pm1})=1,\qquad S(K_i^{\pm1})=K_i^{\mp1}.
\end{split}
\end{equation}

By definition, it is obvious that $U_q(\mathfrak {gl}_m)\subset U_q(\mathfrak{gl}(m|n))$ and 
 $U_{q^{-1}}(\mathfrak {gl}_n)\subset U_q(\mathfrak{gl}(m|n))$.

\subsection{Arithmetic properties of $q$-binomials}
Let $\mathbb{Z}[v,v^{-1}]$ be the Laurant polynomial ring in variable $v$. For any $n\geq 0$, define
\begin{displaymath}
[n]_v=\frac{v^n-v^{-n}}{v-v^{-1}}, \qquad [n]_v!=[n]_v[n-1]_v\cdots[1]_v.
\end{displaymath}
Obviously, $[n]_v$, $[n]_v!\in \mathbb{Z}[v,v^{-1}]$.

In \cite{Lus}, for integers $m,\, r\geq 0$, denote by
\begin{displaymath}
{\,m\,\brack\, r\,}_v=\prod\limits_{i=1}^{r}\frac{v^{m-i+1}-v^{-m+i-1}}{v^i-v^{-i}}=\in\mathbb{Z}[v,v^{-1}].
\end{displaymath}
We have

\text{\rm (1)} For $0\leq r\leq m$, ${\,m\,\brack\, r\,}_v=\frac{[m]_v!}{([r]_v![m-r]_v!)};$

\text{\rm (2)} For $0\leq m< r$, ${\,m\,\brack\, r\,}_v=0$;

\text{\rm (3)} For $m< 0$, ${\,m\,\brack\, r\,}_v=(-1)^r{\,-m+r-1\,\brack\, r\,}_v$;

\text{\rm (4)} Set ${\,m\,\brack\, r\,}_v=0$, when $r< 0$.

For $q\in \mathbb{K}^*$, we briefly set
$$[n]:=[n]_{v=q},\qquad [n]!:=[n]_{v=q}!,\qquad {\,m\,\brack\, r\,}={\,m\,\brack\, r\,}_{v=q},$$
when $v$ is specialized to $q$, where $q$-binomials satisfy
$${\,m\,\brack\, r\,}=q^{r-m}{\,m-1\,\brack\, r-1\,}+q^r{\,m-1\,\brack\, r\,}.$$
Define the characteristic of $q$, $\textbf{char}(q):=\text{min}\{\ell\mid[\ell]=0, \ell\in\mathbb{Z}_{\geq 0}\}$. $\textbf{char}(q)=0$ if and only if $q$ is generic. If $\textbf{char}(q)=\ell>0$, and $q\neq\pm 1$, then either

\text{\rm (1)} $q$ is the $2\ell$-th primitive root of unity; or

\text{\rm (2)} $\ell$ is odd and $q$ is the $\ell$-th primitive root of unity.

\begin{lemma} $($\cite{Lus, GH}$)$
When $q\in \mathbb{K}^\ast$, and $\textbf{char}(q)=\ell\geq 3$, then

\text{\rm (1)} If $s=s_0+s_1\ell$, $r=r_0+r_1\ell$ with $0\leq s_0,r_0< \ell$, $s_1,r_1\in \mathbb{Z}_+$, and $s\geq r$, then  ${\,s\,\brack\, r\,}={\,s_0\,\brack\, r_0\,}\binom{s_1}{r_1}$ when $q$ is an $\ell$-$th$ primitive root of unity with $\ell$ odd and ${\,s\,\brack\, r\,}=(-1)^{(s_1+1)r_1\ell+s_0r_1-r_0s_1}{\,s_0\,\brack\, r_0\,}\binom{s_1}{r_1}$ when $q$ is a $2\ell$-$th$ primitive root of unity, where $\binom{s_1}{r_1}$ is an ordinary binomial coefficient.

\text{\rm (2)} If $s=s_0+s_1\ell$, with $0\leq s_0< \ell$, $s_1\in \mathbb{Z}$, then ${\,s\,\brack\, \ell\,}=s_1$, when $q$ is an $\ell$-$th$ primitive root of unity with $\ell$ odd and ${\,s\,\brack\, \ell\,}=(-1)^{(s_1+1)\ell+s_0}s_1$, when $q$ is a $2\ell$-th primitive root of unity.

\text{\rm (3)} If $s=s_0+s_1\ell$, $s^{\prime}=s_0^{\prime}+s_1^{\prime}\ell\in \mathbb{Z}$ with $0\leq s_0,s_0^{\prime}< \ell$ satisfy $q^s=q^{s^{\prime}}$, ${\,s\,\brack\, \ell\,}={\,s^{\prime}\,\brack\, \ell\,}$, then $s=s^{\prime}$ if $\ell$ is odd or $\ell$ is even but $s_1s_1^{\prime}\geq 0$, and $s^{\prime}=\bar{s}=s_0-s_1\ell$ if $\ell$ is even but $s_1s_1^{\prime}< 0$.
\end{lemma}

\subsection{Quantum (restricted) divided power algebras} A bi-addictive map $\ast:\mathbb{Z}^m\times \mathbb{Z}^m\rightarrow \mathbb{Z}$ is defined as follows:
$$\beta\ast \gamma=\sum\limits_{j=1}^{m-1}\sum\limits_{i>j}\beta_i\gamma_j,$$
for any $\beta=(\beta_1,\cdots,\beta_m)$, $\gamma=(\gamma_1,\cdots,\gamma_m)\in \mathbb{Z}^m$.

For $q\in \mathbb{K}^{\ast}$, the $quantum\ divided\ power\ algebra\ \mathcal{A}_q(m)$ is defined as a $\mathbb{K}$-vector space with a monomial basis $\{x^{(\beta)}\mid \beta\in \mathbb{Z}_+^m\}$ with $x^{(0)}=1$. The multiplication is given by
$$x^{(\beta)}x^{(\gamma)}=q^{\beta\ast\gamma}{\,\beta+\gamma\,\brack\, \beta\,}x^{(\beta+\gamma)},$$
where ${\,\beta+\gamma\,\brack\,\beta\,}:=\prod_{i=1}^{m}{\,\beta_i+\gamma_i\,\brack\,\beta_i\,}$ and ${\,\beta_i+\gamma_i\,\brack\,\beta_i\,}=[\,\beta_i+\gamma_i\,]!/[\,\beta_i\,]![\,\gamma_i\,]!$ for any $\beta_i$,$\gamma_i\in \mathbb{Z}_+$.
Briefly, we denote $x^{(\epsilon_i)}$ by $x_i$, for each $1\leq i\leq m$, with $\epsilon_i=(0,\cdots,0,\underset{i}1,0,\cdots,0)$.

In particular, when $\textbf{char}(q)=\ell\geq 3$, $\mathcal{A}_q(m,\mathbf{1})$ is a subalgebra of
$\mathcal{A}_q(m)$ with the basis $\{\,x^{(\beta)}\mid \beta\in \mathbb{Z}_+^m,\beta \leq \mathbf 1\,\}$,
where $\mathbf 1:=(\ell{-}1,\cdots,\ell{-}1)\in \mathbb{Z}_+^m$, called the quantum $restricted$ divided power algebra.
Here $\beta\leq \gamma$ means $\beta_i\leq \gamma_i$ for all $i$. Since $[\,\ell\,]=0$, it is clear that
$\dim\mathcal{A}_q(m,\mathbf{1})=\ell^m$.

\subsection{Quantum exterior superalgebra}
The {\it quantum exterior algebra} is defined as the quotient of the free associative algebra $\mathbb{K}\{x_{m+1},\cdots,x_{m+n}\}$ by the quadratic ideal satisfying the quantum antisymmetric relations as follows:
$$
\Lambda_q(n):=\mathbb{K}\{\,x_{m+1},\cdots,\,x_{m+n}\,\}/\langle\, x_i^2,\, x_jx_i+q^{-1}x_ix_j,\ j>i,\ i,\, j\in I_1\,\rangle,
$$
which is a comodule-algebra of the quantum general linear group $GL_q(n)$, dually, a module-algebra of the quantum algebra $U_q(\mathfrak{gl}(n))$.

\subsection{Quantum Manin superspace $A_q^{m|n}$}
As a generalization of Manin's quantum $m|0$-space $A_q^{m|0}$ and $0|n$-space $A_{q^{-1}}^{0|n}$ (see \cite{M}),  we started in \cite{FHRZ} to introduce the notation of {\it quantum Manin superspace $A_q^{m|n}$} on a $(m+n)$-dimensional $\mathbb Z_2$-graded vector representation $\textbf{V}$ of $U_q(\mathfrak{gl}(m|n))$.
\begin{defi}
Associated with the natural $U_q(\mathfrak{gl}(m|n))$-module $\textbf{V}$, the quantum Manin superspace $A_q^{m|n}$ is defined to be the quotient of the free associative algebra $\mathbb{K}\{v_i\mid i\in I\}$ over $\mathbb{K}$ by the quadratic ideal $I(\textbf{V})$ generated by
\begin{gather*}
v_jv_i-q\,v_iv_j,\quad \textit{for \ }\ i\in I_0,\ j\in I, \ j>i;\\
v_i^2,\ v_jv_i+q\,v_iv_j,\quad \textit{for \ }\ i,\ j\in I_1, \ j>i.
\end{gather*}
\end{defi}
As an associative $\mathbb{K}$-superalgebra, it has a monomial basis consisting of $\{v^{\langle\alpha,\mu\rangle}:=v^\alpha\otimes v^\mu\mid \alpha=(\alpha_1,\cdots,\alpha_m)\in \mathbb{Z}_+^m, \mu=\langle\mu_1,\cdots,\mu_n\rangle\in \mathds{Z}_2^n\}$, where $v^\alpha=v_1^{\alpha_1}\cdots v_m^{\alpha_m}$ and $v^\mu=v_{m+1}^{\mu_1}\cdots v_{m+n}^{\mu_n}$. By abuse of notation, it is convenient for us to consider $(m{+}n)$-tuple $\langle \alpha,\mu\rangle\in \mathbb{Z}_+^{m+n}$, for any $\alpha=(\alpha_1,\cdots,\alpha_m)\in \mathbb{Z}_+^m, \mu=\langle \mu_1,\cdots,\mu_n\rangle\in \mathds{Z}_2^n$. Then we extend the product $\ast$ to the $(m{+}n)$-tuples in $\mathbb{Z}^{m+n}$, satisfying the following relation:
$$
\langle\alpha,\mu\rangle\ast\langle\beta,\nu\rangle
=\alpha\ast\beta+\mu\ast\nu+\mu\ast\beta=\alpha\ast\beta+\mu\ast\nu+|\,\mu\,|\cdot|\,\beta\,|
\,,
$$
with $\alpha, \, \beta\in \mathbb{Z}_+^m$, $\mu,\nu\in \mathds{Z}_2^n$, where
$|\,\mu\,|=\sum_i\mu_i$ and $|\,\beta\,|=\sum_j\beta_j$.

\subsection{Quantum Grassmann superalgebra $\Omega_q(m|n)$ as $\mathcal U_q(\mathfrak{gl}(m|n))$-module superalgebra}

Let us denote by $\mathcal{U}_q:=U_q(\mathfrak{gl}(m|n))$, $U_q(\mathfrak{sl}(m|n))$ respectively, or the restricted object --- the small quantum superalgebra $u_q(\mathfrak{gl}(m|n))$, $u_q(\mathfrak{sl}(m|n))$, respectively, in the root of unity case when $\textbf{char}(q)=\ell>2$. They have the same generators.

Consider the tensor space of the quantum divided power algebra and the quantum exterior algebra, which is the quantum divided power version of the quantum Manin superspace $A_q^{m|n}$ we defined above,
$$\Omega_q(m|n):=\mathcal{A}_q(m)\otimes_{\mathbb{K}}\Lambda_{q^{-1}}(n).$$
\begin{defi} $($\cite{FHRZ}$)$
The quantum Grassmann superalgebra $\Omega_q(m|n)$ is defined as an associative $\mathbb{K}$-superalgebra on a superspace over $\mathbb{K}$ with the multiplication given by
\begin{equation}
(x^{(\alpha)}\otimes x^\mu)\star(x^{(\beta)}\otimes x^\nu)=q^{\mu\ast\beta}x^{(\alpha)}x^{(\beta)}\otimes x^\mu x^\nu,
\end{equation}
for any $x^{(\alpha)},x^{(\beta)}\in \mathcal{A}_q(m), x^\mu,x^\nu\in \Lambda_{q^{-1}}(n)$. 

When $\textbf{char}(q)=\ell\geq 3$, $\Omega_q(m|n,\mathbf {1}):=\mathcal{A}_q(m,\mathbf{1})\otimes \Lambda_{q^{-1}}(n)$ is a sub-superalgebra, we call it the quantum restricted Grassmann superalgebra.
\end{defi}
Clearly, 
$\{\,x^{(\alpha)}\otimes x^\mu\mid \alpha\in \mathbb{Z}_+^m, \ \mu\in \mathds{Z}_2^n\,\}
$ forms a monomial $\mathbb{K}$-basis of $\Omega_q(m|n)$, and the set
$$
\{\,x^{(\alpha)}\otimes x^\mu\mid \alpha\in \mathbb{Z}_+^m, \ \alpha\leq \textbf{1}, \ \mu\in \mathbb{Z}_2^n\,\}
$$
forms a monomial $\mathbb{K}$-basis of $\Omega_q(m|n,\textbf{1})$.
The $\mathcal{U}_q$-module algebra structure over the quantum (restricted) Grassmann superalgebra $\Omega_q(m|n)$ ($\Omega_q(m|n,\textbf{1})$) is defined by Theorem 27 \cite{FHRZ}, where $\mathcal{U}_q:=U_q(\mathfrak{gl}(m|n))$, $U_q(\mathfrak{sl}(m|n))$ respectively, or the restricted object $u_q(\mathfrak{gl}(m|n))$, $u_q(\mathfrak{sl}(m|n))$, respectively, when $\textbf{char}(q)=\ell>2$.

The following $\mathcal U_q$-action formulae will be extremely useful to analysing the submodules structure of $\Omega_q(m|n)$, especially for $\Omega_q^{(s)}(m|n)$ as $\mathfrak u_q(\mathfrak{gl}(m|n))$-modules in the root of unity case.
\begin{proposition} $($Theorem 27 \cite{FHRZ}$)$
For any $x^{(\alpha)}\otimes x^\mu \in \Omega_q(m|n)$, we have the following acting formulae:

\text{\rm (i)} for $i\in I_0 \setminus\{ m\}$, we have
\begin{equation}
\begin{split}
e_i(x^{(\alpha)}\otimes x^\mu)&=[\,\alpha_i{+}1\,]\,x^{(\alpha+\varepsilon_i-\varepsilon_{i+1})}\otimes x^\mu,\\
f_i(x^{(\alpha)}\otimes x^\mu)&=[\,\alpha_{i+1}{+}1\,]\,x^{(\alpha-\varepsilon_i+\varepsilon_{i+1})}\otimes x^\mu,\\
K_i(x^{(\alpha)}\otimes x^\mu)&=q^{\alpha_i}x^{(\alpha)}\otimes x^\mu,\\
\mathcal{K}_i(x^{(\alpha)}\otimes x^\mu)&=q^{\alpha_i-\alpha_{i+1}}x^{(\alpha)}\otimes x^\mu;
\end{split}
\end{equation}

\text{\rm (ii)} for $i=m$, we have
\begin{equation}
\begin{split}
e_m(x^{(\alpha)}\otimes x^\mu)&=\delta_{1,\mu_1}[\,\alpha_m{+}1\,]\,x^{(\alpha+\varepsilon_m)}\otimes x^{\mu-\varepsilon_{m+1}},\\
f_m(x^{(\alpha)}\otimes x^\mu)&=\delta_{0,\mu_1}x^{(\alpha-\varepsilon_m)}\otimes x^{\mu+\varepsilon_{m+1}},\\
K_m(x^{(\alpha)}\otimes x^\mu)&=q^{\alpha_m}x^{(\alpha)}\otimes x^\mu,\\
K_{m+1}(x^{(\alpha)}\otimes x^\mu)&=q^{-\mu_1}x^{(\alpha)}\otimes x^\mu,\\
\mathcal{K}_m(x^{(\alpha)}\otimes x^\mu)&=q^{\alpha_m+\mu_1}x^{(\alpha)}\otimes x^\mu;
\end{split}
\end{equation}

\text{\rm (iii)} for $i\in I_1$, we have
\begin{equation}
\begin{split}
e_i(x^{(\alpha)}\otimes x^\mu)&=\delta_{0,\mu_i}\delta_{1,\mu_{i+1}}x^{(\alpha)}\otimes x^{\mu+\varepsilon_i-\varepsilon_{i+1}},\\
f_i(x^{(\alpha)}\otimes x^\mu)&=\delta_{1,\mu_i}\delta_{0,\mu_{i+1}}x^{(\alpha)}\otimes x^{\mu-\varepsilon_i+\varepsilon_{i+1}},\\
K_i(x^{(\alpha)}\otimes x^\mu)&=q^{-\mu_i}x^{(\alpha)}\otimes x^\mu ,\\
\mathcal{K}_i(x^{(\alpha)}\otimes x^\mu)&=q^{\mu_{i+1}-\mu_i}x^{(\alpha)}\otimes x^\mu, \quad i\neq m{+}n,
\end{split}
\end{equation}
where $e_i,\, f_i,\, K_j$ $(1\leq i< m{+}n$, $1\leq j\leq m{+}n)$ are the generators of $U_q(\mathfrak{gl}(m|n))$ with $q$ generic or $\mathfrak u_q(\mathfrak{gl}(m|n))$ in the root of unity case.
\end{proposition}

Write $x^{\langle\alpha,\mu\rangle}:=x^{(\alpha)}\otimes x^\mu$ with $\alpha\in \mathbb{Z}_+^m,\ \mu\in \mathbb{Z}_2^n$ for brevity. Denote by $|\alpha|=\Sigma \alpha_i,\, |\mu|=\Sigma \mu_i$ the degree of $x^{(\alpha)}$ and $x^\mu$, respectively, then $|\langle \alpha,\mu\rangle|=|\alpha|+|\mu|=\Sigma \alpha_i+\Sigma \mu_i$ is the degree of $x^{(\alpha)}\otimes x^\mu\in \Omega_q(m|n)$. Set $\Omega_q:=\Omega_q(m|n)$ or $\Omega_q(m|n,\textbf{1})=$, where $\textbf{1}=(\ell{-}1,\ell{-}1,\cdots,\ell{-}1)\in \mathbb{Z}_+^m$,
and 
$$
\Omega_q^{(s)}:=\text{Span}_\mathbb{K}\{x^{(\alpha)}\otimes x^\mu\in\Omega_q\mid |\langle \alpha,\mu\rangle|=s\}
$$ 
is the subspace of $\Omega_q$ spanned by homogeneous elements of degree $s$.

Using the notations above, we restate Theorem 28 in \cite{FHRZ} as follows:
\begin{theorem}
Each subspace $\Omega_q^{(s)}(m|n)$ is a $U_q(\mathfrak{gl}(m|n))$-submodule of $\Omega_q(m|n)$, and $\Omega^{(s)}_q(m|n,\mathbf{1})$ is a $u_q(\mathfrak{gl}(m|n))$-submodule of $\Omega_q(m|n,\mathbf{1})$.

\text{\rm (1)} If $\textbf{char}(q)=0$, then each submodule $\Omega_q^{(s)}(m|n)\cong V(s\omega_1)$ is a simple module generated by highest weight vector $x^{(\underline{\textbf{s}})}\otimes 1$, where $\underline{\textbf{s}}=s\epsilon_1=(s,0,\cdots,0)$.

\text{\rm (2)} If $\textbf{char}(q)=\ell>2$, then

\text{\rm (i)} when $s=(j{-}1)(\ell{-}1)+s_j$ with $0\leq s_j\leq \ell{-}1,\, j\in I_0$, the submodule $\Omega_q^{(s)}(m|n,\mathbf{1})\cong V((\ell{-}1{-}s_j)\omega_{j-1}+s_j\omega_j)$ is a simple module generated by highest weight vector $x^{(\underline{\textbf{s}})}$, where $\underline{\textbf{s}}=(\ell{-}1,\cdots,\ell{-}1, s_j, 0, \cdots, 0)$, and $\omega_j=\epsilon_1+\cdots+\epsilon_j$ is the $j$-th fundamental weight. 

\text{\rm (ii)} when $s=m(\ell{-}1)+j$, $0\leq j\leq n$, the submodule 
$\Omega^{(s)}_q(m|n,\mathbf{1})\cong V((\ell{-}2)\omega_m+\omega_{m+j})$ is a simple module generated by highest weight vector $x^{(\ell{-}1,\cdots,\ell{-}1)}\otimes x_{m+1}\cdots x_{m+j}$.
\end{theorem}

\begin{lemma}
  $\dim\Omega_q^{(s)}(m|n,\mathbf{1})=\sum\limits_{i=0}^s\sum\limits_{j=0}^{\lfloor \frac {s-i}{\ell}\rfloor}(-1)^j\binom{n}{i}\binom{m}{j}\binom{m{+}s{-}i{-}j\ell{-}1}{m{-}1}$, where $\lfloor x\rfloor$ means the integer part of $x\in \mathbb{Q}$.
\end{lemma}
\begin{proof} $\dim\,\Omega_q^{(s)}(m|n,\mathbf{1})=\sum\limits_{i+j=s}\dim \mathcal A_q^{(j)}(m,\mathbf 1)\cdot\dim\,\Lambda_{q^{-1}}(n)_{(i)}$ follows from
Corollary 2.6 of \cite{GH}.
\end{proof}

\subsection{Convention}
In the rest of this paper, we will focus our discussion on the case when $\mathbb{Q}(q)\subseteq \mathbb{K}$, $\textbf{char}(q)=\ell>2$ and $\mathcal U_q=u_q(\mathfrak{gl}(m|n))$ with $m\ge 2$.

\section{Truncated $\mathcal U_q$-modules $\Omega_q^{(s)}(m|n,\textbf{r})$: energy grades and stratification} 

\subsection{Graded truncated objects $\Omega_q^{(s)}(m|n,\textbf{r})$}
%Let $\Omega_q(m|n,\textbf{1})=\text{span}\{x^{(\alpha)}\otimes x^\mu\mid \alpha\leq %\textbf{1} \}$, where 
Now we consider a class of graded truncated objects $\Omega_q^{(s)}(m|n,\textbf{r})$ of $\Omega_q(m|n)$, where $\textbf{r}=(r\ell{-}1,r\ell{-}1,\cdots,r\ell{-}1)\in \mathbb{Z}_+^m, \ r\in \mathbb{N}$, and
\begin{gather*}
\Omega_q(m|n,\textbf{r}):=\text{Span}_\mathbb{K}\{\,x^{(\alpha)}\otimes x^\mu\in \Omega_q(m|n)\mid \alpha\leq \textbf{r}\,\}\subseteq \Omega_q(m|n),\\
\Omega_q^{(s)}(m|n,\textbf{r}):=\text{Span}_\mathbb{K}\{\,x^{(\alpha)}\otimes x^\mu\in \Omega_q(m|n)\mid |\langle\alpha,\mu\rangle|=s\,\},
\end{gather*}
then $\Omega_q(m|n,\textbf{r})=\bigoplus_{s=0}^N\Omega_q^{(s)}(m|n,\textbf{r}),$ where $N=|\textbf{r}|+n=m(r\ell{-}1)+n$.

By Proposition 4 and $\dim\mathcal A_q^{(j)}(m, \mathbf r)$ in Proposition 3.1 in \cite{GH}, similar to Lemma 6, we have
\begin{proposition}
\text{\rm (1)} $\Omega_q^{(s)}(m|n,\textbf{r})$ are $\mathcal U_q$-submodules of $\Omega_q^{(s)}(m|n)$, where $0\leq s\leq N$.

\text{\rm(2)} $\dim\,\Omega_q^{(s)}(m|n,\textbf{r})=\sum\limits_{i=0}^{s}\sum\limits_{j=0}^{\lfloor\frac{s-i}{r\ell}\rfloor}(-1)^j\binom{n}{i} \binom{m}{j}\binom{m{+}s{-}i{-}jr\ell{-}1}{ m{-}1}$.
\end{proposition}
%\begin{proof} (1) follows from Proposition 4.
%\text{\rm (1)} For any $x^{(\alpha)}\otimes x^\mu\in \Omega_q^{(s)}(m|n,\textbf{r})$,

%\text{\rm (i)} when $i\in I_0\setminus \{m\}$, by Proposition 4, we get $e_i.(x^{(\alpha)}\otimes x^\mu),\, f_i.(x^{(\alpha)}\otimes x^\mu),\, K_i^\pm.(x^{(\alpha)}\otimes x^\mu)\in \Omega_q^{(s)}(m|n,\textbf{r})$.

%\text{\rm (ii)} when $i=m$, we consider in two cases:

%If $\mu_1=0$ or $\mu_1=1$ and $\alpha_m=r^{\prime}\ell{-}1$, where $0\leq r^{\prime}\leq r$, then $ e_m.(x^{(\alpha)}\otimes x^\mu)=0 $;

%If $\mu_1=1$ and $\alpha_m<r\ell{-}1$, then $\alpha+\varepsilon_m\leq \textbf{r}$ and $x^{(\alpha+\varepsilon_m)}\otimes x^{\mu-\varepsilon_{m+1}}\in \Omega_q^{(s)}(m|n,\textbf{r}),$ thus $e_m.(x^{(\alpha)}\otimes x^\mu)\in \Omega_q^{(s)}(m|n,\textbf{r})$. Similarly, $f_m.(x^{(\alpha)}\otimes x^\mu),\, K_m^\pm.(x^{(\alpha)}\otimes x^\mu)\in \Omega_q^{(s)}(m|n,\textbf{r})$.

%\text{\rm (iii)} when $i\in I_1$, we also consider in two cases.

%If $\mu_{i+1-m}=0$ or $\mu_{i-m}=1$, then $e_i.(x^{(\alpha)}\otimes x^\mu)=0$;

%If $\mu_{i+1-m}=1$ and $\mu_{i-m}=0$, then $e_i.(x^{(\alpha)}\otimes x^\mu)=x^{(\alpha)} \otimes x^{\mu+\varepsilon_i-\varepsilon_{i+1}}\in \Omega_q^{(s)}(m|n,\textbf{r})$.

%\text{\rm (2)} Similar to the proof of Lemma 6, by Proposition 3.1 in \cite{GH}, we get the dimension formula.
%\end{proof}

\subsection{Energy grades and action rules}
For the quantum (restricted) general linear Lie super-algebras $\mathcal U_q$ at roots of unity, we introduce the trucialconcept ``energy grade" of elements in the $\mathcal U_q$-modules $\Omega_q^{(s)}$ to describe their $\ell$-adic behavior in the so-called ``modular'' representation theory. As will be seen, it captures the essential structural feature of the indecomposable modules in the root of unity case, due to the explicit action formulas given in Proposition 4.% like the case of $u_q(\mathfrak{sl}_n)$ did in \cite{GH}.
\begin{defi}
For any $x^{(\alpha)}\otimes x^\mu\in\Omega_q= \Omega_q(m|n,\textbf{r})$ or $\Omega_q(m|n)$, the {\bf energy grade} of $x^{(\alpha)}\otimes x^\mu $, denoted by $ \text{Edeg}\, (x^{(\alpha)}\otimes x^\mu)$, is defined as the one of $x^{(\alpha)}:$
\begin{equation}
\text{Edeg}\, (x^{(\alpha)}\otimes x^\mu):=\text{Edeg}\, (x^{(\alpha)})=\sum\limits_{i=1}^m\lfloor\frac{\alpha_i}{\ell}\rfloor=\sum\limits_{i=1}^m\text{Edeg}_i\,(x^{(\alpha)}\otimes x^\mu), 
\end{equation}
where the $i$-th component energy grade $\text{Edeg}_i\, (x^{(\alpha)}\otimes x^\mu):=\text{Edeg}_i\,(x^{(\alpha)})=\lfloor\frac{\alpha_i}{\ell}\rfloor, \ {\it for \ } i\in I_0$.

In general, for any $y=\sum\limits_{\alpha}k_\alpha x^{(\alpha)}\otimes x^\mu\in \Omega_q$, we define
\begin{equation}
\text{Edeg}\,(\,y\,):=\max_\alpha\,\{\text{Edeg}\,(x^{(\alpha)}\otimes x^\mu)=\text{Edeg}\,(x^{(\alpha)})\mid\alpha\in \mathbb Z_+^m, \ k_\alpha\in \mathbb K^*\,\}.
\end{equation}
\end{defi}
\begin{lemma}
$\text{Edeg\,} \, (u .\, (x^{(\alpha)}\otimes x^\mu))\leq \text{Edeg}\,(x^{(\alpha)}\otimes x^\mu)$, for any $u\in \mathcal U_q$, $x^{(\alpha)}\otimes x^\mu \in \Omega_q^{(s)}$. In fact, $\text{Edeg}_j(u .\, (x^{(\alpha)}\otimes x^\mu))\leq \text{Edeg}_j(x^{(\alpha)}\otimes x^\mu)$, for each $j$.
\end{lemma}
\begin{proof} It suffices to take $u=e_i$, or $f_i$, $(i\in I)$ for checking the result.

\text{\rm (i)} For $i\in I_0\setminus\{ m\}$,  the result follows directly from Proposition 3.3  \cite{GH}, since $\mathcal U_q(\mathfrak{sl}_m)\subset \mathcal U_q$.

\text{\rm (ii)} For $i=m$, if $\mu_1=0$ or $\mu_1=1$ and $\alpha_m=r^{\prime}\ell{-}1$ $(0\leq r^{\prime}\leq r)$, then $[\alpha_m+1]=0$, and by Proposition 4, $ \text{Edeg}\,(e_m.\, (x^{(\alpha)}\otimes x^\mu))=0$, the result is as required. Otherwise, $\mu_1=1$ and $\alpha_m<r^{\prime}\ell{-}1$ $(0\leq r^{\prime}\leq r)$, then $e_m.\, (x^{(\alpha)}\otimes x^\mu)=[\alpha_m+1]\,x^{(\alpha+\varepsilon_m)} \otimes x^{\mu-\varepsilon_{m+1}}$. Since $\lfloor\frac{\alpha_m+1}{\ell}\rfloor=\lfloor\frac{\alpha_m}{\ell}\rfloor$, $\text{Edeg}\,(e_m.\, (x^{(\alpha)}\otimes x^\mu))= \text{Edeg}\,(x^{(\alpha)}\otimes x^\mu)$.

Similarly, if $\alpha_m=0$, or $\mu_1=1$, then by Proposition 4, $\text{Edeg}\,(f_m.\, (x^{(\alpha)}\otimes x^\mu))=0$, the result is true. Otherwise, $0<\alpha_m\leq r\ell{-}1$ and $\mu_1=0$, then $\lfloor\frac{\alpha_m-1}{\ell}\rfloor\leq\lfloor\frac{\alpha_m}{\ell}\rfloor$, the conclusion is true.

\text{\rm (iii)} For $i \in I_1$, by definition, $\text{Edeg}$ of $x^{(\alpha)}\otimes x^\mu$ concentrates on $x^{(\alpha)}$-part, so the conclusion is clear.  %if $\mu_{i+1-m}=0$ or $\mu_{i-m}=1$, then $ \text{Edeg}\,(e_i.\, (x^{(\alpha)}\otimes x^\mu))=0 $, the result is obvious. Otherwise, $\mu_{i+1-m}=1$ and $\mu_{i-m}=0$, then $ e_i.\, (x^{(\alpha)}\otimes x^\mu)=x^{(\alpha)} \otimes x^{\mu+\varepsilon_i-\varepsilon_{i+1}}$, thus $\text{Edeg}\,(e_i.\, (x^{(\alpha)}\otimes x^\mu))= \text{Edeg}\,(x^{(\alpha)}\otimes x^\mu)$. Similarly, we can check the result for $f_i$.
%By the proof above, we actually show that $\text{Edeg}_j(u.\, (x^{(\alpha)}\otimes x^\mu))\leq %\text{Edeg}_j(x^{(\alpha)}\otimes x^\mu)$, for each $j$ and for any $u\in \mathcal U_q$, %$x^{(\alpha)}\otimes x^\mu \in \Omega_q$.
\end{proof}

\begin{lemma}
Let $x^{(\alpha)}\otimes x^\mu$, $x^{(\beta)}\otimes x^\nu\in \Omega_q^{(s)}$ with $\text{Edeg}\,(x^{(\alpha)}\otimes x^\mu)=\text{Edeg}\,(x^{(\beta)}\otimes x^\nu)$. If $\text{Edeg}_i\,(x^{(\alpha)}\otimes x^\mu)\neq \text{Edeg}_i\,(x^{(\beta)}\otimes x^\nu)$, for some $i\in I_0$, then $x^{(\alpha)}\otimes x^\mu\not\in \mathcal U_q.\, (x^{(\beta)}\otimes x^\nu), \ x^{(\beta)}\otimes x^\nu\not\in \mathcal U_q.\, (x^{(\alpha)}\otimes x^\mu)$.
\end{lemma}
\begin{proof}
If $x^{(\alpha)}\otimes x^\mu\in \mathcal U_q.\, (x^{(\beta)}\otimes x^\nu)\setminus\mathbb K.\,(x^{(\beta)}\otimes x^\nu)$, then $\exists$ $u_k\in \mathcal U_q$, such that
$x^{(\alpha)}\otimes x^\mu=\sum u_k.\,(x^{(\beta)}\otimes x^\nu)$. By definition and Lemma 9, 
$\exists$ $u_{k_0}\in \mathcal U_q$, such that  
\begin{gather*}
\text{Edeg}\,(x^{(\alpha)}\otimes x^\mu)=\text{Edeg}\,\bigl(u_{k_0}.\,(x^{(\beta)}\otimes x^\nu)\bigr)\le \text{Edeg}\,(x^{(\beta)}\otimes x^\nu),\\  
\text{Edeg}_j\,(x^{(\alpha)}\otimes x^\mu)=\text{Edeg}_j\,(u_{k_0}.\,(x^{(\beta)}\otimes x^\nu))\le \text{Edeg}_j\,(x^{(\beta)}\otimes x^\nu), \quad \text{\it for all }\  j. 
\end{gather*}
So it follows from the first assumption that $\text{Edeg}_j\,(x^{(\alpha)}\otimes x^\mu)=\text{Edeg}_j\,(x^{(\beta)}\otimes x^\nu)$ for all $j$. This contradicts the second assumption.
\end{proof}

For $x^{(\alpha)}\otimes x^\mu\in \Omega_q^{(s)}$, put $m_j=\text{Edeg}_j\,(x^{(\alpha)})$,  $|\,\mu\,|=u$, then $s=|\,\alpha\,|+|\,\mu\,|=\sum_im_i\ell+r+u$, where $r=|\,\alpha\,|-\sum_im_i\ell=(i_\alpha-1)(\ell-1)+r_{i_\alpha}$, for some $1\le i_\alpha\le m$, $0\le r_{i_{\alpha}}\le \ell-1$. Denote
\begin{gather}
\gamma_\alpha=(m_1\ell+(\ell-1),\cdots,m_{i_\alpha-1}\ell+(\ell-1),m_{i_\alpha}\ell+r_{i_{\alpha}},m_{i_\alpha+1}\ell,\cdots,m_m\ell), \\
\mu'=(0,\cdots,0,\underset{u}{\underbrace{1,\cdots,1}}),
\end{gather} 
with $|\,\mu'\,|=u=|\,\mu\,|$. Thus we get
\begin{gather}
x^{(\gamma_\alpha)}\otimes x^{\mu’}\in \Omega_q^{(s)}, \\
\text{Edeg}_j\,(x^{(\alpha)}\otimes x^\mu)=\text{Edeg}_j\,(x^{(\gamma_\alpha)}\otimes x^{\mu'}), \quad \textit{\it for\  each\ }\  j\le m. 
\end{gather}
By Proposition 3.5 \cite{GH}, $\exists$ $z_\alpha\in \mathcal U_q(\mathfrak{sl}_m)\subset\mathcal U_q$, such that $z_\alpha.\,x^{(\alpha)}=x^{(\gamma_\alpha)}$. Note that $\Lambda_{q^{-1}}^{(u)}(n)$ is a simple $\mathcal U_{q^{-1}}(\mathfrak{sl}_n)$-module, so $\exists \ z_\mu\in \mathcal U_{q^{-1}}(\mathfrak{sl}_n)\subset\mathcal U_q$, such that $z_\mu.\,x^{\mu}=x^{\mu'}$. So there exists $z_{\alpha,\mu}=c\,z_\mu z_\alpha\in\mathcal U_q$ $(c\in\mathbb K^*)$, such that $z_{\alpha,\mu}.\,(x^{(\alpha)}\otimes x^{\mu})=x^{(\gamma_\alpha)}\otimes x^{\mu'}$. Set $v=|\,\nu\,|=|\,\nu'\,|$.

\begin{lemma} Let $x^{(\alpha)}\otimes x^\mu$, $x^{(\beta)}\otimes x^\nu\in \Omega_q^{(s)}$, such that $\text{Edeg}_i\,(x^{(\alpha)}\otimes x^\mu)=\text{Edeg}_i\,(x^{(\beta)}\otimes x^\nu)$ for each $i\in I_0$.  Then 

$(\text{\rm i})$ $u-v=(i_\beta-i_\alpha)(\ell-1)+r_{i_\beta}-r_{i_\alpha}$. 

$(\text{\rm ii})$ There exist $z,\ z'\in \mathcal U_q$ such that $z.\, (x^{(\alpha)}\otimes x^\mu)=x^{(\beta)}\otimes x^\nu, \ z'.\,(x^{(\beta)}\otimes x^\nu)=x^{(\alpha)}\otimes x^\mu$. In other words, $\mathcal U_q.\, (x^{(\alpha)}\otimes x^\mu)=\mathcal U_q.\, (x^{(\beta)}\otimes x^\nu)$.
\end{lemma}
\begin{proof}
According to the discussion before this Lemma, it suffices to prove that the result holds for
$x^{(\gamma_\alpha)}\otimes x^{\mu'}$, $x^{(\gamma_\beta)}\otimes x^{\nu'}\in\Omega^{(s)}$. (i) follows from $s=|\,\alpha\,|+u=|\,\beta\,|+v$ and the assumption.

(ii) If $u=v$, the result is clear since $\mu'=\nu'$ and $x^{(\gamma_\alpha)}=x^{(\gamma_\beta)}$, by definition.  

Suppose $u<v$ without loss of generality. Then we have $\gamma_\alpha>\gamma_\beta$. By the $f_i$-action formulas in Proposition 4, we can act successively some operators $f_i$'s on $x^{(\gamma_\alpha)}\otimes x^{\mu'}$ to move the superfluous degree components up to $\gamma_\alpha-\gamma_\beta$ from the left to the right, finally we can get nonzero times $x^{(\gamma_\beta)}\otimes x^{\nu'}$. Namely,
$\exists$ $z\in\mathcal U_q$, such that $z.\,(x^{(\gamma_\alpha)}\otimes x^{\mu'})=x^{(\gamma_\beta)}\otimes x^{\nu'}$. 

Similarly, we can use some operators $e_i$'s to act successively on $x^{(\gamma_\beta)}\otimes x^{\nu'}$ in order to move the superfluous degree components up to $v-u$ from the right to the left, then to get $x^{(\gamma_\alpha)}\otimes x^{\mu'}$ (noticing that during the process every $e_i$-acting coefficient in each step is nonzero).
\end{proof}

\begin{lemma}
Let $x^{(\alpha)}\otimes x^\mu, x^{(\beta)}\otimes x^\nu\in\Omega_q^{(s)}$, such that $\text{Edeg}\,(x^{(\alpha)}\otimes x^\mu) > \text{Edeg}\,(x^{(\beta)}\otimes x^\nu)$. If $\text{Edeg}_i\,(x^{(\alpha)}\otimes x^\mu)\geq\text{Edeg}_i\,(x^{(\beta)}\otimes x^\nu)$, for each $1\leq i\leq m$, then $\exists$ $z\in \mathcal U_q$, such that $z.\,(x^{(\alpha)}\otimes x^\mu)=x^{(\beta)}\otimes x^\nu$. That is, $\mathcal U_q.\, (x^{(\beta)}\otimes x^\nu)\subsetneqq \mathcal U_q.\,  (x^{(\alpha)}\otimes x^\mu)$.
\end{lemma}
\begin{proof} 
According to the discussion before Lemma 11, we will prove that the result holds for
$x^{(\gamma_\alpha)}\otimes x^{\mu'}$, $x^{(\gamma_\beta)}\otimes x^{\nu'}\in\Omega^{(s)}$. 

It suffices to prove the result holds under the assumption 
that $\text{Edeg}\,(x^{(\gamma_\alpha)})=\text{Edeg}\,(x^{(\gamma_\beta)})+1$, for an induction on $\text{Edeg}\,(x^{(\gamma_\alpha)})-\text{Edeg}\,(x^{(\gamma_\beta)})$ will give rise to the conclusion as required. 

Clearly, our assumption implies that $\exists$ 
$j\le m$, such that $\text{Edeg}_i\,(x^{(\gamma_\alpha)})=\text{Edeg}_i\,(x^{(\gamma_\beta)})$, for all $i\ne j$, and $\text{Edeg}_j\,(x^{(\gamma_\alpha)})=\text{Edeg}_j\,(x^{(\gamma_\beta)})+1$. 
From $s=|\,\gamma_\alpha\,|+u=|\,\gamma_\beta\,|+v$, namely, 
$$
\ell\cdot \text{Edeg}\,(x^{(\gamma_\alpha)})+(i_\alpha-1)(\ell-1)+r_{i_\alpha}+u=\ell\cdot \text{Edeg}\,(x^{(\gamma_\beta)})+(i_\beta-1)(\ell-1)+r_{i_\beta}+v,
$$ 
we deduce that  $(\ell+u)-v=(i_\beta-i_\alpha)(\ell-1)+r_{i_\beta}-r_{i_\alpha}\le (i_\beta-i_\alpha+1)(\ell-1)-r_{i_\alpha}$, namely,
\begin{equation}
u+i_\alpha(\ell-1)+(r_{i_\alpha}{+}1)=s= v+(i_\beta-1)(\ell-1)+r_{i_\beta}.
\end{equation}

(I) If $i_\beta>i_\alpha$, then (3.7) indicates
\begin{equation}
    u+(r_{i_\alpha}+1)\le v+(i_\beta-i_\alpha)(\ell-1).
\end{equation}

(1) If $i_\alpha<j<m$,  by Proposition 4 (2.3) and the definition of $\gamma_\alpha$, we see that the $f_j$-action on $x^{(\gamma_\alpha)}\otimes x^{\mu'}$ will lead to merely lower down the $i$-th component energy grade by $1$ but keep the others invariant. That is,
$\text{Edeg}_i\,\bigl(f_j.\,(x^{(\gamma_\alpha)}\otimes x^{\mu'})\bigr)=\text{Edeg}_i\,\bigl(x^{(\gamma_\alpha-\varepsilon_j+\varepsilon_{j+1})}\otimes x^{\mu'}\bigr)=\text{Edeg}_i\,\bigl(x^{(\gamma_\beta)}\otimes x^{\nu'}\bigr)$, for any $i$.
So, Lemma 11 gives the result.

If $j=m$ and $i_\alpha=m{-}1$ (so $i_\beta=m$) and $r_{m-1}<\ell{-}1$, then by Proposition 4 (2.3) and the definition of $\gamma_\alpha$, $e_{m-1}$ acts on $x^{(\gamma_\alpha)}\otimes x^{\mu'}$ to get a nonzero element since $[\,r_{i_\alpha}{+}1\,]\ne 0$, which is of a lower energy grade by $1$ only on the $m$-th component. So, the result is true, by Lemma 11.

If $j=m$ and $i_\alpha=m{-}1$ (so $i_\beta=m$) and $r_{m-1}=\ell{-}1$, then (3.8) implies
$u+1\le v$, so $f_m.\,(x^{(\gamma_\alpha)}\otimes x^{\mu'})=x^{(\gamma_\alpha-\varepsilon_m)}\otimes x^{\epsilon_{m+1}+\mu'}$ has the same component energy grades as $x^{(\gamma_\beta)}\otimes x^{\nu'}$, so Lemma 11 implies the result.

(2) If $j=i_\alpha\, (\le m-1)$ and $r_{i_\alpha}<\ell{-}1$, then by Proposition 4 (2.3) and the definition of $\gamma_\alpha$, we have 
\[f_j^{r_{i_\alpha}+1}.\,\bigl(x^{(\gamma_\alpha)}\otimes x^{\mu'}\bigr)=[\,r_{i_\alpha}{+}1\,]\,!\,
x^{(\gamma_\alpha-(r_{i_\alpha}+1)\varepsilon_{i_\alpha}+(r_{i_\alpha}+1)\varepsilon_{i_\alpha+1})}\ne0,
\]
which has the same component energy grades as $x^{(\gamma_\beta)}\otimes x^{\nu'}$. So Lemma 11 gives the result.

 If $j=i_\alpha< m{-}1$ and $r_{i_\alpha}=\ell{-}1$, then
 $$
 \bigl(f_jf_{j+1}\bigr)f_j^{\ell-1}.\,\bigl(x^{(\gamma_\alpha)}\otimes x^{\mu'}\bigr)=
 [\,\ell{-}1\,][\,\ell{-}1\,]\,!\,x^{(\gamma_\alpha-\ell\varepsilon_j+(\ell{-}1)\varepsilon_{j+1}+\varepsilon_{j+2})}\otimes x^{\mu'}, 
 $$
has the same component energy grades as $x^{(\gamma_\beta)}\otimes x^{\nu'}$, so the result is true, by Lemma 11. 

 If $j=i_\alpha= m{-}1$ (so $i_\beta=m$, in this case $u+1\le v$ still holds) and $r_{i_\alpha}=\ell{-}1$, then by Proposition 4 (2.4), $f_{m-1}f_mf_{m-1}^{\ell-1}$ acting on
 $x^{(\gamma_\alpha)}\otimes x^{\mu'}$ will lower down only $(m{-}1)$-st component energy grade by $1$, which implies the result as expected. 

(3) If $j<i_\alpha\,(<i_\beta)$ and $r_\alpha<\ell{-}1$, then by Proposition 4 (2.3) and the definition of $\gamma_\alpha$, we have 
\begin{equation}
f_j^{\ell-1}\,f_{j+1}^{\ell-1}\cdots f_{i_\alpha-1}^{\ell-1}f_{i_\alpha}^{r_{i_\alpha}}.\,
\bigl(x^{(\gamma_\alpha)}\otimes x^{\mu'}\bigr)=([\,\ell{-}1\,]\,!)^{i_\alpha-j}[\,r_{i_\alpha}\,]\,!\,
x^{(\gamma_\alpha')}\otimes x^{\mu'}\ne0,
\end{equation}
where $\gamma_\alpha'=\bigl(\cdots,m_j\ell, m_{j+1}\ell+(\ell{-}1),\cdots,m_{i_\alpha}\ell+(\ell{-}1),
m_{i_\alpha+1}\ell+r_{i_\alpha},\cdots\bigr)$. By the same reason, we have
\[
f_j\,f_{j+1}\cdots f_{i_\alpha-1}f_{i_\alpha}.\,
\bigl(x^{(\gamma_\alpha')}\otimes x^{\mu'}\bigr)=([\,\ell{-}1\,])^{i_\alpha-j}[\,r_{i_\alpha}{+}1\,]\,
x^{(\gamma_\alpha'')}\otimes x^{\mu'}\ne0,
\]
where $\gamma_\alpha''=\bigl(\cdots,(m_j{-}1)\ell+(\ell{-}1), m_{j+1}\ell+(\ell{-}1),\cdots,m_{i_\alpha}\ell+(\ell{-}1),
m_{i_\alpha+1}\ell+(r_{i_\alpha}{+}1),\cdots\bigr)$.  Note that 
$$
\text{Edeg}_i\,\bigl(x^{(\gamma_\alpha'')}\otimes x^{\mu'}\bigr)=
\text{Edeg}_i\,\bigl(x^{(\gamma_\beta)}\otimes x^{\nu'}\bigr), \quad \text{\it for \ each }\ i.
$$
So, Lemmas 11 gives us the result as required.

 If $j<i_\alpha\,(<i_\beta)$ and $r_\alpha=\ell{-}1$, then after using the operator in (3.9), we continue to act product operator $f_j\cdots f_{i_\alpha}f_{i_\alpha+1}$ on $x^{(\gamma_\alpha')}\otimes x^{\mu'}$. This will yield the required result.

(II) If $1\le i_\beta\le i_\alpha$, then (3.7) and $r_{i_\beta}\le \ell-1$ yields %$u+\ell+(i_\alpha-i_\beta-1)(\ell-1)-v=
%-(\ell-1)+r_{i_\beta}-r_{i_\alpha}\le -r_{i_\alpha}\le 0$, namely, 
\begin{equation}
u+(r_{i_\alpha}{+}1)+(i_\alpha-i_\beta)(\ell-1)\le v.
\end{equation}

(1) If $j>i_\alpha$, then use the same $f_j$ or $e_{m-1}$ as (I) (1) to act on $x^{(\gamma_\alpha)}\otimes x^{\mu'}$, to get the result.

(2) If $j=i_\alpha<m$ and $r_j<\ell-1$, then use the same $f_j^{r_j+1}$ as (I) (2) to act on $x^{(\gamma_\alpha)}\otimes x^{\mu'}$, to arrive at the result.

If $j=i_\alpha<m$ and $r_j=\ell-1$, then by Proposition 4 (2.3), we obtain the following nonzero element
$$
f_j^{\ell-1}.\,\bigl(x^{(\gamma_\alpha)}\otimes x^{\mu'}\bigr)=[\,\ell{-}1\,]\,!\,x^{(\gamma_\alpha''')}\otimes x^{\mu'},
$$
where $\gamma_\alpha'''=(\cdots, m_j\ell, m_{j+1}\ell+(\ell{-}1),\cdots)$, which has the same $i$-th component energy grade for any $i$ as that of $x^{(\gamma_\alpha)}\otimes x^{\mu'}$. Furthermore, by Proposition 4 (2.3) or (2.4), we have
\[
f_jf_{j+1}.\,\bigl(x^{(\gamma_\alpha''')}\otimes x^{\mu'}\bigr)=x^{(\gamma_\alpha'''')}\otimes x^{\delta_{j+1,m}\epsilon_{m+1}+\mu'},
\]
where $\gamma_\alpha''''=\bigl(\cdots, (m_j{-}1)\ell+(\ell{-}1), m_{j+1}\ell+(\ell{-}1), (1-\delta_{j+1,m})(m_{j+2}\ell+1),\cdots\bigr)$. 

Note that both $x^{(\gamma_\alpha'''')}\otimes x^{\delta_{j+1,m}\epsilon_{m+1}+\mu'}$ and $x^{(\gamma_\beta)}\otimes x^{\nu'}$ have the same component energy grades. So, the result holds, according to Lemma 11.

If $j=i_\alpha=m$, then by Proposition 4 (2.4) \& (2.5), and (3.10), we get the following element
\begin{equation}
f_m\bigl(f_{m+1}f_m\bigr)\cdots\bigl(f_{m+r_m-1}\cdots f_{m+1}f_m\bigr)\bigl(f_{m+r_m}\cdots f_{m+1}f_m\bigr).\,\bigl(x^{(\gamma_\alpha)}\otimes x^{\mu'}\bigr)=x^{(\gamma_\alpha^*)}\otimes
x^{\omega+\mu'},
\end{equation}
which has the same $i$-th component energy grade as that of $x^{(\gamma_\beta)}\otimes x^{\nu'}$ for all $i$, where $\omega=(\underset{r_m+1}{\underbrace{1,\cdots,1}},0,\cdots,0)$, and 
$\gamma_\alpha^*=\bigl(\,m_1\ell+(\ell{-}1),\cdots, m_{m-1}\ell+(\ell{-}1), (m_m{-}1)\ell+(\ell{-}1)\,\bigr)$. Thus, Lemma 11 yields the desired result. 

(3) If $j<i_\alpha$, then we consider the following cases.

If $i_\alpha<m$ and $r_{i_\alpha}<\ell{-}1$, then by Proposition 4 (2.3), \& (2.4), we have
\begin{equation}
f_j^{\ell-1}\cdots f_{i_\alpha-1}^{\ell-1}f_{i_\alpha}^{r_{i_\alpha}}.\,\bigl(x^{(\gamma_\alpha)}\otimes x^{\mu'}\bigr)=
\bigl([\,\ell{-}1\,]\,!\,\bigr)^{i_\alpha-j}[\,r_{i_\alpha}\,]\,!\,x^{(\gamma_\alpha^{*'})}\otimes x^{\mu'},
\end{equation}
where $\gamma_\alpha^{*'}=\bigl(\cdots, m_j\ell, m_{j+1}\ell+(\ell{-}1),\cdots, m_{i_\alpha}\ell+(\ell{-}1), m_{i_\alpha+1}\ell+r_{i_\alpha},\cdots \bigr)$. Furthermore, we have
\[
f_j\cdots f_{i_\alpha-1} f_{i_\alpha}.\,\bigl(x^{(\gamma_\alpha^{*'})}\otimes x^{\mu'}\bigr)=\bigl([\,\ell{-}1\,]\bigr)^{i_\alpha-j}
[\,r_{i_\alpha}{+}1\,]\,x^{(\gamma_\alpha^{*''})}\otimes x^{\mu'}\ne0,
\]
where $\gamma_\alpha^{*''}=\bigl( \cdots, (m_j{-}1)\ell+(\ell{-}1), \cdots, m_{i_\alpha}\ell+(\ell{-}1), m_{i_\alpha+1}\ell+(r_{i_\alpha}{+}1),\cdots\bigr)$. Both
$x^{(\gamma_\alpha^{*''})}\otimes x^{\mu'}$ and $x^{(\gamma_\beta)}\otimes x^{\nu'}$ have the same component energy grades, then we get the result.

If $i_\alpha<m$ and $r_{i_\alpha}=\ell{-}1$, then based on (3.12) and by Proposition 4 (2.3), \& (2.4), we have
\[
f_j\cdots f_{i_\alpha}f_{i_\alpha+1}.\,\bigl(x^{(\gamma_\alpha^{*'})}\otimes x^{\mu'}\bigr)=\bigl([\,\ell{-}1\,]\bigr)^{i_\alpha+1-j}
\,x^{(\gamma_\alpha^{*'''})}\otimes x^{\delta_{i_\alpha+1,m}\epsilon_{m+1}+\mu'}\ne0,
\]
where $\gamma_\alpha^{*'''}=\bigl(\cdots, (m_j{-}1)\ell{+}(\ell{-}1), \cdots,  m_{i_\alpha+1}\ell{+}(\ell{-}1), (1-\delta_{i_\alpha+1,m})(m_{i_\alpha+2}\ell+1),\cdots\bigr)$. The fact that both
$x^{(\gamma_\alpha^{*'''})}\otimes x^{\delta_{i_\alpha+1,m}\epsilon_{m+1}+\mu'}$ and $x^{(\gamma_\beta)}\otimes x^{\nu'}$ have the same component energy grades gives the result.

If $i_\alpha=m$, then similar to (3.11) and by Proposition 4 (2.4) \& (2.5), we first act product operator $f_m(f_{m+1}f_m)\cdots (f_{m+r_m-1}\cdots f_{m+1}f_m)$ on $x^{(\gamma_\alpha)}\otimes x^{\mu'}$ to get a nonzero element $x^{(\gamma_\alpha^{**})}\otimes x^{\omega-\epsilon_{r_m+1}+\mu'}$, 
where $\gamma_\alpha^{**}=\bigl(\cdots, m_j\ell{+}(\ell{-}1),\cdots, m_{m-1}{+}(\ell{-}1),0\bigr)$.
Next let us act the following product operator on it. Then we get 
\begin{equation}
f_j^{\ell-1}\cdots f_{m-1}^{\ell-1}.\,\bigl(x^{(\gamma_\alpha^{**})}\otimes x^{\omega-\epsilon_{r_m+1}+\mu'}\bigr)=\bigl([\,\ell{-}1\,]\,!\,\bigr)^{m-j}x^{(\gamma_\alpha^{**'})}\otimes x^{\omega-\epsilon_{r_m+1}+\mu'},
\end{equation}
where $\gamma_\alpha^{**'}=\bigl(\cdots, m_j\ell,m_{j+1}\ell{+}(\ell{-}1),\cdots, m_m\ell{+}(\ell{-}1)\,\bigr)$.

Continually, acting product operator $f_mf_{m+1}\cdots f_{m+r_m}$ on (3.13), we get a nonzero element $x^{(\gamma_\alpha^{**''})}\otimes x^{\omega+\mu'}$, where 
$\gamma_\alpha^{**''}=\bigl(\cdots, m_j\ell,m_{j+1}\ell{+}(\ell{-}1),\cdots, m_m\ell{+}(\ell{-}2)\,\bigr)$. 

Finally, acting product operator $f_jf_{j+1}\cdots f_{m-1}$ on $x^{(\gamma_\alpha^{**''})}\otimes x^{\omega+\mu'}$, we get 
$$
f_jf_{j+1}\cdots f_{m-1}.\,\bigl(x^{(\gamma_\alpha^{**''})}\otimes x^{\omega+\mu'}\bigr)=
\bigl([\,\ell{-}1\,]\bigr)^{m-j}x^{(\gamma_\alpha^{**'''})}\otimes x^{\omega+\mu'},
$$
where $\gamma_\alpha^{**'''}=\bigl(\cdots, m_{j-1}\ell{+}(\ell{-}1), (m_j{-}1)\ell{+}(\ell{-}1),\cdots, m_m\ell{+}(\ell{-}1)\,\bigr)$, which has the same component energy grades as $\gamma_\beta$. So the same reason leads to our result.

This completes the proof.
\end{proof}

\subsection{Equivalence and ordering on $\langle m,n\rangle$-tuples}
Note that the set of $\langle m,n\rangle$-tuples of nonnegative integers indexes a basis of $\Omega_q(m|n)$ via the mapping $\chi:\mathbb{Z}_+^m\times \mathds{Z}_2^n\longrightarrow \Omega_q(m|n)$ such that $\chi(\langle \alpha,\mu\rangle)=x^{\langle \alpha,\mu\rangle}:=x^{(\alpha)}\otimes x^{\mu}$ with $\alpha\in \mathbb{Z}_+^m$ and $\mu\in \mathds{Z}_2^n$. Set 
\begin{gather}
\mathbb{Z}_+^m\times\mathds{Z}_2^n(s):=\{\,\langle \alpha,\mu\rangle\in \mathbb{Z}_+^m\times \mathds{Z}_2^n\mid |\langle\, \alpha,\mu\,\rangle|=s\,\},\\
\mathbb{Z}_+^m\times \mathds{Z}_2^n(s,\textbf{r}):=\{\,\langle \alpha,\mu\rangle\in \mathbb{Z}_+^m\times \mathds{Z}_2^n(s)\mid \alpha\leq \textbf{r}\,\}.
\end{gather}
These $\langle m, n\rangle$-tuples index bases of $\Omega_q^{(s)}(m|n)$ and $\Omega_q^{(s)}(m|n,\textbf{r})$, respectively.

Set $\mathcal{E}_i\langle \alpha,\mu\rangle:=\lfloor\frac{\alpha_i}{\ell}\rfloor$ and $\mathcal{E}\langle \alpha,\mu\rangle=(\mathcal{E}_1\langle \alpha,\mu\rangle,\cdots,\mathcal{E}_m\langle \alpha,\mu\rangle)$. Define an equivalence $\thicksim$ on $\mathbb{Z}_+^m\times \mathds{Z}_2^n(s)$ or $\mathbb{Z}_+^m\times \mathbb{Z}_2^n(s,\textbf{r})$, and introduce an ordering  $\succeq$ on $\mathbb{Z}_+^m\times \mathds{Z}_2^n(s)$ or $\mathbb{Z}_+^m\times \mathbb{Z}_2^n(s,\textbf{r})$ as follows: 

\begin{defi} For any $\langle \alpha,\mu\rangle$, $\langle \beta,\nu\rangle\in \mathbb{Z}_+^m\times \mathds{Z}_2^n(s)$ or $\mathbb{Z}_+^m\times \mathds{Z}_2^n(s,\textbf{r})$, i.e., $|\, \langle \alpha, \mu\rangle\,|=s=|\,\langle \beta,\nu\rangle\,|$,
define an equivalence $\thicksim\,:$  $\langle \alpha, \mu\rangle\thicksim\langle \beta,\nu\rangle \Longleftrightarrow \mathcal{E}\langle \alpha,\mu\rangle=\mathcal{E}\langle \beta,\nu\rangle$. 

For any $\langle \alpha,\mu\rangle$, $\langle \beta,\nu\rangle\in \mathbb{Z}_+^m\times \mathds{Z}_2^n(s)$ or $\mathbb{Z}_+^m\times \mathds{Z}_2^n(s,\textbf{r})$, 
i.e., $|\, \langle \alpha, \mu\rangle\,|=s=|\,\langle \beta,\nu\rangle\,|$,
introduce a partial order $\succeq$ on $\mathbb{Z}_+^m\times \mathds{Z}_2^n(s)$ or  $\mathbb{Z}_+^m\times \mathds{Z}_2^n(s,\textbf{r})$, and define\,$:$  $\langle \alpha, \mu\rangle\succeq\langle \beta,\nu\rangle  \Longleftrightarrow\mathcal{E}\langle \alpha,\mu\rangle\ge \mathcal{E}\langle \beta,\nu\rangle  \Longleftrightarrow\mathcal{E}_i\langle \alpha,\mu\rangle\ge \mathcal{E}_i\langle \beta,\nu\rangle$, for all $i\in I_0$. 
\end{defi}

\begin{theorem} $($Stratification Theorem of $\Omega^{(s)}_q\,)$
For any $\langle \alpha,\mu\rangle$, $\langle \beta,\nu\rangle\in \mathbb{Z}_+^m\times \mathds{Z}_2^n(s)$ or $\mathbb{Z}_+^m\times \mathds{Z}_2^n(s,\textbf{r})$, the $\mathcal U_q$-modules $\Omega_q^{(s)}(m|n)$ and $\Omega_q^{(s)}(m|n, \mathbf r)$ have the following stratification structure.

\text{\rm (i)} If $\langle \alpha,\mu\rangle\thicksim\langle \beta,\nu\rangle$, then $\,\mathcal U_q.\, x^{\langle \alpha,\mu\rangle}=\mathcal U_q.\, x^{\langle \beta,\nu \rangle}$. 

\text{\rm (ii)} If $\langle \alpha, \mu\rangle\succeq\langle \beta,\nu\rangle$, then  $\,\mathcal  U_q.\, x^{\langle \beta,\nu \rangle}\subsetneq\mathcal U_q.\, x^{\langle \alpha,\mu\rangle}$.

\text{\rm (iii)} If $|\,\mathcal{E}\langle \alpha,\mu\rangle\,|=|\,\mathcal{E}\langle \beta,\nu\rangle\,|$, but $\langle \alpha,\mu\rangle\nsim\langle \beta,\nu\rangle$, 
then $\,\mathcal Hd\bigl(\,\mathcal U_q.\, x^{\langle \alpha,\mu\rangle}\bigr)\cap \mathcal Hd\bigl(\,\mathcal U_q.\, x^{\langle \beta,\nu\rangle}\bigr)=0$, where $\mathcal Hd(M)$ indicates the head of module $M$.  
\end{theorem}
\begin{proof} (i) is just Lemma 11. (ii) is given by Lemma 12.
(iii) follows from Lemma 10 and (ii).
\end{proof}

\section{Indecomposability of $\mathcal U_q$-module $\Omega_q^{(s)}(m|n,\textbf{r})$: its Loewy filtration and rigidity}

\subsection{Lowest energy grade $E_0(s)$ and highest one $E(s)$ of $\Omega_q^{(s)}(m|n,\textbf{r})$}
Given $0\leq s \leq N$, where $N=|\,\textbf{r}\,|+n=m\,(r\ell{-}1)+n=m\,(r{-}1)\ell{+}m\,(\ell{-}1){+}n$, denote by $E_0(s)$, $E(s)$ the lowest and highest energy grade of elements in $\Omega_q^{(s)}(m|n,\textbf{r})$, respectively. 

For $\langle\alpha,\mu\rangle\in \mathbb{Z}_+^m\times\mathds{Z}_2^n(s,\textbf{r})$, we have $s=|\,\alpha\,|+|\,\mu\,|$, and set 
$\alpha_i=\mathcal E_i\langle \alpha,\mu\rangle\ell+a_i$ $(0\le a_i\le \ell{-}1)$, 
\begin{gather}
\gamma\langle\alpha,\mu\rangle:=\langle \alpha, \mu\rangle-\ell\mathcal{E}\langle\alpha,\mu\rangle=\langle a_1,\cdots,a_m; \mu_1,\cdots,\mu_n\rangle=\gamma_0\langle\alpha,\mu\rangle+\gamma_1\langle\alpha,\mu\rangle,\\
\gamma_0\langle\alpha,\mu\rangle=\langle a_1,\cdots,a_m; 0,\cdots,0\rangle,\qquad
\gamma_1\langle\alpha,\mu\rangle=\langle 0,\cdots,0; \mu_1,\cdots,\mu_n\rangle.
\end{gather} 
Clearly, $\gamma_0\langle\alpha,\mu\rangle\leq \textbf{1}$.

By definition (of $E_0(s)$, $E(s)$\,), there exist $\langle\alpha,\mu\rangle,\langle\beta,\nu\rangle\in \mathbb{Z}_+^m\times\mathds{Z}_2^n(s,\textbf{r})$ such that $E_0(s)=|\,\mathcal{E}\langle\alpha,\mu\rangle\,|$ and $E(s)=|\,\mathcal{E}\langle\beta,\nu\rangle\,|$, as well as $s=|\,\langle\alpha,\mu\rangle\,|=\ell\cdot E_0(s)+|\,\gamma\langle\alpha,\mu\rangle\,|$ with $|\,\gamma\langle\alpha,\mu\rangle\,|=\sum_{i=1}^{m}a_i+\sum_{j=1}^n\mu_j$ as large as possible (so in this case  $|\,\mu\,|$ is maximal), and $s=|\,\langle\beta,\nu\rangle\,|=\ell\cdot E(s)+|\,\gamma\langle\beta,\nu\rangle\,|$, with $0\le |\,\gamma\langle\beta,\nu\rangle\,|\le \ell{-}1$ as small as possible (so in this case, $|\,\nu\,|=0$).

The following Lemma describes the energy grade of any element in $\Omega_q^{(s)}(m|n,\textbf{r})$ by considering the different intervals $[\,0, \,N\,]$ where $s$ may fall.
\begin{lemma}
Suppose $m\geq 2$, and $\textbf{char}(q)=\ell\geq 3$. For all $s\in\bigl[\,0, N\,\bigr]$, with $N=m\,(r\ell{-}1)+n$.

\text{\rm (1)} For \,$s\in \bigl[\,0, \,\ell{-}1\,\bigr]:$ \,$E_0(s)=E(s)=0$.

\text{\rm (2)} For \,$s\in\bigl[\,\ell, \,m\,(\ell{-}1)+n\,\bigr]=\cup_{j=1}^{m-1}\cup_{n'=0}^n\bigl[\,j\,(\ell{-}1){+}n'{+}1,\, (j{+}1)(\ell{-}1){+}n'\,\bigr]:$ 
\begin{equation*}
E_0(s)=0, \quad 1\leq E(s)\leq n_1+m-m_1-\sum_{i<0}\delta_{i,n_2-m_2}=
E\bigl(\,m\,(\ell{-}1)+n\,\bigr),
\end{equation*}
where $m=m_1\ell+m_2 \ \,(\,0\leq m_2\leq \ell{-}1\,)$ and $n=n_1\ell+n_2\ \,(\,0\leq n_2\leq \ell{-}1\,)$. More precisely,

\text{\rm (i)} For \,$s=j\,(\ell{-}1){+}h{+}n'$ \ $(\,0\leq n^{\prime}\leq n$, $1\leq j\leq m{-}1$, $1\leq h\leq \ell{-}1\,)$,
$$
E(s)=(j-j_1+n_1^{\prime})-\sum_{i=-(\ell{-}1)}^{-1}\delta_{i,n_2^{\prime}+h-j_2}
+\sum_{i=\ell}^{2l-2}\delta_{i,n_2^{\prime}+h-j_2},
$$ 
where $n^{\prime}=n_1^{\prime}\ell+n_2^{\prime}$, $j=j_1\ell+j_2\, \ (\,0\leq n_2',\ j_2\leq \ell{-}1\,)$.

\text{\rm (ii)} For \,$s\in\bigl[\,j\,(\ell{-}1){+}n'{+}1, \,(j{+}1)(\ell{-}1){+}n'\,\bigr]$ with $1\leq j\leq m{-}1$ and $0\leq n^{\prime}\leq n$, 
$$
j-j_1+n_1^{\prime}-1\leq E(s)\leq j-j_1+n_1^{\prime}+1.
$$

\text{\rm (3)}  For \,$s\in\bigl[\,m\,(\ell{-}1){+}n{+}1, \,N{-}\ell\,\bigr]=\cup_{k=1}^{m(r{-}1){-}1}\cup_{h=0}^{\ell{-}1}\bigl[\,k\ell+s(m|n), \,k\ell{+}s(m|n){+}\ell{-}1\,\bigr]:$  
\begin{gather*}
E_0(s)=k, \quad k+1\leq E(s)\leq m\,(r-1), \quad \text{\it for } \ s=k\ell+(m{-}1)(\ell{-}1)+h+n, 
\end{gather*} 
where $s(m|n)=(m{-}1)(\ell{-}1)+n$, $0\leq h\leq \ell{-}1$, $1\leq k\leq m\,(r{-}1)-1$. Namely, 

\text{\rm (i)} for \,$k\leq m\,(r{-}1)-E\bigl(s(m{+}1|n)\bigr)$, 
$$
k+E\bigl(s(m|n)\bigr)\leq E(s)\leq k+E\bigl(s(m{+}1|n)\bigr).
$$

\text{\rm (ii)}  for \,$k>m\,(r{-}1)-E\bigl(s(m{+}1|n)\bigr)$, 
$$E(s)=m\,(r{-}1),$$
where $E\bigl(s(m|n)\bigr)\geq1$ under the assumption $m\geq 2$.

\text{\rm (4)} For \,$s\in\bigl[\,N-\ell+1, N\,\bigr]:$ \ $E_0(s)=E(s)=m\,(r{-}1)$.
\end{lemma}
\begin{proof}
\text{\rm (1)} is obvious. 

(4) In this case, note that for any $x^{\langle\alpha,\mu\rangle} \in \Omega_q^{(s)}(m|n,\textbf{r})$, $\langle\alpha,\mu\rangle$ is of the form $\langle\, (r{-}1)\ell+a_1,\cdots,(r{-}1)\ell+a_m;\,\mu_1,\cdots,\mu_n\,\rangle$ with $\gamma_0\langle\alpha,\mu\rangle\leq \textbf{1}$, such that $|\gamma\langle\alpha,\mu\rangle|=(m{-}1)(\ell{-}1)+h+n$ with $0\leq h\leq \ell{-}1$, and $\mathcal{E}\langle\alpha,\mu\rangle=(\underset{m}{\underbrace{r-1,\cdots,r-1}})$. So, $E_0(s)=E(s)=m\,(r{-}1)$.

\text{\rm (2)} When $\ell\leq s\leq m\,(\ell{-}1)+n$: it is clear that $E(s)_0=0$, as even for the extreme case $s=m\,(\ell{-}1)+n$, taking $\langle\alpha,\mu\rangle=\langle\, \underset{m}{\underbrace{\ell{-}1,\cdots, \ell{-}1}};\,\underset{n}{\underbrace{1,\cdots,1}}\,\rangle$, we get $\gamma\langle\alpha,\mu\rangle=\langle\alpha,\mu\rangle$ and $\mathcal{E}\langle\alpha,\mu\rangle=\textbf{0}$, that is, $E_0(s)=0$.

In order to estimate $E(s)$, we can assume that $j\,(\ell{-}1)+n'+1\leq s\leq (j+1)(\ell{-}1)+n'$, for $1\leq j\leq m-1$, $0\leq n^{\prime}\leq n$. Let us consider the general case: $s=j\,(\ell{-}1)+h+n'$, with $1\leq h\leq \ell{-}1$, $1\leq j\leq m{-}1$ and $0\leq n^{\prime}\leq n$. Write $j=j_1\ell+j_2$, $n^{\prime}=n_1^{\prime}\ell+n_2^{\prime}$ with $0\leq j_2, \, n_2^{\prime}\leq \ell{-}1$. Then rewrite $s=j\,(\ell{-}1)+h+n'=(j-j_1+n_1^{\prime})\ell+(n_2^{\prime}+h-j_2)$. 

Clearly, when $-(\ell{-}1)\leq n_2^{\prime}+h-j_2\leq -1$, $E(s)=j-j_1+n_1^{\prime}-1$; and when $\ell\leq n_2^{\prime}+h-j_2\leq 2\ell-2$, $E(s)=j-j_1+n_1^{\prime}+1$; otherwise, $E(s)=j-j_1+n_1^{\prime}$. That is,
$$
E(s)=(j-j_1+n_1^{\prime})-\sum\limits_{i=-(\ell{-}1)}^{-1}\delta_{i,n_2^{\prime}+h-j_2}
+\sum\limits_{i=\ell}^{2\ell-2}\delta_{i,n_2^{\prime}+h-j_2}.
$$
So, we obtain $j-j_1+n_1^{\prime}-1\leq E(s)\leq j-j_1+n_1^{\prime}+1$, for $j\,(\ell{-}1)+n'+1\leq s\leq (j+1)(\ell{-}1)+n'$ with $1\leq j\leq m-1$ and $0\leq n^{\prime}\leq n$. 

For the extreme case $s=m\,(\ell{-}1)+n$, we get $s=m\,(\ell{-}1)+n=(n_1+m-m_1)\ell+(n_2-m_2)$ with $m=m_1\ell+m_2$, $n=n_1\ell+n_2$, where $0\leq m_2, \, n_2\leq \ell{-}1$. In this case, 
$$
E(s)=n_1+m-m_1-\sum_{i<0}\delta_{i,n_2-m_2}.
$$

\text{\rm (3)} When $m\,(\ell{-}1)+n+1\leq s \leq N-\ell$: 

Firstly, we rewrite $N-\ell=m\,(r\ell{-}1)-\ell+n=(m\,(r{-}1)-1)\ell+m\,(\ell{-}1)+n$. So, now for the given $s$ as above, we can put it into a certain strictly smaller interval: 
$$
k\,\ell+s(m|n)\leq s\leq k\,\ell+s(m{+}1|n),
$$ 
for some $k$ with $1\leq k\leq m\,(r{-}1)-1$. Namely, $s=k\,\ell+s(m|n)+h$ with $0\leq h\leq \ell{-}1$.

Secondly, write $k=k_1m+k_2$ with $0\leq k_2\leq m-1$. Taking
$$
\langle\alpha,\mu\rangle=\langle\,(\underset{k_2}{\underbrace{k_1+1,\cdots,k_1+1}}, k_1,\cdots,k_1)\ell;\,\underset{n}{\underbrace{0,\cdots,0}}\,\rangle
+\langle\, h,\underset{m-1}{\underbrace{\ell{-}1,\cdots,\ell{-}1}};\,1,\cdots,1\,\rangle,
$$
we obtain $|\,\langle\alpha,\mu\rangle\,|=s$, that is, $\langle\alpha,\mu\rangle\in \mathbb{Z}_+^m\times \mathds{Z}_2^n(s,\textbf{r})$, $\mathcal{E}\langle\alpha,\mu\rangle=(\underset{k_2}{\underbrace{k_1{+}1,\cdots,k_1{+}1}},k_1,\cdots,k_1)$, $\gamma\langle\alpha,\mu\rangle=\langle\, h,\underset{m-1}{\underbrace{\ell{-}1,\cdots,\ell{-}1}};1,\cdots,1\,\rangle$ with $|\,\gamma\langle\alpha,\mu\rangle\,|=s(m|n)+h$ large enough. So, $E_0(s)=|\,\mathcal{E}\langle\alpha,\mu\rangle\,|=k$.

Finally, as for the estimate of $E(s)$, for $k\ell+(m{-}1)(\ell{-}1)+n\leq s\leq k\ell+m\,(\ell{-}1)+n$, in view of \text{\rm (2)}, $k+E\bigl(s(m|n)\bigr)\leq E(s)\leq k+E\bigl(s(m{+}1|n)\bigr)$. 

Still write $m=m_1\ell+m_2$, $n=n_1\ell+n_2$, with $0\leq m_2,\ n_2\leq \ell{-}1$. Then we have 
\begin{gather*}
(m{-}1)(\ell{-}1)+n=(m+n_1-m_1-1)\ell+(n_2-m_2+1),\\
E\bigl(s(m|n)\bigr)=m+n_1-m_1-1+\delta_{\ell{-}1,n_2-m_2}-\sum_{i=-(\ell{-}1)}^{-1}\delta_{i,n_2-m_2}.
\end{gather*}
Similar discussion for $m\,(\ell{-}1)+n$ yields $E\bigl(s(m{+}1|n)\bigr)=m+n_1-m_1-\sum_{i=-(\ell{-}1)}^{-1}\delta_{i,n_2-m_2}$. So for the above $s$, we get
$$
k{+}m{+}n_1{-}m_1{-}1{+}\delta_{\ell{-}1,n_2-m_2}{-}\sum\limits_{i=-(\ell{-}1)}^{-1}\delta_{i,n_2-m_2}\leq E(s)\leq k{+}m{+}n_1{-}m_1{-}\sum\limits_{i=-(\ell{-}1)}^{-1}\delta_{i,n_2-m_2},
$$
only if $k+m+n_1-m_1-\sum_{i=-(\ell{-}1)}^{-1}\delta_{i,n_2-m_2}\leq m\,(r{-}1)$. Otherwise, $E(s)=m\,(r{-}1)$.

This completes the proof. 
\end{proof}

\subsection{Socle of $\Omega_q^{(s)}(m|n,\textbf{r})$ determined by the lowest energy grade}
The notions of highest and lowest energy grade will play useful roles in the $\mathcal U_q$-module structure of $\Omega_q^{(s)}(m|n,\textbf{r})$. Those elements of highest energy grade $E(s)$ compose the indispensable generators of $\Omega_q^{(s)}(m|n,\textbf{r})$, meanwhile, those of lowest energy grade $E_0(s)$ constitute exactly the socle generators. 

\begin{theorem}
For the $\mathcal U_q$-modules $\Omega_q^{(s)}(m|n,\textbf{r})$, with $0\leq s\leq N=m\,(r\ell{-}1)+n$, we have

\text{\rm (1)} A submodule $\mathfrak{V}_y=\mathcal U_q.\, y$ is simple if and only if $\text{Edeg}\,(y)=E_0(s)$, where $s=|\,y\,|$.

\text{\rm (2)} \,$\emph{\text{Soc}}\ \Omega_q^{(s)}(m|n,\textbf{r})=\emph{\text{Span}}_{\,\mathbb{K}}\Bigl\{\,x^{\langle\alpha, \mu\rangle} \in \Omega_q^{(s)}(m|n,\textbf{r})\ \Bigl|\ |\,\mathcal{E}\langle\alpha,\mu\rangle\,|=E_0(s)\,\Bigr\}$.

\text{\rm (3)} \,$\Omega_q^{(s)}(m|n,\textbf{r})=\sum_{\langle\beta,\nu\rangle\in \mathbb{Z}^m\times \mathds{Z}_2^n, \,|\langle\beta,\nu\rangle|=E(s)}\mathfrak{V}_{\langle\beta,\nu\rangle}$, where $\,\mathfrak{V}_{\langle\beta,\nu\rangle}=\mathcal U_q.\, x^{\langle \beta,\nu\rangle}$.
\end{theorem}
\begin{proof}
\text{\rm (1)} If $\text{Edeg}\,(y)>E_0(s)$, then by Definition 8, in the expression of $y=\sum_{\langle\alpha,\mu\rangle}k_{\langle\alpha,\mu\rangle}
x^{\langle\alpha,\mu\rangle}$, there exists some $\langle\beta,\nu\rangle\in \mathbb{Z}_+^m\times \mathds{Z}_2^n(s,\textbf{r})$, $k_{\langle\beta,\nu\rangle}\neq 0$ such that $|\,\mathcal{E}\langle\beta,\nu\rangle\,|=\text{Edeg}(y)$. 
By Lemmas 12 (see its proof) \& 9 or Theorem 14, we can find $z\in\mathcal U_q$ such that $z.\, x^{\langle\beta,\nu\rangle}\neq 0$ (then $z.\, y\neq 0$), but $\text{ Edeg}\,(z.\, y)=\text{Edeg}\,(z.\, x^{\langle\beta,\nu\rangle})<\text{Edeg}\,(y)$, so we get a proper submodule $(0\neq)$ $\mathfrak{V}_{u.\, y}\subsetneq \mathfrak{V}_y$, by Lemma 12 or Theorem 14. It is a contradiction. So the assertion is true.

\text{\rm (2)} follows from conclusion \text{\rm (1)}, together with Lemmas 10-12. Since for those $\langle\alpha,\mu\rangle,\ \langle\beta,\nu\rangle$ $\in \mathbb{Z}_+^m\times \mathds{Z}_2^n(s,\textbf{r})$ with $|\,\mathcal{E}\langle\alpha,\mu\rangle\,|=|\,\mathcal{E}\langle\beta,\nu\rangle\,|=E_0(s)$, if $\langle\alpha,\mu\rangle\thicksim \langle\beta,\nu\rangle $, then $\mathfrak{V}_{\langle\alpha,\mu\rangle}=\mathfrak{V}_{\langle\beta,\nu\rangle}$ by Lemma 11; and if $\langle\alpha,\mu\rangle\nsim \langle\beta,\nu\rangle $, then by Theorem 14, $\mathfrak{V}_{\langle\alpha,\mu\rangle}\cap \mathfrak{V}_{\langle\beta,\nu\rangle}=0$.

\text{\rm (3)} For any monomial element $x^{\langle\alpha,\mu\rangle}\in \Omega_q^{(s)}(m|n,\textbf{r})$, namely, for any $\langle\alpha,\mu\rangle \in \mathbb{Z}_+^m\times \mathds{Z}_2^n(s,\textbf{r})$ with $|\,\langle\alpha,\mu\rangle\,|=s$. According to the partial order $\succeq$ defined in Subsection 3.3, we assert that there exists a $\langle\beta,\nu\rangle=\langle\beta,\mathbf 0\rangle \in \mathbb{Z}_+^m\times\mathds{Z}_2^n(s,\textbf{r})$ with $|\,\mathcal{E}\langle\beta,\nu\rangle\,|=E(s)$, such that $\langle\beta,\nu\rangle\succeq \langle\alpha,\mu\rangle$. 

Actually, this fact follows from the proof of Lemma 12. Since $s=\ell\cdot |\,\mathcal{E}\langle\alpha,\mu\rangle\,|+|\,\gamma\langle\alpha,\mu\rangle\,|=\ell E(s)+\hbar$ with $0\leq \hbar\le \ell{-}1$, if $|\,\mathcal{E}\langle\alpha,\mu\rangle\,|<E(s)$. Write $|\,\gamma\langle\alpha,\mu\rangle\,|=k\,\ell+h \ (0\leq h\leq \ell{-}1)$. Then by definition (of $E(s)$), we have $h=\hbar $ and $k=E(s)-|\,\mathcal{E}\langle\alpha,\mu\rangle\,|$.  Take $\underline{\kappa}=(k_1,\cdots, k_m)\in \mathbb Z_+^m$, such that $\sum_{i=1}^m k_i=k$ and $\mathcal{E}_i\langle\alpha,\mu\rangle+k_i\leq r-1$ for each $i\in I_0$. Construct 
$$
\langle \beta, \nu\rangle:=\langle\, (\mathcal E_1\langle\alpha,\mu\rangle+k_1)\,\ell+h, \,(\mathcal E_2\langle\alpha,\mu\rangle+k_2)\,\ell,\,\cdots,\,(\mathcal E_m\langle\alpha,\mu\rangle+k_m)\,\ell\,;\,\mathbf 0\,\rangle.
$$
Then $|\,\langle \beta,\mathbf 0\rangle\,|=\ell\cdot|\mathcal E\langle\alpha,\mu\rangle\,|+k\,\ell+h=s$, i.e., $\langle \beta,\mathbf 0\rangle\in\mathbb Z_+^m\times\mathds Z_2^n(s, \mathbf r)$. Obviously,  $|\,\mathcal{E}\langle\beta,\nu\rangle\,|=E(s)$,
$\mathcal E\langle\beta,\mathbf 0\rangle=\mathcal E\langle\alpha,\mu\rangle+\underline{\kappa}\ge \mathcal E\langle\alpha,\mu\rangle$, so $\langle\beta,\mathbf 0\rangle \succeq\langle\alpha,\mu\rangle$.

Again, from Lemma 12, together with its proof, there exists $z\in \mathcal U_q$ such that $z.\, x^{\langle\beta,\mathbf 0\rangle}=x^{\langle\alpha,\mu\rangle}$, that is, $x^{\langle\alpha,\mu\rangle}\in \mathfrak V_{\langle\beta,\mathbf 0\rangle}$. Hence, we arrive at the result as stated.
\end{proof}

\subsection{Indecomposability of $\mathcal U_q$-modules $\Omega_q^{(s)}(m|n,\textbf{r})$}
In this subsection, we will use Lemma 15 to prove that the $\mathcal U_q$-modules $\Omega_q^{(s)}(m|n,\textbf{r})$ for $s$ lying in different intervals are indecomposable.

We shall have to consider in the following theorem: the set $\wp_s(k)$ of equivalence classes of $\langle m,n\rangle$-tuples $\langle\alpha,\mu\rangle\in \mathbb{Z}_+^m\times \mathds{Z}_2^n(s,\textbf{r})$ satisfying  $s=|\,\langle\alpha,\mu\rangle \,|$ and $|\,\mathcal{E}\langle\alpha,\mu\rangle\,|=k$ with respect to the equivalence relation $\thicksim$ defined in Subsection 3.3.  We will consider the case $k=E_0(s)$, for $s = k\ell+(m-1)(\ell-1)+h+n$, with $0 \leq h\leq \ell-1$,  and will then denote it by $\wp_s(E_0(s))$. 

It is obvious that representatives $\eta\in\wp_s(E_0(s))$ can be constructed as follows: 

For a given
$k=E_0(s)$, let $\underline{\kappa}:=(k_1,\cdots,k_m)$, $(0\le k_i\le r-1)$ such that $\sum k_i=k$; for $s_h=h+n$ with $0 \leq h \leq \ell-1$, put
$\gamma_{s_h}=\langle\,\ell-1,\cdots, \ell-1, s_h\,;\,\mathbf 0\,\rangle$  if $s_h\le \ell-1$ (so 
$|\,\gamma_{s_h}\,|=(m-1)(\ell-1)+s_h$; otherwise, if $\ell\le s_h=(\ell-1)+\hslash$ $(1\le\hslash\le n)$, put $\gamma_{s_h}=\langle\,\ell-1,\cdots, \ell-1\,; \,\omega(\hslash)\,\rangle$. 

Now put $\eta(\underline{\kappa}):=\langle\,\ell\cdot\underline{\kappa}\,;\,\mathbf 0\,\rangle+\gamma_{s_h}$, then 
$\mathcal E\langle\eta(\underline{\kappa})\rangle=\underline{\kappa}$, $\gamma(\eta(\underline{\kappa}))=\gamma_{s_h}$, as well as $|\,\eta(\underline{\kappa})\,|=k\,\ell+(m-1)(\ell-1)+s_h=s$, i.e., $\eta(\underline{\kappa})\in\mathbb Z_+^m\times\mathds Z_2^n(s, \,\mathbf r)$.  
So, we will write, by abuse (choosing specific representatives) 
\begin{equation}
\wp_s(E_0(s))=\{\,\eta(\underline{\kappa})\in\mathbb Z_+^m\times\mathds Z_2^n(s, \,\mathbf r)\mid \underline{\kappa}=(k_1,\cdots,k_m)\in \mathbb{Z}_+^m, \ |\,\underline{\kappa}\,|=E_0(s), \ k_i\leq r{-}1\,\}.
\end{equation}

\begin{theorem}
For the $\mathcal U_q$-modules $\Omega_q^{(s)}(m|n,\textbf{r})$, with $0\leq s\leq N=m\,(r{-}1)\ell+m\,(\ell{-}1)+n$, then

\text{\rm (A)} \,When \,$s\in \bigl[\,0, \,\ell{-}1\,\bigr]\cup \bigl[\,N{-}(\ell{-}1),\,N\,\bigr]\,:$ \ $\Omega_q^{(s)}(m|n,\textbf{r})=\mathcal U_q.\,x^{\langle\eta\rangle}=\mathfrak{V}_\eta$ is simple. Moreover, $x^{\langle\,\eta\,\rangle}$ is the highest weight vector with $|\,\eta\,|=s$, 
where  $\tau=(r-1,\cdots,r-1)$, 
\begin{equation}
\eta=\begin{cases}\langle \,s,0,\cdots,0;\,\mathbf 0\,\rangle, & \text{\it if } \ s\in \bigl[\,0, \ell{-}1\,\bigr];\\
\langle\, \ell\cdot \tau;\,\mathbf 0\,\rangle +\gamma_{s_h}, & \text{\it if } \ s\in \bigl[\,N{-}(\ell{-}1),N\,\bigr],
\end{cases}
\end{equation}
where for $s=m\,(r{-}1)\ell{+}(m{-}1)(\ell{-}1){+}h{+}n\in \bigl[\,N{-}(\ell{-}1),N\,\bigr]$ with  $\,0\leq h\leq \ell{-}1:$  we set $s_h:=h{+}n$ and put:
\begin{equation}
\gamma_{s_h}=\begin{cases}\langle\, \ell{-}1,\cdots,\ell{-}1, s_h;\,\mathbf 0\,\rangle, & \text{\it if } \ 0\le s_h\leq \ell{-}1;\\
\langle\,\ell{-}1,\cdots,\ell{-}1,\ell{-}1; \,\omega(\hslash)\,\rangle, & \text{\it if } \ \ell\le s_h=(\ell{-}1){+}\hslash \ \,(1\le \hslash\le n),
\end{cases}
\end{equation}
where $\omega(\hslash)=\langle\,\underset{\hslash}{\underbrace{1,\cdots,1}},0,\cdots,0\,\rangle$.

\text{\rm (B)} \,When \,$s\in\bigl[\,\ell, \,m\,(\ell{-}1){+}n\,\bigr]:$ \, $\Omega_q^{(s)}(m|n,\textbf{r})$ is indecomposable, but nonsimple, for $s=\jmath\,(\ell{-}1)+h$ $(0\le h< \ell{-}1)$. In this case, the socle of the $\mathcal U_q$-module $\Omega_q^{(s)}(m|n,\textbf{r}):$
\begin{equation}
\emph{\text{Soc}}\,\Bigl(\Omega_q^{(s)}(m|n,\textbf{r})\,\Bigr)=\mathfrak{V}_\eta
\end{equation}
is simple. Moreover, $x^{\langle\,\eta\,\rangle}$ is the highest weight vector with $|\,\eta\,|=s$, 
where  $n=n'(\ell{-}1)+h'$ $(n'\ge 0,\ 0\le h'< \ell{-}1)$,
\begin{equation}
\eta=\begin{cases}\langle\,\underset{\jmath}{\underbrace{\ell{-}1,\cdots,\ell{-}1}},h,0,\cdots,0\,;\mathbf 0\,\rangle, & \text{\it if } \ 1\le \jmath\le m{-}1;\\
\langle\,\ell{-}1,\cdots,\ell{-}1;\, \omega(\hbar)\,\rangle, & \text{\it if } \ m\le \jmath\le m+n',
\end{cases}
\end{equation}
where $\hbar=s-m\,(\ell{-}1)=(\jmath{-}m)(\ell{-}1)+h$ $(0\le h< \ell{-}1)$.

\text{\rm (C)} \,When $s\in\bigl[\,m\,(\ell{-}1){+}1{+}n, \,N{-}\ell\,\bigr]:$  \,$\Omega_q^{(s)}(m|n,\textbf{r})$ is indecomposable, but nonsimple, for $s=\kappa\,\ell{+}(m{-}1)(\ell{-}1){+}h{+}n$ $(\,0\leq h\leq \ell{-}1$, $1\leq \kappa\leq m\,(r{-}1){-}1\,)$. In this case, the socle of the $\mathcal U_q$-module $\Omega_q^{(s)}(m|n,\textbf{r}):$
\begin{equation}
\emph{\text{Soc}}\,\Bigl(\Omega_q^{(s)}(m|n,\textbf{r})\,\Bigr)=\bigoplus_{\eta({\underline{\kappa}})\in \wp_s(E_0(s))}\mathfrak{V}_{\eta({\underline{\kappa}})}
\end{equation}
is semisimple rather than simple, where  $\wp_s(E_0(s))=\{\,\eta(\underline{\kappa})\in\mathbb Z_+^m\times\mathds Z_2^n(s, \,\mathbf r)\mid\mathcal E(\eta(\underline{\kappa}))= \underline{\kappa}=(k_1,\cdots,k_m)\in \mathbb{Z}_+^m, \ |\,\underline{\kappa}\,|=E_0(s), \ k_i\leq r{-}1\,\}$. Moreover, $x^{\langle\,\eta({\underline{\kappa}})\,\rangle}$'s for all $\underline{\kappa})$ as above are the respective highest weight vectors with  $|\,\eta({\underline{\kappa}})\,|=s$, and setting
   $s_h=h+n$,
\begin{equation}
\eta({\underline{\kappa}}):=\langle \,\ell\cdot {\underline{\kappa}}\,;\,\mathbf 0\,  \rangle+\gamma_{s_h}\in \mathbb{Z}_+^m\times \mathds{Z}_2^n(s,\textbf{r}),
\end{equation}
where
\begin{equation}
\gamma_{s_h}=\begin{cases}\langle\, \ell{-}1,\cdots,\ell{-}1, s_h;\,\mathbf 0\,\rangle, & \text{\it if } \ 0\le s_h\leq \ell{-}1;\\
\langle\,\ell{-}1,\cdots,\ell{-}1,\ell{-}1; \,\omega(\hslash)\,\rangle, & \text{\it if } \ \ell\le s_h=(\ell{-}1){+}\hslash \ \,(1\le \hslash\le n).
\end{cases}
\end{equation}
\end{theorem}
\begin{proof}
\text{\rm (A)} In these two extreme cases: $s\in \bigl[\,0, \,\ell{-}1\,\bigr]\cup \bigl[\,N{-}(\ell{-}1),N\,\bigr]$, by Lemma 15 (1) \& (4), we have $E_0(s)=E(s)$. Note that the generating sets of \text{\rm (A)} in both cases only contain one equivalence class with respect to the equivalent relation $\thicksim$ defined in Subsection 3.3. Thus, Theorem 16 \text{\rm (2)} and \text{\rm (3)} give us the desired result below:
$$
\Omega_q^{(s)}(m|n,\textbf{r})=\text{Soc}\ \Bigl(\Omega_q^{(s)}(m|n,\textbf{r})\Bigr)=\mathcal U_q.\,x^{\langle\eta\rangle}=\mathfrak{V}_\eta
$$
is simple, here $\eta=\langle\, s,0,\cdots,0;\,\mathbf 0\,\rangle$ for $0\leq s\leq \ell{-}1$; $\eta:=\langle \,\ell\cdot \tau;\,\mathbf 0\,\rangle +\gamma_{s_h}$ as defined above for $N-(\ell{-}1)\leq s=m\,(r{-}1)\ell+(m{-}1)(\ell{-}1)+h+n\leq N$ with $0\leq h\leq \ell{-}1$. $x^{\langle\,\eta\,\rangle}$ is the relevant highest weight vector, i.e., $e_i.\,x^{\langle\,\eta\,\rangle}=0$ for all $i\in I$ (by Propsition 4, or Theorem 5 (2)).

\text{\rm (B)} For $s\in \bigl[\,\ell,\, m\,(\ell{-}1){+}n\,\bigr]$, by Lemma 15 \text{\rm (2)}, we have $0=E_0(s)< E(s)$. Consequently, Theorem 16 \text{\rm (2)} and \text{\rm (3)} give rise to 
$$
\text{Soc}\ \Bigl(\Omega_q^{(s)}(m|n,\textbf{r})\Bigr)\subsetneq \Omega_q^{(s)}(m|n,\textbf{r}).
$$
So, $\Omega_q^{(s)}(m|n,\textbf{r})$ is indecomposable.

Clearly, by definition, those $\langle m, n\rangle$-tuples $\langle\alpha,\mu\rangle\in \mathbb{Z}_+^m\times \mathbb{Z}_2^n(s,\textbf{r})$ with $|\,\mathcal{E}\langle\alpha,\mu\rangle\,|=E_0(s)=0$ are equivalent to each other with respect to $\thicksim$. Write $s=j\,(\ell{-}1)+h$ with $0\leq h<\ell{-}1$, $j\ge 1$.

If $j\leq m{-}1$, then $\eta=\langle\, \underset{j}{\underbrace{\ell{-}1,\cdots,\ell{-}1}},h,0,\cdots,0\,;\,\mathbf 0\,\rangle\in \mathbb{Z}_+^m\times \mathds{Z}_2^n(s,\textbf{r})$ is one representative of this equivalence class. While, for 
$j\geq m$, write $\hbar=(j{-}m)(\ell{-}1)+h$, $\eta=\langle\, \underset{m}{\underbrace{\ell{-}1,\cdots,\ell{-}1}};\,\omega(\hbar)\,\rangle$ is one representative of this equivalence class, here $|\,\eta\,|=m\,(\ell{-}1)+\hbar=s$, still $\eta\in \mathbb{Z}_+^m\times \mathds{Z}_2^n(s,\textbf{r})$.

Hence, $\text{Soc}\,\Bigl(\Omega_q^{(s)}(m|n,\textbf{r})\Bigr)=\mathfrak{V}_\eta$ is simple, and $x^{\langle\,\eta\,\rangle}$ is the highest weight vector. 

\smallskip
\text{\rm (C)} For $s\in\bigl[\,m\,(\ell{-}1){+}1{+}n, \, N{-}\ell\,\bigr]$, write $s=k\,\ell+(m{-}1)(\ell{-}1)+h+n$, and $s_h=h+n$, with $0\leq h\leq \ell{-}1$, $1\leq k\leq m\,(r{-}1)-1$. Then by Lemma 15 (3), $E_0(s)=k$, and $k+1\leq E(s)\leq m\,(r{-}1)$ under the assumption $m\ge 2$. So, Theorem 16 (2) \& (3) implies that 
$$
\text{Soc}\ \Bigl(\Omega_q^{(s)}(m|n,\textbf{r})\Bigr)\subsetneq \Omega_q^{(s)}(m|n,\textbf{r}).
$$
This means that $\Omega_q^{(s)}(m|n,\textbf{r})$ is indecomposable.

Now consider the set  $\wp_s(E_0(s))$ and its representatives $\eta\in\wp_s(E_0(s))$ constructed above: 

\begin{equation}
\wp_s(E_0(s))=\{\,\eta(\underline{\kappa})\in\mathbb Z_+^m\times\mathds Z_2^n(s, \,\mathbf r)\mid \underline{\kappa}=(k_1,\cdots,k_m)\in \mathbb{Z}_+^m, \ |\,\underline{\kappa}\,|=E_0(s), \ k_i\leq r{-}1\,\}.
\end{equation}

According to Lemma 11 and Theorem 16 \text{\rm (1)}, we see that each $\mathfrak{V}_{\eta(\underline{\kappa})}$ is a simple submodule, and  $x^{\langle\,\eta(\underline{\kappa})\,\rangle}$ is the highest weight vector (i.e., $e_i.\,x^{\langle\,\eta(\underline{\kappa})\,\rangle}=0$ for all $i\in I$, by Proposition 4 and the definition of $\eta(\underline{\kappa})$\,) of $\mathfrak{V}_{\eta(\underline{\kappa})}$. As the $\langle m,n\rangle$-tuples in $\wp_s(E_0(s))$ are not equivalent with each other with respect to $\thicksim$, Theorem 14 (i) means that $\emph{\text{Soc}}\ \Bigl(\Omega_q^{(s)}(m|n,\textbf{r})\Bigr)=\bigoplus_{\eta\in \wp_s(E_0(s))}\mathfrak{V}_\eta$ is semisimple but non-simple.

This completes the proof.
\end{proof}

\begin{remark}
Write $\mathcal K(\kappa)=\{\,\underline{\kappa}=(k_1,\cdots,k_m)\in\mathbb Z_+^m\mid \sum_i k_i=\kappa, \ k_i\le r{-}1\,\}$. So, for any $\eta(\underline{\kappa})\in\wp_s(E_0(s))$, we have
$\mathcal E(\eta(\underline{\kappa}))\in \mathcal K(E_0(s))$. In particular, when $E_0(s)=0$, $\mathcal K(0)=\{\,\underline{\mathbf 0}\,\}$. %According to Lemma 11 or Theorem 14 \text{\rm (ii)}, since $\langle\alpha, \mu\rangle\thicksim\langle\alpha, \mu'\rangle$ $($for the definition of $\mu'$, see (3.4)\,$)$, $\mathcal U_q.\,x^{\langle\alpha, \mu'\rangle}=\mathcal U_q.\,x^{\langle\alpha, \mu\rangle}$. 
So, the generator sets given in $(4.6)\ \& \ (4.8)$ of the socles of $\mathcal U_q$-module $\Omega_q^{(s)}(m|n,\textbf{r})$ for different $s$ in Theorem 17 $(B)$ $\&$ $(C)$ can be uniformly written as 
\begin{equation}
\begin{split}
\wp_s(\mathcal K(E_0(s)))&=\{ \langle\,\gamma_\alpha, \mu\,\rangle:=\langle\,\ell\cdot \underline{\kappa}\,;\,\mathbf 0\,\rangle+\gamma_{s_h}\in\mathbb Z_+^m\times\mathds Z_2^n(s, \mathbf r)\mid \mathcal E\langle\,\gamma_\alpha, \mu\,\rangle=\underline{\kappa}\in \mathcal K(E_0(s))\,\},\\
\gamma_{s_h}&=\begin{cases}\langle\,\underset{\jmath}{\underbrace{\ell{-}1,\cdots,\ell{-}1}},h,0,\cdots,0\,;\mathbf 0\,\rangle, & \text{\it if } \ 1\le \jmath\le m{-}1,\ 0\le h\le \ell{-}1;\\
\langle\,\ell{-}1,\cdots,\ell{-}1,\ell{-}1; \,\omega(\hslash)\,\rangle, & \text{\it if } \ 1\le \hslash\le n.
\end{cases}
\end{split}
\end{equation}
Using the notation in (4.2), we have $\gamma_1(\langle\,\gamma_\alpha, \mu\,\rangle)=\omega(\hslash)$, here we make a convention: $\omega(\hslash)=\mathbf 0$ when $\hslash=0$. Afterwards, we also set $\gamma(\eta(\underline{\kappa}))=\gamma_{s_h}$ if instead write $\eta(\underline{\kappa})$ for $\langle\,\gamma_\alpha, \mu\,\rangle\in \wp_s(\mathcal K(E_0(s)))$.
\end{remark}

\subsection{Weaving a net from the socle of $\Omega_q^{(s)}(m|n,\textbf{r})$ with $s\in[\,\ell, N{-}\ell\,]$}
In accordance with the notation in Remark 18, we can write 
\begin{equation}
\wp_s(\mathcal K(\kappa))=\{\,\eta(\underline{\kappa})=\langle\gamma_\alpha,\mu\rangle\in\mathbb Z_+^m\times\mathds Z_2^n(s, \mathbf r)\mid \mathcal E\langle\gamma_\alpha,\mu\rangle=\mathcal E(\eta(\underline{\kappa}))=\underline{\kappa}\in\mathcal K(\kappa)\,\}
\end{equation}
for a generator subsystem consisting of representatives of energy grade $\kappa$ in $\Omega_q^{(s)}(m|n,\textbf{r})$.

The set of $m$-tuples of energy grade $\kappa$: $\mathcal K(\kappa)$ introduced in Remark 18  has two orders as follows. One is the partial order $\succeq$ in Subsection 3.3, which plays a crucial role in Lemma 12 or Theorem 14 (iii) to judge the subordinate relationship of a proper submodule in a bigger one. Another is the lexicographical order $\succ$ which is a total order among $m$-tuples of non-negative integers. Note that we have: 
\begin{equation}
    \begin{split}
\cdots\succ\underline{\kappa}&\succ\underline{\kappa}{-}\varepsilon_{m-1}{+}\varepsilon_m\quad \hskip3.3cm(\text{$m{-}1$-st line occurs if $k_{m-1}>0$})\\
&\succ\cdots\\ 
&\succ\underline{\kappa}{-}\varepsilon_j{+}\varepsilon_{j+1}\succ\cdots\succ\underline{\kappa}{-}\varepsilon_j{+}\varepsilon_m\qquad\quad(\text{$j$-th line occurs if $k_j>0$})\\
        &\cdots\\
  &\succ\underline{\kappa}{-}\varepsilon_i{+}\varepsilon_{i+1}\succ\cdots\succ\underline{\kappa}{-}\varepsilon_i{+}\varepsilon_m\qquad\quad \,\,(\text{$i$-th line occurs if $k_i>0$})\\      
   &\succ\cdots\\
   &\succ\underline{\kappa}{-}\varepsilon_1{+}\varepsilon_2\succ\cdots\succ\underline{\kappa}{-}\varepsilon_1{+}\varepsilon_m\succ \cdots.\quad (\text{$1$-st line occurs if $k_1>0$})    \end{split}
\end{equation}

Furthermore, $\mathcal K(\kappa)$ also inherits a third order $\succcurlyeq$, coming from the weight order relative to the system of simple positive roots  $\Delta=\{\,\alpha_i=\varepsilon_i-\varepsilon_{i+1}\mid 1\le i<m\,\}$ of type $A$, i.e., $\kappa\succcurlyeq\kappa'$ if and only if $\kappa-\kappa'\in Q^+$ (the  semi-lattice of positive roots w.r.t. $\Delta$). This is a partial order compatible with $e_i$- and $f_i$-actions for any $e_i, f_i\in \mathcal U_q(\mathfrak{sl}_m)\subset \mathcal U_q(\mathfrak{gl}(m|n))$. 

Clearly, the lexicographic order $\succ$ on each line of figure (4.13)
exactly coincides with the pre-order $\succcurlyeq$ given by the type-A weight system, namely, 
$\underline{\kappa}-\varepsilon_i+\varepsilon_j\succ\underline{\kappa}-\varepsilon_i+\varepsilon_k$ for $i<j<k$ implies  $\underline{\kappa}-\varepsilon_i+\varepsilon_j\succcurlyeq\underline{\kappa}-\varepsilon_i+\varepsilon_k$. 

However, for any two $m$-tuples $\underline{\kappa}' \succ \underline{\kappa}\in \mathcal K(\kappa)$ of energy grade $\kappa$ lying in different two lines ($i<j$) of some $\underline{\kappa}''\in\mathcal K(\kappa)$, as shown in (4.13)\,: 
$$
\underline{\kappa}=\underline{\kappa}''-\varepsilon_i+\varepsilon_p, \ (i<p), \qquad \underline{\kappa}'=\underline{\kappa}''-\varepsilon_j+\varepsilon_q,  \ (j<q). 
$$

(1) When $q\le p$, 
$\underline{\kappa}'-\underline{\kappa}=(\varepsilon_i-\varepsilon_p)-(\varepsilon_j-\varepsilon_q)\in Q^+$, so, $\underline{\kappa}'\succcurlyeq\underline{\kappa}$. 

(2) When $q>p$, $\underline{\kappa}'\not\succcurlyeq\underline{\kappa}$ and $\underline{\kappa}\not\succcurlyeq\underline{\kappa}'$. But in this case, we have some $\underline{\kappa}''\in\mathcal K(\kappa)$, such that $\underline{\kappa}''\succcurlyeq\underline{\kappa}$ and $\underline{\kappa}''\succcurlyeq\underline{\kappa}'$. 

Thereby, we arrive at the following
\begin{lemma} 
For any two $m$-tuples $\kappa$, $\kappa'$ in the totally ordered set $(\mathcal K(\kappa), \succ)$ with $\kappa'\succ\kappa$, then either they are compatible w.r.t. $\succcurlyeq$: $\kappa'\succcurlyeq\kappa$; or $\kappa'\not\succcurlyeq\kappa$ and $\kappa\not\succcurlyeq\kappa'$, but there exists a $\kappa''\in \mathcal K(\kappa)$, such that
$\kappa''\succcurlyeq\kappa'$, and $\kappa''\succcurlyeq\kappa$. That is, $(\mathcal K(\kappa), \succcurlyeq)$ is a preorder set.
\end{lemma}

For  $\eta(\underline{\kappa})\in \wp_s(\mathcal K(\kappa))$ with $s\in[\,\ell, N{-}\ell\,]$ and $\gamma_0(\eta(\underline{\kappa}))=
(\underset{\jmath}{\underbrace{\ell{-}1,\cdots,\ell{-}1}}, h, 0,\cdots,0)$ satisfying $1\le \jmath\le m-1$, $0\le h\le\ell-1$, and $h\ge 1$ only if $\jmath=1$ (since $s\ge \ell$). 

Let us consider a basic case: $(\underline{\kappa}+\varepsilon_j-\varepsilon_q, \underline{\kappa})$ in $\mathcal K(\kappa)$ with preorder $\underline{\kappa}+\varepsilon_j-\varepsilon_q\succcurlyeq \underline{\kappa}$ ($j<q$).
Note that to $\underline{\kappa}+\varepsilon_j\in\mathcal K(\kappa+1)$ for $\underline{\kappa}\in\mathcal K(\kappa)$ with $\underline{\kappa}+\varepsilon_j\succeq\underline{\kappa}+\varepsilon_j-\varepsilon_q$ and
$\underline{\kappa}+\varepsilon_j\succeq\underline{\kappa}$,  we can associate two equivalent 
$\langle m, n\rangle$-tuples: $\eta_1(\underline{\kappa}{+}\varepsilon_j)\thicksim\eta_2(\underline{\kappa}{+}\varepsilon_j)$ in $\mathbb Z_+^m\times\mathds Z_2^n(s, \mathbf r)$ so that $\mathfrak V_{\eta_1(\underline{\kappa}{+}\varepsilon_j)}=\mathfrak V_{\eta_2(\underline{\kappa}{+}\varepsilon_j)}$ (by Theorem 14 (ii)) 
as follows.  

Case (i): If $\jmath\ge 2, \ j<q\le\jmath:$
\begin{equation}
    \begin{split}
 \eta_1(\underline{\kappa}{+}\varepsilon_j)&=\langle\, \cdots, (k_j{+}1)\ell,\cdots\cdots\cdots\cdots,k_q\ell+(\ell{-}2),\cdots,  k_{\jmath+1}\ell+h,\cdots;\,\omega(\hslash)\,\rangle,   \\
 \eta_2(\underline{\kappa}{+}\varepsilon_j)&=\langle\, \cdots,\,(k_j {+}1)\ell+(\ell{-}2),\, \cdots,\,k_q\ell,\, \cdots\cdot\cdots\cdots,\,k_{\jmath+1}\ell+h,\cdots\,;\,\omega(\hslash)\,\rangle.
    \end{split}
\end{equation}
 Using the action formulae (1) \& (2) in (\cite{Hu}, 4.5) and Proposition 4.6 \cite{Hu}, $\exists$ $e_{\alpha_{jq}}$, $f_{\alpha_{jq}}\in\mathcal U_q(\mathfrak{sl}_m)\subset\mathcal U_q(\mathfrak{gl}(m|n))$, such that $f_{\alpha_{jq}}.\,x^{\langle \eta_1(\underline{\kappa}+\varepsilon_j)\rangle}=c_1\,x^{\langle \eta(\underline{\kappa})\rangle}$, and $e_{\alpha_{jq}}.\,x^{\langle \eta_2(\underline{\kappa}+\varepsilon_j)\rangle}=c_2\,x^{\langle \eta(\underline{\kappa}+\varepsilon_j-\varepsilon_q)\rangle}$, $(c_1, \,c_2$ $\in \mathbb K^*)$.
That is, $\mathfrak V_{\eta(\underline{\kappa})}\bigoplus \mathfrak V_{\eta(\underline{\kappa}+\varepsilon_j-\varepsilon_q)}\subsetneq \mathfrak V_{\eta_1(\underline{\kappa}+\varepsilon_j)}=\mathfrak V_{\eta_2(\underline{\kappa}{+}\varepsilon_j)}$.

Case (ii): If $\jmath\ge 2, \ \jmath{+}1<j<q:$
\begin{equation}
    \begin{split}
 \eta_1(\underline{\kappa}{+}\varepsilon_j)&=\langle\, \cdots, k_{\jmath+1}\ell+h,\cdots,(k_j{+}1)\ell,\cdots\cdots\cdots\cdots,k_q\ell+(\ell{-}2)\cdot  \cdots;\,\omega(\hslash)\,\rangle,   \\
 \eta_2(\underline{\kappa}{+}\varepsilon_j)&=\langle\, \cdots,\,k_{\jmath+1}\ell+h,\cdots,(k_j {+}1)\ell+(\ell{-}2), \cdots,\,k_q\ell, \cdots\cdots\cdots\cdots\,;\,\omega(\hslash)\,\rangle,\\
 \gamma_0(\eta_1(\underline{\kappa}{+}\varepsilon_j))&=(0,0,\ell{-}1,\cdots,\ell{-}1,\underset{\jmath+1}h,0,\cdots,\underset{j}0,\cdots,\underset{q}{\ell{-}2},0,\cdots,0),\\
 \gamma_0(\eta_2(\underline{\kappa}{+}\varepsilon_j))&= (0,0,\ell{-}1,\cdots,\ell{-}1,\underset{\jmath+1}h,0,\cdots,\underset{j}{\ell{-}2},\cdots,\underset{q}{0},0,\cdots,0),
    \end{split}
\end{equation}
In this case, by definition, $\eta_1(\underline{\kappa}{+}\varepsilon_j)-\varepsilon_j+\varepsilon_q\thicksim \eta(\underline{\kappa})$, and  $\eta_2(\underline{\kappa}{+}\varepsilon_j){+}\varepsilon_j{-}\varepsilon_q\thicksim \eta(\underline{\kappa}{+}\varepsilon_j{-}\varepsilon_q)$. 

Case (iii): If $\jmath\ge 1, \ j<\jmath{+}1<q:$
\begin{equation}
    \begin{split}
 \eta_1(\underline{\kappa}{+}\varepsilon_j)&=\langle\, \cdots, (k_j{+}1)\ell,\cdots\cdots\cdots\cdots,k_{\jmath+1}\ell+h,\cdots,k_q\ell+(\ell{-}2)\cdot  \cdots;\,\omega(\hslash)\,\rangle,   \\
 \eta_2(\underline{\kappa}{+}\varepsilon_j)&=\langle\, \cdots,(k_j {+}1)\ell+(\ell{-}2),\cdots,\,k_{\jmath+1}\ell+h,\cdots,\,k_q\ell, \cdots\cdots\cdots\cdots\,;\,\omega(\hslash)\,\rangle,\\
 \gamma_0(\eta_1(\underline{\kappa}{+}\varepsilon_j))&=(0,\ell{-}1,\cdots,\ell{-}1,\,\underset{j}0,\ \ell{-}1,\cdots\cdots, \ell{-}1,\underset{\jmath+1}h,0,\cdots,0,\underset{q}{\ell{-}2},0,\cdots,0),\\
 \gamma_0(\eta_2(\underline{\kappa}{+}\varepsilon_j))&= (0,\ell{-}1,\cdots,\ell{-}1,\underset{j}{\ell{-}2},\ell{-}1,\cdot\cdots,\ell{-}1,\underset{\jmath+1}h, 0,\cdots,0,\quad \underset{q}{0},\ 0,\cdots,0),
    \end{split}
\end{equation}
Clearly, we see that  $\eta_1(\underline{\kappa}{+}\varepsilon_j)-\varepsilon_j+\varepsilon_q\thicksim \eta(\underline{\kappa})$, and  $\eta_2(\underline{\kappa}{+}\varepsilon_j){+}\varepsilon_j{-}\varepsilon_q\thicksim \eta(\underline{\kappa}{+}\varepsilon_j{-}\varepsilon_q)$. 

Case (iv): If $\jmath\ge 1, \ j<\jmath{+}1=q:$ --- including two extreme cases: $\jmath=1$ with $h>0$; and $h=0$ but $\jmath>1$. For the latter case, since $(\ell{-}1,\ell{-}1,0)\thicksim (\ell{-}2,\ell{-}1,1)$, so we always assume that $h\ge 1$ without loss of generality
\begin{equation}
    \begin{split}
 \eta_1(\underline{\kappa}{+}\varepsilon_j)&=\langle\, \cdots, (k_j{+}1)\ell,\cdots\cdots\cdots\cdots,k_{\jmath+1}\ell+(h{-}1),\cdots;\,\omega(\hslash)\,\rangle,   \\
 \eta_2(\underline{\kappa}{+}\varepsilon_j)&=\langle\, \cdots,(k_j {+}1)\ell+(h{-}1),\cdots,\,k_{\jmath+1}\ell,\cdots\cdots\cdots\cdots\,;\,\omega(\hslash)\,\rangle,
    \end{split}
\end{equation}
It is clear that $\eta_1(\underline{\kappa}{+}\varepsilon_j)-\varepsilon_j+\varepsilon_{\jmath+1}\thicksim \eta(\underline{\kappa})$, and  $\eta_2(\underline{\kappa}{+}\varepsilon_j){+}\varepsilon_j{-}\varepsilon_{\jmath+1}\thicksim \eta(\underline{\kappa}{+}\varepsilon_j{-}\varepsilon_{\jmath+1})$. 

Case (v): If $\jmath\ge 1, \ j=\jmath{+}1<q:$  Same as Case (iv), we can assume $h\ge 1$ without loss of generality,
\begin{equation}
    \begin{split}
 \eta_1(\underline{\kappa}{+}\varepsilon_j)&=\langle\, \cdots, \,k_\jmath\ell, \qquad\,(k_j{+}1)\ell,\cdots\cdots,\,k_q\ell+(h{-}1),\cdots;\,\omega(\hslash)\,\rangle,   \\
 \eta_2(\underline{\kappa}{+}\varepsilon_j)&=\langle\, \cdots,\,k_\jmath\ell,\,(k_j {+}1)\ell+(h{-}1),\cdots,\,k_q\ell,\,\cdots\cdots\cdot\cdots\,;\,\omega(\hslash)\,\rangle,
    \end{split}
\end{equation}
By definition, clearly, $\eta_1(\underline{\kappa}{+}\varepsilon_j)-\varepsilon_j+\varepsilon_q\thicksim \eta(\underline{\kappa})$, and  $\eta_2(\underline{\kappa}{+}\varepsilon_j){+}\varepsilon_j{-}\varepsilon_q\thicksim \eta(\underline{\kappa}{+}\varepsilon_j{-}\varepsilon_q)$. 

In conclusion, by Theorem 14 (ii), all the above cases reveal the fact that 
$$
\mathfrak V_{\eta(\underline{\kappa}+\varepsilon_j-\varepsilon_q)}\,\bigoplus\, \mathfrak V_{\eta(\underline{\kappa})}\,\subsetneq\, \mathfrak V_{\eta_1(\underline{\kappa}+\varepsilon_j)}=\mathfrak V_{\eta_2(\underline{\kappa}{+}\varepsilon_j)}\,,\qquad \text{\it for } \quad \underline{\kappa}+\varepsilon_j-\varepsilon_q\succcurlyeq\underline{\kappa}\,.
$$

Note that for a given $s$, the stratification of ``energy grades'' of $m$-tuples with degree $s$ of non-negative integers (or elements in $\Omega_q^{(s)}(m|n)$) is distributed in a ``$\ell$-adic" way, 
for instance, for each $\gamma_0(\eta_i(\underline{\kappa}{+}\varepsilon_j))\thicksim(\underset{\jmath-1}{\underbrace{\ell{-}1,\cdots,\ell{-}1}},h{-}1,0,\cdots,0)
$, every upgrading one energy grade will shorten one $\jmath$-component and lower one degree in height. After raising energy grades up to a certain height, in the end of such lifting procedure we have to use $e_m,\,e_{m+1},\cdots\in \mathcal U_q(\mathfrak{gl}(m|n))$ to move $|\,\omega(\hslash)\,|$ ($>0$ if necessary) positions in the right hand side to the left in order to add energy grades. 

We are now in a position to establish a key claim/insight, which shows for any pair $(\underline{\kappa}', \underline{\kappa})$ in $\mathcal K(\kappa)$ (of the same energy grade $\kappa$) with pre-order $\underline{\kappa}'\succcurlyeq\underline{\kappa}$, how to embed the direct sum $\mathfrak V_{\eta(\underline{\kappa})}\bigoplus \mathfrak V_{\eta(\underline{\kappa}')}$ of energy grade $\kappa$ into a larger cyclic submodule $\mathfrak V_{\eta(\underline{\kappa}+\varepsilon_j)}$ of energy grade $\kappa+1$ in $\Omega_q^{(s)}(m|n,\textbf{r})$. Actually, combining Lemma 19 with the above argument, we can do a bit more as follows.

\begin{proposition} For any two successive $m$-tuples $(\underline{\kappa}', \underline{\kappa})$ in $\mathcal K(\kappa)$ with $\underline{\kappa}'\succ \underline{\kappa}$, then either

$(1)$ when  $\underline{\kappa}'\succcurlyeq\underline{\kappa}:$ $\mathfrak V_{\eta(\underline{\kappa}')}\,\bigoplus\,\mathfrak V_{\eta(\underline{\kappa})}\,\subsetneq\, \mathfrak V_{\eta(\underline{\kappa}+\varepsilon_j)}$, \ for some $j\in I_0;$ and or

$(2)$ when $\underline{\kappa}'\not\succcurlyeq\underline{\kappa}$, $($also $\underline{\kappa}\not\succcurlyeq\underline{\kappa}'):$ $\exists$ $\underline{\kappa}''\in\mathcal K(\kappa)$, satisfying    $\underline{\kappa}''\succcurlyeq\underline{\kappa}'$, and   $\underline{\kappa}''\succcurlyeq\underline{\kappa}$, such that 
$$
\mathfrak V_{\eta(\underline{\kappa}'')}\,\bigoplus\,\mathfrak V_{\eta(\underline{\kappa})}\,\bigoplus\,\mathfrak V_{\eta(\underline{\kappa}')}\,\subsetneq\, \mathfrak V_{\eta(\underline{\kappa}+\varepsilon_i)}\,\bigoplus\,\mathfrak V_{\eta(\underline{\kappa}'+\varepsilon_j)}, \quad \text{\it for some } \ i, \, j\in I_0.
$$
\end{proposition}

\begin{coro} The graph with $\cup_{\kappa=E_0(s)}^{E(s)}\mathcal K(\kappa)$ as its vertex set and edges given by the preorder $\succeq$ between vertices from grade $\kappa$ to grade $\kappa+1$ weaves a net describing inclusions of submodules pertaining to $\Omega_q^{(s)}(m|n,\textbf{r})$, 
which consists of cyclic submodules $\mathfrak V_{\eta(\underline{\kappa})}$ for all vertex indices $\underline{\kappa}$ with $|\,\eta(\underline{\kappa})\,|=s$. 
\end{coro}

\subsection{Loewy filtrations of $\Omega_q^{(s)}(m|n,\textbf{r})$ and Loewy layers}
For the indecomposable $\mathcal U_q$-modules $\Omega_q^{(s)}(m|n,\textbf{r})$ with $s\in \bigl[\, \ell, \,N{-}\ell\,]$ (see Theorem 17 (B) \& (C)), we will use a method of filtration to explore its submodule-constituents.

Associated to (4.12), the generator subsystem of grade $\kappa$
$$
\wp_s(\mathcal K(\kappa))=\{\, \eta(\underline{\kappa})=\langle\gamma_\alpha,\mu\rangle\in\mathbb Z_+^m\times\mathds Z_2^n(s, \mathbf r)\mid \mathcal E\langle\gamma_\alpha,\mu\rangle=\mathcal E(\eta(\underline{\kappa}))=\underline{\kappa}\in\mathcal K(\kappa)\,\}\subsetneq \mathbb Z_+^m\times \mathds Z_2^n(s, \mathbf r),$$ 
we rewrite 
\begin{equation}
    \mathcal K^{(s)}_i:=\mathcal K\bigl(E_0(s)+i\bigr)=\{\,\underline{\kappa}=(k_1,\cdots,k_m)\in \mathbb Z_+^m\mid |\,\underline{\kappa}\,|=\kappa=E_0(s) + i, \,  \forall j k_j\leq r-1\,\},     
\end{equation}
for $0\leq i\leq E(s)-E_0(s)$.
\begin{defi}
Set $\mathcal{V}_0=\text{Soc}\,\Bigl(\Omega_q^{(s)}(m|n,\textbf{r})\Bigr)$, and for $\,0<i\le E(s)-E_0(s)$, define 
\begin{equation}
\mathcal{V}_i=\text{\rm Span}_\mathbb{K}\left.\left\{\,x^{\langle\,\alpha, \mu\,\rangle}\in \Omega_q^{(s)}(m|n,\textbf{r})\ \right|\ E_0(s) \leq\mathcal E\langle\alpha, \mu\rangle\leq E_0(s)+i\,\right\}.
\end{equation}
Obviously,  $\mathcal{V}_{i-1}\subseteq \mathcal{V}_i$, for $0\le i\le E(s)-E_0(s)$.  
\end{defi}

Recall in Theorem 5 (2), we proved 
\[
\Omega_q^{(s')}(m|n,\textbf{1})\cong
\begin{cases}
V\bigl((\ell{-}1{-}s_j)\,\omega_{j-1}{+}s_j\omega_j\bigr)=\mathfrak V_{\underline{\textbf{s}}'_j},  \ & \text{\it if } \ s'=(j{-}1)(\ell{-}1)+s_j, \ 1\le j\le m,
 \\
 V\bigl((\ell{-}2)\,\omega_m+\omega_{m+\jmath}\bigr)=\mathfrak V_{\tau(\jmath)},  \ & \text{\it if } \ s'=m(\ell{-}1)+\jmath,  \  0\leq \jmath\leq n,
\end{cases}
\]
is a simple module, where $\underline{\textbf{s}}'_j=\langle\,\ell{-}1,\cdots,\ell{-}1, s_j, 0, \cdots, 0\,;\,\mathbf 0\,\rangle$, $0\leq s_j\leq \ell{-}1$, 
$\omega_j=\epsilon_1+\cdots+\epsilon_j$ is the $j$-th fundamental weight, and
$\tau(\jmath)=\langle\,\ell{-}1,\cdots,\ell{-}1\,;\,\omega(\jmath)\,\rangle$.

For $s\in \bigl[\, \ell, \,N{-}\ell\,]$, we define $s'=s-\kappa\,\ell$, and $\eta(\underline{\kappa})= \langle\,\underline{\kappa}\,\ell\,;\,\mathbf 0\,\rangle+\eta_j$,
\begin{equation}
    \eta_j=\begin{cases}
       \underline{\mathbf s}_j', \ & \ \text{\it if } \ s'=(j{-}1)(\ell{-}1)+s_j, \ 0\leq s_j\leq \ell{-}1, \ 1\le j\le m,\\
   \tau(j), \  & \ \text{\it if } \ s'=m(\ell{-}1)+j,  \  0\leq j\leq n.
\end{cases}    
\end{equation}

Recall that a Loewy filtration is a semi-simple filtration (i.e. with semi-simple successive quotients) of minimal length. The socle filtration and the radical filtration are examples of Loewy filtration ( see e.g, \cite{Hum1}, 8.14, p. 174\,)

\begin{theorem} $($Loewy Filtration$)$
For $s\in \bigl[\, \ell, \,N{-}\ell\,]$, the indecomposable $\mathcal U_q$-module $\Omega_q^{(s)}(m|n,\textbf{r})$ has the following filtration structure:  
\text{\rm (1)} $\mathcal{V}_i=\sum_{\eta(\underline{\kappa})\in \wp_s(\mathcal K^{(s)}_i)}\mathfrak V_{\eta(\underline{\kappa})}$, for $0\le i\le E(s)-E_0(s)$, is an increasing $\mathcal U_q$-submodules of $\Omega_q^{(s)}(m|n,\emph{\textbf{r}})$, and the chain of these submodules
\begin{equation}
 0\subset \mathcal{V}_0\subset\mathcal{V}_1\subset \cdots \subset \mathcal{V}_{E(s)-E_0(s)}=\Omega_q^{(s)}(m|n,\textbf{r})
\end{equation}
forms a Loewy filtration $($for definition, see \cite{Hum1}, 8.14, p. 174\,$)$ of \,$\Omega_q^{(s)}(m|n,\textbf{r})$.

\text{\rm (2)} $x^{\langle\,\eta(\underline{\kappa})\,\rangle}$'s, for all $\eta(\underline{\kappa})\in \wp_s(\mathcal K_i^{(s)})$ with $\underline{\kappa}\in \mathcal{K}_i^{(s)}$, are primitive vectors $($relative to $\mathcal V_{i-1}$$)$ of grade $\kappa$ in $\mathcal{V}_i$ $(\,i.e.$, $e_j.\,(x^{\langle\,\eta(\underline{\kappa})\,\rangle}+\mathcal V_{i-1})\in\mathcal V_{i-1}$, for all $j\in I$\,$)$, and 
$$
\mathcal U_q.\, (x^{\langle\eta(\underline{\kappa})\,\rangle}+\mathcal{V}_{i-1} )\cong\mathfrak{V}_{\langle\eta(\underline{\kappa})\,\rangle}/\bigl(\mathfrak{V}_{\langle\eta(\underline{\kappa})\,\rangle}\,\cap\mathcal V_{i-1}\bigr).
$$ 
Its $i$-th Loewy layer
\begin{equation}
\begin{split}
\mathcal{V}_i/\mathcal{V}_{i-1}&=\emph{\text{Span}}_\mathbb{K}\{\,x^{\langle\,\alpha, \mu\,\rangle}+\mathcal{V}_{i-1}\mid \mathcal E(\langle\,\alpha, \mu\,\rangle)=E_0(s)+i\,\}\\
&=\bigoplus\limits_{\eta(\underline{\kappa})\in \mathcal K_i^{(s)}}\mathcal U_q.\, \bigl(x^{\langle\,\eta(\underline{\kappa})\,\rangle}+\mathcal{V}_{i-1}\bigr)=\text{Soc}\,\Bigl(\Omega_q^{(s)}(m|n,\textbf{r})/\mathcal{V}_{i-1}\Bigr)\\
&\cong (\sharp\,\mathcal K_i^{(s)})\,\mathfrak{V}_{\gamma(\eta(\underline{\kappa}))}
\cong  (\sharp\,\mathcal K_i^{(s)})\,\Omega_q^{(s')}(m|n,\textbf{1})
\end{split}
\end{equation}
is the direct sum of $\sharp\,\mathcal K^{(s)}_i$ isomorphic copies of simple module $\mathfrak
V_{\gamma(\eta(\underline{\kappa}))}=\Omega_q^{(s')}(m|n,\textbf{1})$, 
where $\gamma(\eta(\underline{\kappa}))=\eta(\underline{\kappa})-\langle\,\underline{\kappa}\,\ell\,; \mathbf 0\,\rangle=\eta_j\in\mathbb Z_+^m\times\mathds Z_2^n(s',\bf 1)$, 
$($\,see $(4.1)\,)$, $s'=s-\kappa\,\ell=|\,\gamma(\eta(\underline{\kappa}))\,|$.
\end{theorem}
\begin{proof} 
(1) By definition of $E_0(s)$, $\text{Edeg}\,\bigl(x^{\langle\,\alpha,\mu\,\rangle}\bigr)\geq E_0(s)$,  for all $x^{\langle\,\alpha,\mu\,\rangle}\in \mathcal{V}_i\subseteq \Omega_q^{(s)}(m|n,\emph{\textbf{r}})$. Meanwhile, Lemma 9 gives $\text{Edeg}\,\bigl(u.\, x^{\langle\,\alpha,\mu\,\rangle}\bigr)\leq \text{Edeg}\,(x^{\langle\,\alpha,\mu\,\rangle})\leq E_0(s)+i$. Thus, Definition 8 implies that $\mathcal{V}_i$ is a $\mathcal U_q$-submodule of $\Omega_q(m|n,\textbf{r})$.

If $\text{Edeg}\,(x^{\langle\,\alpha,\mu\,\rangle})=E_0(s)+i$, then $x^{\langle\,\alpha,\mu\,\rangle}\notin \mathcal{V}_{i-1}$, by definition, $\mathcal{V}_i/\mathcal{V}_{i-1}$ is spanned by $\{\,x^{\langle\,\alpha,\mu\,\rangle}+\mathcal{V}_{i-1}\mid\mathcal E\,(\langle\,\alpha,\mu\,\rangle)=E_0(s)+i\,\}$. Theorem 14 (ii) indicates that $\mathcal K_i^{(s)}$ parameterizes the generating set of $\mathcal{V}_i/\mathcal{V}_{i-1}$.  Combining Proposition 20 with a similar method of the proof of Theorem 16 (3), one can deduce that $\mathcal{V}_i=\sum_{\eta(\underline{\kappa})\in \wp_s(\mathcal K^{(s)}_i)}\mathfrak V_{\eta(\underline{\kappa})}$. Moreover, Theorem 14 (i) \& (iii) implies that $\mathcal{V}_i/\mathcal{V}_{i-1}$ is semisimple and maximal (by Theorem 16 (1)). %Since  $\Omega_q^{(s)}(m|n,\textbf{r})/\mathcal{V}_{i-1}$ is still indecomposable, $\mathcal{V}_i/\mathcal{V}_{i-1}$ is maximal semisimple, that is,   
Namely, (4.22) is a Loewy filtration, by definition. This is equivalent to say $$\mathcal{V}_i/\mathcal{V}_{i-1}=\text{Soc}\Bigl(\Omega_q^{(s)}(m|n,\textbf{r})/\mathcal{V}_{i-1}\Bigr).$$

(2) When $s'=s-\kappa\,\ell\leq m(\ell{-}1)$, by (4.21), that is, $1\le j\le m$, $0\le s_j\le \ell{-}1$, we assert that $x^{\langle\,\eta(\underline{\kappa})\,\rangle}$ with $\eta(\underline{\kappa})=\langle\,\underline{\kappa}\,\ell\,;\,\mathbf 0\,\rangle+\underline{\textbf{s}}_j'$ is a primitive vector of $\mathcal{V}_i$ relative to $\mathcal{V}_{i-1}\ (i\geq 1)$. 

In fact, by Proposition 4, we have
$$
e_t.\, x^{\langle\,\eta(\underline{\kappa})\,\rangle}=\begin{cases}
      [\,k_t\ell+\ell\,]\,x^{\langle\,\eta(\underline{\kappa})+\varepsilon_t-\varepsilon_{t+1}\,\rangle}=0, &\qquad t<j\le m,\\
[\,k_t\ell+\delta_{j,t}s_j+1\,]\,x^{\langle\,\eta(\underline{\kappa})+\varepsilon_t-\varepsilon_{t+1}\rangle}, &\qquad j\leq t< m,\\
      0, &\qquad t\geq m.
      \end{cases}
      $$
Namely, 
\text{\rm (i)} $j=m:$ \ $e_t.\, x^{\langle\,\eta(\underline{\kappa})\,\rangle}=0$, for $1\leq t< m+n$, $x^{\langle\,\eta(\underline{\kappa})\,\rangle}$ is a maximal weight vector.

\text{\rm (ii)} $j<m:$ \ $e_t.\, x^{\langle\,\eta(\underline{\kappa})\,\rangle}=0$, for $t<j$, or $t\geq m$; and $e_t.\, x^{\langle\,\eta(\underline{\kappa})\,\rangle}=c\,x^{\langle\,\eta(\underline{\kappa})+\varepsilon_t-\varepsilon_{t+1}\rangle}\in \mathcal{V}_{i-1}$, for $j\leq t< m$ and $c\in \mathbb{K}$. So, $x^{\langle\,\eta(\underline{\kappa})\,\rangle}$ is a primitive vector of $\mathcal{V}_i$ relative to $\mathcal{V}_{i-1}\ (i\geq 1)$.

When $s'=s-\kappa\,\ell\ge m(\ell{-}1)$, by (4.21), that is, $0\le j\le n$, we see that $x^{\langle\,\eta(\underline{\kappa})\,\rangle}$ with $\eta(\underline{\kappa})=\langle\,\underline{\kappa}\,\ell\,;\,\mathbf 0\,\rangle+\tau(j)$ is a maximal weight vector, as $e_t.\, x^{\langle\,\eta(\underline{\kappa})\,\rangle}=0$ for  all $t\in J$, by Proposition 4.

Set $\overline{\mathfrak {V}}_{\eta(\underline{\kappa})}:=\mathcal U_q.\,
(x^{\langle\,\eta(\underline{\kappa})\,\rangle}+\mathcal{V}_{i-1})$, for $\eta(\underline{\kappa})\in\wp_s(\mathcal K_i^{(s)})$. By Lemma 11 and Theorem 5 (2), as well as the proof of Proposition 20, we get that
\begin{equation*}
%\begin{split}
\overline{\mathfrak{V}}_{\eta(\underline{\kappa})} \cong \mathcal U_q.\,x^{\langle\,\eta(\underline{\kappa})\,\rangle}/\bigl(\mathcal U_q.\,x^{\langle\,\eta(\underline{\kappa})\,\rangle}\cap \mathcal{V}_{i-1}\bigr)
\cong \mathcal U_q.\, x^{\langle\,\eta_j\,\rangle}=\mathfrak{V}_{\eta_j}=\Omega_q^{(s')}(m|n,\textbf{1}).
%\end{split}
\end{equation*}
So, $\overline{\mathfrak{V}}_{\langle\,\eta(\underline{\kappa})\,\rangle}$ is a simple submodule of $\mathcal{V}_i/\mathcal{V}_{i-1}$.
%
%For any $\underline{\kappa}, \, \underline{\kappa}'\in \mathcal K_i^{(s)}$ with $\underline{\kappa}\neq\underline{\kappa}'$, that is,  $\eta(\underline{\kappa})\nsim \eta(\underline{\kappa})$, by Theorem 14 (i), $\overline{\mathfrak{V}}_{\eta(\underline{\kappa})}$, $\overline{\mathfrak{V}}_{\eta(\underline{\kappa}')}$ are simple submodules of $\mathcal{V}_i/\mathcal{V}_{i-1}$ with $\overline{\mathfrak{V}}_{\eta(\underline{\kappa})}\cap \overline{\mathfrak{V}}_
%{\eta(\underline{\kappa}')}=0$, but $\overline{\mathfrak{V}}_{\eta(\underline{\kappa})}\cong \overline{\mathfrak{V}}_{\eta(\underline{\kappa}')}\cong \mathfrak{V}_{\langle\,\eta_j\,\rangle}
%=\Omega_q^{(s')}(m|n,{\textbf{1}})$. As $\mathcal K_i^{(s)}$ parameterizes the generator set of $\mathcal{V}_i/\mathcal{V}_{i-1}$, $\mathcal{V}_i/\mathcal{V}_{i-1}=\bigoplus_{\underline{\kappa}\in \mathcal{K}_i^{(s)} }\overline{\mathfrak{V}}_{\underline{\kappa}}\cong(\sharp\,\mathcal{K}_i^{(s)})\,\mathfrak{V}_{\eta_j}$.
\end{proof}

As a consequence of Theorem 23, we obtain some similar combinatorial formulas.
\begin{coro}
\text{\rm (1)} $\sharp\,\mathcal{K}_i^{(s)}
=\sum\limits_{j=0}^{\lfloor\frac{E_0(s)+i}{r}\rfloor}(-1)^j\binom{m}{j}\binom{m{+}(E_0(s){+}i){-}jr{-}1}{m{-}1}.$

\text{\rm (2)}
For $s'=s-\kappa\,\ell\leq m(\ell{-}1)$, or $m(\ell{-}1)<s'\leq m(\ell{-}1)+n$, we have
\begin{equation*}
\begin{split} &\dim\,\Omega_q^{(s)}(m|n,\textbf{r})=\sum\limits_{i=0}^{n}\sum\limits_{j=0}^{\lfloor\frac{s-i}{r\ell}\rfloor}(-1)^j\binom{n}{i} \binom{m}{j}\binom{m{+}s{-}i{-}jr\ell{-}1}{m{-}1}\\
&=\sum\limits_{j=0}^{\lfloor\frac{E_0(s)+i}{r}\rfloor}(-1)^j\binom{m}{j}\binom{m{+}(E_0(s){+}i){-}jr{-}1}{m{-}1}\times \sum\limits_{i=0}^n\sum\limits_{j=0}^{\lfloor \frac {s_j'-i}{\ell}\rfloor}\binom{n}{i}
\binom{m}{j}\binom{m{+}s_j'{-}i{-}j\ell{-}1}{m{-}1}. 
\end{split}
\end{equation*}
%
%\text{\rm (3)} For we have
%\begin{equation*}
%\begin{split}
%& \dim\,\Omega_q^{(s)}(m|n,\textbf{r})=
%\sum\limits_{i=0}^{n}\sum\limits_{j=0}^{\lfloor\frac{s-i}{r\ell}\rfloor}(-1)^j
%\binom{n}{i}\binom{m}{j}\binom{m{+}s{-}i{-}jr\ell{-}1}{m{-}1}\\
%& =\sum\limits_{j=0}^{\lfloor\frac{E(s)_0+i}{r}\rfloor}(-1)^j\binom{m}{j}\binom{m{+}
%(E(s)_0{+}i){-}jr{-}1}{m{-}1}\times \sum\limits_{i=0}^n\sum\limits_{j=0}^{\lfloor 
%\frac {s_j'-i}{\ell}\rfloor}\binom{n}{i}
%\binom{m}{j}\binom{m{+}s_j'{-}i{-}j\ell{-}1}{ m{-}1}.
%\end{split}
%\end{equation*}
\begin{proof}
\text{\rm (1)} From the definition of $\mathcal{K}_i^{(s)}$, $\sharp\,\mathcal{K}_i^{(s)}$ is equal to the coefficient of $t^{E_0(s)+i}$ of polynomial $P_{m,r}(t)=(1+t+t^2+\cdots+t^{r-1})^m$. So, it is true.

\text{\rm (2)} For $s'=s-\kappa\,\ell\leq m(\ell{-}1)$, or $m(\ell{-}1)<s'\leq m(\ell{-}1)+n$, it follows from \text{\rm (1)},  Lemma 6 \& Proposition 7, as well as
$$\Omega_q^{(s)}(m|n,\textbf{r})\cong \bigoplus\limits_{i=0}^{E(s)-E_0(s)}\mathcal{V}_i/\mathcal{V}_{i-1}\cong \bigoplus\limits_{i=0}^{E(s)-E_0(s)}\sharp\,\mathcal{K}_i^{(s)}\,\Omega_q^{(s')}(m|n,\textbf{1}),$$
as vector spaces.
\end{proof}
\end{coro}

\subsection{Example} Now we give an example to show the structural variations of $\Omega_q^{(s)}(m|n,\textbf{r})$ by increasing the degree $s$, as well as energy grades. In the following picture, each point represents one simple submodule of a Loewy layer, and each arrow represents the linked relationships existing among the simple subquotients. For example, $a\rightarrow b$ means that there exists $u\in\mathcal  U_q$ such that $u.\, a=b$.
\begin{example}
In $\Omega_q^{(s)}(m|n,\textbf{r})$, set $m=3=\ell,\ n=2=r:$ $N=m(r\ell{-}1)+n=17$, consider
the cases when $1\le s\le 17$.
\begin{figure}[h]
\centering
\includegraphics[width=9cm,height=2cm]{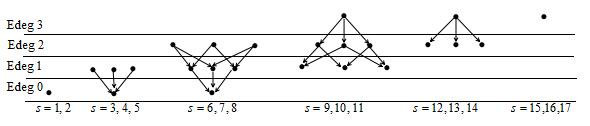}
\caption{Loewy filtrations (submodules-including nets) of $\Omega^{(s)}(m|n, \mathbf r)$ for $s\le 17$}
\label{fig:1}
\end{figure}
%\vspace{16cm}
$$
\mathcal K(\imath)=\{\,\underline{\kappa}=k_1k_2k_3\mid k_1+k_2+k_3=\imath, \ 0\le k_i\le 1\,\}.
$$ 
that is,
\begin{equation*}
\mathcal K(0)=\{\,000\,\}, \quad 
\mathcal K(1)=\{\,100,010,001\,\}, \quad
\mathcal K(2)=\{\,110,101,011\,\},\quad
\mathcal K(3)=\{\,111\,\}.
\end{equation*}

For any $\langle\alpha,\emph{\textbf{0}}\rangle\in \bigcup\limits_{i=0}^{m(r-1)}\mathcal K(i)\,\ell\times \mathds{Z}_2^2$, $\exists$ $\langle\,\beta, \mu\,\rangle\in \mathbb Z_+^3\times\mathds Z_2^2(3|2, \mathbf 2)$ where $\beta_i\leq\ell{-}1$, such that $|\langle\alpha,\emph{\textbf{0}}\rangle+\langle\beta,\mu\rangle|=s$, $\mathcal E\langle\,\alpha+\beta;\mu\,\rangle=\mathcal E(\alpha)$. 
Loewy filtrations of indecomposable $\mathcal U_q$-modules $\Omega^{(s)}(m|n, \mathbf r):$

\medskip\noindent
\text{\rm (1)} For $s=1,\, 2:$ \ $E_0(s)=0=E(s)$, $\mathcal K(0)$.
$ \emph{\text{Soc}}\ \Omega_q^{(s)}(3|2,\emph{\textbf{2}})=\emph{\text{Span}}_{\mathbb{K}}\{\,x^{\langle\,\alpha+\beta, \mu\,\rangle}\mid \mathcal E(\alpha)=000\,\}=\mathfrak{V}_{\langle \eta\rangle}=
 \Omega_q^{(s)}(3|2,\emph{\textbf{2}})$
is simple.

\medskip\noindent
\text{\rm (2)} For $s=3, \,4, \,5:$ \ $E_0(s)=0$, $E(s)=1$, $\mathcal K(0)$,
$\mathcal K(1)$.
$$
\emph{\text{Soc}}\ \Omega_q^{(s)}(3|2,\emph{\textbf{2}})=\emph{\text{Span}}_{\mathbb{K}}\{\,x^{\langle\alpha+\beta, \mu\rangle}\mid \mathcal E(\alpha)=000\,\}=\mathfrak V_{\langle\eta(000)\rangle}\subset
 \sum_{\mathcal E(\alpha)\in\mathcal K(1)}\mathfrak{V}_{\langle\alpha+\beta,\mu\rangle}=\Omega_q^{(s)}(3|2,\emph{\textbf{2}}).
$$
\text{\rm (3)} For $s=6,\, 7,\, 8:$ \ $E_0(s)=0$, $E(s)=2$, $\mathcal K(0)$, $\mathcal K(1)$,
$\mathcal K(2)$.
\begin{equation*}
\begin{split}
 \emph{\text{Soc}}\ \Omega_q^{(s)}(3|2,\emph{\textbf{2}})&=\emph{\text{Span}}_{\mathbb{K}}\{\,x^{\langle\,\alpha+\beta, \mu\,\rangle}\mid \mathcal E(\alpha)=000\,\}=\mathfrak V_{\langle\,\eta(000)\,\rangle}\\
 &\subset \sum_{\mathcal E(\alpha)\in\mathcal K(1)}\mathfrak{V}_{\langle\,\alpha+\beta,\mu\,\rangle}\subset \sum_{\mathcal E(\alpha)\in\mathcal K(2)}\mathfrak{V}_{\langle\,\alpha+\beta,\mu\,\rangle}=
\Omega_q^{(s)}(3|2,\emph{\textbf{2}}).
\end{split}
\end{equation*}
\text{\rm (4)} For $s=9, \,10,\, 11:$ \ $E_0(s)=1$, $E(s)=3$, $\mathcal K(1)$, $\mathcal K(2)$,
$\mathcal K(3)$.
\begin{equation*}
\begin{split}
 \emph{\text{Soc}}\ \Omega_q^{(s)}(3|2,\emph{\textbf{2}})&=\emph{\text{Span}}_{\mathbb{K}}\{\,x^{\langle\,\alpha+\beta, \mu\,\rangle}\mid \mathcal E(\alpha)\in\mathcal K(1)\,\}= \sum_{\mathcal E(\alpha)\in\mathcal K(1)}\mathfrak{V}_{\langle\,\alpha+\beta,\mu\,\rangle}\\
 &\subset \sum_{\mathcal E(\alpha)\in\mathcal K(2)}\mathfrak{V}_{\langle\,\alpha+\beta,\mu\,\rangle}\subset
\mathfrak{V}_{\langle\,\eta(111)\,\rangle}=\Omega_q^{(s)}(3|2,\emph{\textbf{2}}).
\end{split}
\end{equation*}
\text{\rm (5)} For $s=12,\, 13,\, 14:$ \ $E_0(s)=2$, $E(s)=3$, $\mathcal K(2)$, $\mathcal K(3)$.
\begin{gather*}
 \emph{\text{Soc}}\ \Omega_q^{(s)}(3|2,\emph{\textbf{2}})=\bigoplus_{\underline{\kappa}\in\mathcal K(2)}\mathfrak V_{\langle\,\eta(\underline{\kappa})\,\rangle}\subset \mathfrak{V}_{\langle\,\eta(111)\,\rangle}=
 \Omega_q^{(s)}(3|2,\emph{\textbf{2}}).
\end{gather*}
\text{\rm (6)} For $s=15,\, 16,\, 17:$ \ $E_0(s)=E(s)=3$, $\mathcal K(3)$.
$\Omega_q^{(s)}(3|2,\emph{\textbf{2}})=\mathfrak{V}_{\langle\, \eta(111), \,\mu\,\rangle}
=\emph{\text{Soc}}\ \Omega_q^{(s)}(3|2,\emph{\textbf{2}})
$
is simple.
\end{example}
%\begin{example}
%For the case $n+m(\ell{-}1)> m(r-1)l$, we take $m=3,n=4,r=2,l=3$ for example and and analyze it in the same way as above.
%\begin{figure}[h]
%\centering
%\includegraphics[width=11cm,height=2cm]{T2.png}
%\caption{Example for $m=3,n=4,r=2,l=3$}
%\label{fig:1}
%\end{figure}
%\vspace{16cm}
%\end{example}

\subsection{Rigidity of $\Omega_q^{(s)}(m|n,\textbf{r})$}
For a module $M$, both radical filtration and socle filtration are Loewy filtrations, and $\text{Rad}^{r-k}M\subseteq \text{Soc}^k\,M$, with $r=\ell\ell\, M$ being the Loewy length of $M$. In what follows, we will demonstrate the coincidence of both filtrations for $\Omega_q^{(s)}(m|n,\textbf{r})$, namely, the indecomposable $\mathcal U_q$-module $\Omega_q^{(s)}(m|n,\textbf{r})$ for any $s$ is rigid in the sense of 8.14 \cite{Hum1} (also see \cite{Ir1, Ir2}).
\begin{theorem} $($Rigidity$)$
Suppose $m\ge 2$, and $\textbf{char}(q)=\ell\geq 3$. Then the $\mathcal U_q$-modules $\Omega_q^{(s)}(m|n,\textbf{r})$ are rigid, and $\ell\ell\,\Omega_q^{(s)}(m|n,\textbf{r})=E(s)-E_0(s)+1$.
\begin{proof}
By definition of rigidity of module, it suffices to prove that the filtration (4.22) in Theorem 23 is both socle and radical.

\text{\rm (1)} Note $\text{Soc}^0\,\Omega_q^{(s)}(m|n,\textbf{r})=0$, $\text{Soc}^1\,\Omega_q^{(s)}(m|n,\textbf{r})=\mathcal{V}_0$, by Theorem 23. Assume that we have proved $\text{Soc}^i\,\Omega_q^{(s)}(m|n,\textbf{r})=\mathcal{V}_{i-1}$, for $i\geq 1$. By Theorem 23, we see that $\mathcal{V}_i/\mathcal{V}_{i-1}=\text{Soc}\,(\Omega_q^{(s)}(m|n,\textbf{r})/\mathcal{V}_{i-1})$, that is, $\text{Soc}^{i+1}\,\Omega_q^{(s)}(m|n,\textbf{r})=\mathcal{V}_i$. So, (4.22) is a socle filtration.

(2) By Theorem 23, (4.22) is a Loewy filtration of $\Omega_q^{(s)}(m|n,\textbf{r})$, so its Loewy length $r=\ell\ell\,\Omega_q^{(s)}(m|n,\textbf{r})$ $=E(s)-E_0(s)+1$. Then for $0\leq i\leq E(s)-E_0(s)$, we have
$$\text{Rad}^i(\Omega_q^{(s)}(m|n,\textbf{r}))\subseteq \text{Soc}^{r-i}(\Omega_q^{(s)}(m|n,\textbf{r}))=\mathcal{V}_{E(s)-E_0(s)-i}.$$

For $i=1$: if  $\exists$  $0\neq y\in \mathcal{V}_{E(s)-E_0(s)-1}$, and $y\notin \text{Rad}^1(\Omega_q^{(s)}(m|n,\textbf{r}))$, by definition, there is a maximal proper submodule $\mathcal{V}\subset \Omega_q^{(s)}(m|n,\textbf{r})$ such that $y\notin \mathcal{V}$. Since $\mathcal{V}$ is maximal, $\mathcal U_q.\, y+\mathcal{V}=\Omega_q^{(s)}(m|n,\textbf{r})=\sum_{|\langle \alpha,\mu\rangle|=E(s)}\mathcal U_q.\, x^{\langle \alpha,\mu\rangle}$, by Theorem 16 (3). However, $\text{Edeg}\,u.\, y\leq \text{Edeg}\,y=E(s)-1$, so we get that $\{x^{\langle \alpha,\mu\rangle}\in \Omega_q^{(s)}(m|n,\textbf{r})\mid \text{Edeg}\,x^{\langle \alpha,\mu\rangle}=E(s)\,\}\subseteq \mathcal{V}$. Hence, $\mathcal{V}=\Omega_q^{(s)}(m|n,\textbf{r})$, it is contrary to the assumption above. This means $\text{Rad}^1(\Omega_q^{(s)}(m|n,\textbf{r}))=\mathcal{V}_{E(s)-E_0(s)-1}$. 

Assume we have proved that $\text{Rad}^i(\Omega_q^{(s)}(m|n,\textbf{r}))=\mathcal{V}_{E(s)-E_0(s)-i}$, for $i\geq 1$. Notice that $\text{Rad}^{i+1}(\Omega_q^{(s)}(m|n,\textbf{r}))\subseteq \mathcal{V}_{E(s)-E_0(s)-i-1}\subset \mathcal{V}_{E(s)-E_0(s)-i}=\text{Rad}^i(\Omega_q^{(s)}(m|n,\textbf{r}))$. By definition, $\text{Rad}^{i+1}(\Omega_q^{(s)}(m|n,\textbf{r}))$ is the intersection of all maximal submodules of $\text{Rad}^i(\Omega_q^{(s)}(m|n,\textbf{r}))$. By Theorem 23 (2), we have $\mathcal{V}_{E(s)-E_0(s)-i}$ is spanned by $\{\,x^{\langle \alpha,\mu\rangle}\in (\Omega_q^{(s)}(m|n,\textbf{r}))\mid \text{Edeg}\,x^{\langle \alpha,\mu\rangle}=E(s)-i\,\}$. Using a similar argument as for $i=1$, we get $\text{Rad}^{i+1}(\Omega_q^{(s)}(m|n,\textbf{r}))=\mathcal{V}_{E(s)-E_0(s)-i-1}$.

Consequently, the filtration (4.22) is a radical filtration.
\end{proof}
\end{theorem}

\begin{coro}
Suppose $m\ge 2$, and $\textbf{char}(q)=\ell\geq 3$. Then for any $s\in\mathds N$, the $\mathcal U_q$-submodules $\Omega_q^{(s)}(m|n)$ of $\Omega_q(m|n)$ are indecomposable and rigid.
\begin{proof}
Since for any $s\in \mathds{N}$, there exists some $r\in \mathds{N}$ such that $(r{-}1)\ell\leq s\leq r\ell{-}1$, then $\Omega_q^{(s)}(m|n,\textbf{r})=\Omega_q^{(s)}(m|n)$. By Theorems 23 \& 26, $\Omega_q^{(s)}(m|n)$ is indecomposable and rigid $\mathcal U_q$-module.
\end{proof}
\end{coro}

%\begin{remark}
%The definitions of rigid module, socle filtration, radical filtration can be found in $($\,\cite{Hum1}, 8.14, p. 174\,$)$. Some relevant elegant investigations on rigidity of a module and Loewy filtration can be consulted in \cite{Ir1, Ir2}. Note that
%Kazhdan–Lusztig polynomials are given by Loewy filtrations of the rigid Verma modules in the principle category $\mathcal O_0$, which is closely related to the Kazhdan–Lusztig conjecture
%for $\mathcal O_0$.
%\end{remark}

\section{Quantum Grassmann superalgebra and quantum de Rham cohomology}

\subsection{Quantum exterior superalgebra $\sqcap_q(m|n)$} 
Define the quantum exterior superalgebra $\sqcap_q(m|n)$ as follows.

\begin{defi}
The quantum exterior superalgebra $\sqcap_q(m|n)$ is the superalgebra generated by $d\xi_i$, $i\in I$, with relations
\begin{equation}
d\xi_i^2=0, \quad {\it for\  } \ i\in I, \quad {\it and } \quad d\xi_j\,\wedge_qd\xi_i=-q\,d\xi_i\,\wedge_qd\xi_j,  \quad \textit{ for } \ i, \,j\in I, \ j>i.
\end{equation}
Denote $|\,\mu\,|:=|\,d\xi^\mu\,|=\sum\mu_i$,\ $|\,\nu\,|:=|\,d\xi^{\nu}\,|=\sum\nu_i$, let $\sqcap_q(m|n)^{(s)}$ be the $s$-th homogeneous subspace of $\sqcap_q(m|n)$, that is,
\begin{equation}
\sqcap_q(m|n)^{(s)}=\text{Span}_\mathbb{K}\left.\left\{\,d\xi^\mu\, \wedge_q d\xi^{\nu}\, \right|\, \mu\in \mathds{Z}_2^m,\ \nu\in \mathds{Z}_2^n,\, |\,\mu\,|+|\,\nu\,|=s\,\right\},
\end{equation}
where the dual space $V^*=\text{Span}_{\mathbb K}\bigl\{\,d\xi_i\,\mid i\in I\,\bigr\}$ is the dual module of the natural 
$\mathcal U_q$-module $V=\text{Span}_{\mathbb K}\,\bigl\{\,x_i\,\mid i\in I\,\bigr\}$ by defining $d\xi_i(x_j)=\delta_{ij}$, and 
$(u.\,d\xi_i)(x_j)=d\xi_i\,(S(u).\,x_j)$, for any  $u\in \mathcal U_q$.

As an associative $\mathbb{K}$-superalgebra, the multiplication of $\sqcap_q(m|n)$ is given as follows.
\begin{equation}
(d\xi^\mu\wedge_q d\xi^{\nu})\wedge_q(d\xi^{\mu'}\wedge_q d\xi^{\nu'})
=(-q)^{\nu*\mu'+\mu*\mu'+\nu*\nu'}d\xi^{\mu+\mu'}\wedge_q d\xi^{\nu+\nu'}
\end{equation}
where $\mu,\,\mu'\in \mathds{Z}_2^m$,\ $\nu,\,\nu'\in \mathds{Z}_2^n$.
\end{defi}

\subsection{$q$-Differentials of type I over $\Omega_q(m|n)$ and quantum de Rham complex $\mathcal D_q(m|n)^{(\bullet)}$} 
Now we introduce the graded $q$-differentials $d$ of type I on $\Omega_q(m|n)$.
\begin{defi}
The $0$-degree $q$-differential $d^0$ on $\Omega_q(m|n)$ is the linear map
$$
d^0:\ \Omega_q(m|n)\longrightarrow \Omega_q(m|n)\otimes _\mathbb{K}\sqcap_q(m|n)^{(1)}
$$ 
defined by 
\begin{equation}
d^0(x^{(\alpha)}\otimes x^\nu)=\sum \limits_{i=1}^{m}q^{\varepsilon_i\ast\alpha}x^{(\alpha-\varepsilon_i)}\otimes x^\nu\otimes d\xi_i+\sum\limits_{i=m+1}^{m+n}q^{|\alpha|}q^{\varepsilon_i\ast\nu}\delta_{\nu_i,1}\bigl(x^{(\alpha)}\otimes x^{\nu-\varepsilon_i}\bigr)\otimes d\xi_i.
\end{equation}
In particular, $d^0(1)=0$, and $d^0(\xi_i)=1\otimes d\xi_i\equiv d\xi_i$.

The $q$-differentials $d^s$ of higher degrees $s\,(>0)$ are the following linear maps:
$$
d^s: \ \Omega_q(m|n)\otimes _\mathbb{K}\sqcap_q(m|n)^{(s)}\longrightarrow \Omega_q(m|n)\otimes _\mathbb{K}\sqcap_q(m|n)^{(s+1)},
$$ 
defined by
\begin{equation}
\begin{split}
d^s\bigl((&x^{(\alpha)}\otimes x^\nu)\otimes (d\xi^{\mu'}\wedge_q d\xi^{\nu'})\bigr) \\ &=\sum\limits_{j=1}^{m}q^{\varepsilon_j\ast\alpha}(x^{(\alpha-\varepsilon_j)}\otimes x^\nu)\otimes d\xi_j\wedge_q (d\xi^{\mu'}\wedge_q d\xi^{\nu'}) \\
&\quad +\sum\limits_{j=m+1}^{m+n}q^{|\alpha|}q^{\varepsilon_j\ast\nu}\delta_{\nu_j,1}(x^{(\alpha)}\otimes x^{\nu-\varepsilon_j})\otimes dx_j\wedge_q (d\xi^{\mu'}\wedge_q d\xi^{\nu'})\\
&=\sum\limits_{j=1}^{m}q^{\varepsilon_j\ast\alpha}\delta_{\mu'_j,0}(-q)^{\varepsilon_j\ast\mu'}
(x^{(\alpha-\varepsilon_j)}\otimes x^\nu)\otimes (d\xi^{\mu'+\varepsilon_j}\wedge_q d\xi^{\nu'}) \\
&\quad +\sum\limits_{j=m+1}^{m+n}q^{|\alpha|}q^{\varepsilon_j\ast\nu}(-q)^{\varepsilon_j\ast(\mu'+\nu')}\delta_{\nu_j,1}\delta_{\nu'_j,0}(x^{(\alpha)}\otimes x^{\nu-\varepsilon_j})\otimes  (d\xi^{\mu'}\wedge_q d\xi^{\nu'+\varepsilon_j}).
\end{split}
\end{equation}
\end{defi}

Set $\mathcal D_q(m|n)^{(\bullet)}:=\Omega_q(m|n)\otimes \sqcap_q(m|n)^{(\bullet)}$. Then we have 
\begin{theorem}
$\bigl(\,\mathcal D_q(m|n)^{(\bullet)}, d^\bullet\,\bigr)$ is a quantum de Rham cochain complex $($of type I\,$)$, that is, $d^{s+1}d^s=0$, for\, $0\le s\le m+n$.
\end{theorem}
\begin{proof}
According to formula (5.5), we can rewrite $d^s$ into the summation of tensor operators of quantum differential operators $_j\partial$ acting on $\Omega_q(m|n)$ and the left multiplication operators $d\xi_j$ of $\sqcap_q(m|n)$ for $j\in I$ as follows:
$$
d^s=\sum_{j\in I}
\,_j\partial\otimes d\xi_j, 
$$
where for $x^{(\alpha)}\otimes x^\nu\in \Omega_q(m|n)$, we define quantum differential operators:
\begin{equation}
_i\partial\,(x^{(\alpha)}\otimes x^\nu)=
\begin{cases}
q^{\varepsilon_i\ast\langle\alpha,\nu\rangle}\,x^{(\alpha-\varepsilon_i)}\otimes x^\nu, \quad &\textit{ for } \ i\in I_0,\\
q^{\varepsilon_i\ast\langle\alpha,\nu\rangle}\,x^{(\alpha)}\otimes x^{\nu-\varepsilon_i}, \quad &\textit{ for } \ i\in I_1.
\end{cases}
\end{equation}
So, it is easy to know that 
\begin{equation*}
\begin{split}
(_j\partial\,\circ\,_i\partial)\,(x^{(\alpha)}\otimes x^\nu)&=
\begin{cases}
q^{\varepsilon_i\ast\langle\alpha,\nu\rangle+\varepsilon_j\ast\langle\alpha{-}\varepsilon_i,\nu\rangle}
    \,x^{(\alpha-\varepsilon_i-\varepsilon_j)}\otimes x^\nu\quad &\textit{ for } \ i,\, j\in I_0,\\
q^{\varepsilon_i\ast\langle\alpha,\nu\rangle+\varepsilon_j\ast\langle\alpha{-}\varepsilon_i,\nu\rangle}
    \,x^{(\alpha-\varepsilon_i)}\otimes x^{\nu-\varepsilon_j}, \quad &\textit{ for } \ i\in I_0, \ j\in I_1,\\
  q^{\varepsilon_i\ast\langle\alpha,\nu\rangle+\varepsilon_j\ast\langle\alpha,\nu-\varepsilon_i\rangle}
    \,x^{(\alpha-\varepsilon_j)}\otimes x^{\nu-\varepsilon_i}, \quad&\textit{ for } \ i\in I_1,\ j\in I_0,\\
    q^{\varepsilon_i\ast\langle\alpha,\nu\rangle+\varepsilon_j\ast\langle\alpha,\nu-\varepsilon_i\rangle}
    \,x^{(\alpha)}\otimes x^{\nu-\varepsilon_i-\varepsilon_j}, \quad&\textit{ for } \ i,\,j\in I_1,
\end{cases}  \\
&=q^{\varepsilon_i\ast\varepsilon_j-\varepsilon_j\ast\varepsilon_i}(_i\partial\,\circ\,_j\partial)\,(x^{(\alpha)}\otimes x^\nu).
        \end{split}
\end{equation*}
So, we get $_j\partial\,\circ\,_i\partial=q^{\varepsilon_i\ast\varepsilon_j-\varepsilon_j\ast\varepsilon_i}\,_i\partial\,\circ\,_j\partial$.  Note that $\varepsilon_i\ast\varepsilon_j=0$, $\varepsilon_j\ast\varepsilon_i=1$, for $i<j$. Thereby,
we obtain
$$
_j\partial\,\circ\,_i\partial=q^{-1}\,_i\partial\,\circ\,_j\partial, \qquad\textit{for } \, j>i.
$$
Combining with (5.1), one can check: $d^{s+1}d^s=0$.

This completes the proof.
\end{proof}

From now on, we view $\mu=(\mu_1,\cdots,\mu_n)\in \mathds{Z}_2^n$ as $\mu=(0,\cdots,0;\mu_1,\cdots,\mu_n)$. For $m<i<m+n$, put $\mu+\varepsilon_i-\varepsilon_{i+1}=(0,\cdots,0;\mu_1,\cdots,\mu_i+1,\mu_{i+1}-1,\cdots,\mu_n)$, in brief.

%The following conclusion indicates our definition of the $q$-differental $d^0$ well-matches Proposition 4 (for details, see its proof in the Appendix).
%
%\begin{proposition}
%The $0$-degree $q$-differential $d^0$ is a $\mathcal U_q$-module homomorphism of parity $0$, that is, $d^0(u.\, x)=u.\, d^0(x)=u.\,d\xi$, for any $u\in %\mathcal U_q$, and $x\in \Omega_q(m|n)$, where we have identification $d^0(x)\equiv d\xi$. 
%\end{proposition}
%
%By Proposition 32, the $0$-degree $q$-differential $d^0$ is a $\mathcal U_q$-module homomorphism, it follows readily that $d^s\ (s=0,1,2,\cdots)$ are %homomorphisms of $\mathcal U_q$-modules. Note that $d^s(\sqcap_q(m|n))=0$. % (also for $u_q(\mathfrak{gl}(m|n))$). 

\smallskip
\subsection{Quantum de Rham subcomplex $(\mathcal{D}_q(m|n,\mathbf {r})^{(\bullet)}, d^\bullet)$ and its de Rham cohomologies}
In the case when $q$ is a root of unity of order $\ell$, one can introduce the following
\begin{defi}
Assume that $q$ is a root of unity of order $\ell$, one can define the truncated super-subspaces 
of $\mathcal D_q(m|n)$ as follows.
\begin{equation}
\mathcal{D}_q(m|n,\mathbf {r})=\text{Span}_\mathbb{K}\left.\left\{\,(x^{(\alpha)}\otimes x^\nu)\otimes (d\xi^{\mu'}\wedge_q d\xi^{\nu'})\in \mathcal{D}_q(m|n)\ \right|\  \alpha\le  \mathbf{r}\,\right\}.
\end{equation}
These are submodules of $\mathcal{D}_q(m|n)$. In particular, $\mathcal{D}_q(m|n,\mathbf {1})$ is a subalgebra of $\mathcal{D}_q(m|n)$.

Set $\mathcal{D}_q(m|n,\mathbf{r})^{(s)}=\bigl(\,\Omega_q(m|n)\otimes \sqcap_q(m|n)_{(s)}\bigr)\cap \mathcal{D}_q(m|n,\mathbf {r}\,)$, then 
$\mathcal{D}_q(m|n,\mathbf {r})=\bigoplus\limits_{s}\,\mathcal{D}_q(m|n,\mathbf {r})^{(s)}$.
One sees that $d^s$ sends  $\mathcal{D}_q(m|n,\mathbf {r})^{(s)}$ to $\mathcal{D}_q(m|n,\mathbf {r})^{(s+1)}$, so get a {\bf quantum de Rham cochain subcomplex:}
$\bigl(\,\mathcal{D}_q(m|n,\mathbf {r})^{(\bullet)}, d^\bullet\,\bigr)$ with $d^s=d^s\Big|_{\mathcal{D}_q(m|n,\mathbf {r})^{(s)}}$ 
\begin{equation}
0\rightarrow \mathcal{D}_q(m|n,\mathbf{r})^{(0)}\overset{d^0}\rightarrow\cdots\rightarrow\mathcal{D}_q(m|n,\mathbf{r})^{(s)}\overset{d^s}\longrightarrow \mathcal{D}_q(m|n,\mathbf{r})^{(s+1)}\rightarrow\cdots\rightarrow\mathcal{D}_q(m|n,\mathbf{r})^{(m+n)}\overset{d^{m+n}}\rightarrow0.  
\end{equation}
\end{defi}

In order to describe the cohomology groups of the quantum de Rham cochain subcomplex (5.6), let us define the super-weight and the super-weight subspace.
\begin{defi}
For $\bigl(x^{(\alpha)}\otimes x^{\nu}\bigr)\otimes \bigl(d\xi^{\mu'}\wedge_q d\xi^{\nu'}\bigr)\in \mathcal D_q(m|n, \mathbf r)^{(s)}$ with $\lambda=\langle\alpha, \nu\rangle\in\mathbb Z_+^m\times \mathds Z_2^n$, and $\mu=\langle\mu', \nu'\rangle\in\mathds Z_2^{m+n}$ and $|\mu|=|\,\mu'\,|+|\,\nu'\,|=s$, define its super-weight in  $\mathbb Z_+^m\times \mathds Z_2^n$as follows\,$:$
\begin{equation}
\textbf{swt}\Bigl(\bigl(x^{(\alpha)}\otimes x^{\nu}\bigr)\otimes \bigl(d\xi^{\mu'}\wedge_q d\xi^{\nu'}\bigr)\Bigr)=\textbf{wt}\Bigl(x^{(\alpha)}\otimes x^{\nu}\Bigr)-\textbf{wt}\Bigl(d\xi^{\mu'}\wedge_q d\xi^{\nu'}\Bigr)=\langle\alpha, \nu\rangle+\langle\mu', \nu'\rangle,
\end{equation}
where $\textbf{wt}\Bigl(d\xi^{\mu'}\wedge_q d\xi^{\nu'}\Bigr)=-\langle \mu', \nu'\rangle$ due to definition of $d\xi_j$ in Definition 29, and we consider $\mu_i$'s as $0$ or $1$ in $\mathbb Z_+$.

By abuse of notation, let us write, for $\tau$ with $\text{supp}(\tau)\cap\text{supp}(\nu)=\emptyset,$
$$
\Bigl(\bigl(x^{(\alpha-\mu')}{\otimes}\, x^{\nu-\nu'}\bigr){\otimes}\bigl(d\xi^{\mu'}\wedge_q  d\xi^{\nu'}\bigr)\Bigr){\odot}\bigr(x^{\tau}{\otimes}\, d\xi^{\tau}\bigr):=\bigl(x^{(\alpha-\mu')}{\otimes}\, x^{\nu-\nu'+\tau}\bigr){\otimes} \bigl(d\xi^{\mu'}\wedge_q  d\xi^{\nu'+\tau}\bigr).
$$
So, for any super-weight $\lambda=\langle\alpha,\nu\rangle\in \mathbb{Z}_+^m\times\mathds{Z}_2^n$, denote by $\mathcal{D}_\lambda^{(s)}$ the super-weight subspace of $\mathcal{D}_q(m|n,\mathbf {r})^{(s)}$ (of degree $s$) as below\,$:$ 
\begin{equation}
\begin{split}
\mathcal D_\lambda^{(s)}
&=\text{Span}_\mathbb{K}\Bigl\{\,\bigl(x^{(\alpha-\mu')}{\otimes}\, x^{\nu-\nu'+\tau}\bigr){\otimes} \bigl(d\xi^{\mu'}\wedge_q  d\xi^{\nu'+\tau}\bigr)\in \mathcal D_q(m|n,\mathbf r)^{(s)}\ \Big| \ \mu'\preceq \alpha, \ \nu'\preceq\nu,
  \\
&\quad    \ \text{supp}(\tau)\cap\text{supp}(\nu)=\emptyset,  \ s=|\,\mu'\,|+|\,\nu'\,|+|\,\tau\,|, \ 0\le |\,\tau\,|=t\le \min\{s,n{-}|\,\nu\,|\}\,\Bigr\}\\
&=\text{Span}_\mathbb{K}\Bigl\{\,\Bigl(\bigl(x^{(\alpha-\mu')}{\otimes}\, x^{\nu-\nu'}\bigr){\otimes}\bigl(d\xi^{\mu'}\wedge_q  d\xi^{\nu'}\bigr)\Bigr){\odot}\Bigr(x^{\tau}{\otimes}\, d\xi^{\tau}\Bigr)\in \mathcal D_q(m|n,\mathbf r)^{(s)}\ \Big| \ \mu'\preceq \alpha, \\ 
&\quad  \  \nu'\preceq\nu, \ \text{supp}(\nu')\cap\text{supp}(\tau)\subseteq \nu\cap\tau=\emptyset, \ |\,\mu'\,|+|\,\nu'\,|=s-t, \ |\,\tau\,|=t\,\Bigr\},
\end{split}
\end{equation}
where $\textbf{swt}\bigl(x^{\tau}\otimes d\xi^{\tau}\bigr)=\tau+\tau=\mathbf 0\in\mathds Z_2^n$, $\mu'=\sum\limits_{j=1}^{s_1}\varepsilon_{p_j}$ with $ p_1<\cdots<p_{s_1}$, and $\nu'=\sum\limits_{j=1}^{s_2}\varepsilon_{q_j}\preceq\nu$ with $q_1<\cdots< q_{s_2}$, and $s_1+s_2=s-t$, for $0\le t\le \min\{s,n{-}|\nu|\}$.

Simplifying notations: for any super-weight $\lambda=\langle \alpha,\nu\rangle\in \mathbb Z_+^m\times \mathds Z_2^n$, set 
\begin{gather}
\lambda_i=\begin{cases}
\alpha_i, \quad &\textit{ for } \ i\in I_0,\\
\nu_i, \quad &\textit{ for } \ i\in I_1.
\end{cases}
\end{gather}
Define $k_\lambda=\sharp\{i\in I_0\mid \lambda_i=r\,\}=k_\alpha$, $h^0_\lambda=\sharp\{i\in I_0\mid \lambda_i=0\,\}=h_\alpha$, $h_\lambda=\sharp\{i\in I\mid \lambda_i=0\,\}$,
$|\nu|=\sharp\{i\in I_1\mid \nu_i=1\}$.
By abuse of notation, view $\mu\in\mathds Z_2^{m+n}$ as a subsequence of $\lambda\in \mathbb Z_+^m\times \mathds Z_2^n$ when $\mu\preceq \lambda$ (that is, each $\mu_i\le \lambda_i$ for $i\in I$), and for any $\tau\subset \mathds Z_2^n\setminus\text{supp}(\nu)$, then rewrite a quantum exterior (monomial) differential form of a fixed super-weight $\lambda$ with differential degree $s$ as
\begin{gather}
x(\lambda{-}\mu|\tau)\xi(\mu|\tau)=x^{\lambda-\mu} x^\tau \otimes d\xi^{\mu+\tau}=x^{(\alpha-\mu')}\otimes x^{\nu-\nu'+\tau}\otimes d\xi^{\mu'}\wedge_q d\xi^{\nu'}\wedge_q d\xi^\tau,
\end{gather}
with $\mu=(\mu',\nu'),\, \tau=(\mathbf0,\tau)\in\mathds Z_2^{m+n}$, and $\mu\preceq \lambda$.
\end{defi}

\begin{lemma}
For a given super-weight $\lambda=\langle\alpha,\nu\rangle\in \mathbb{Z}_+^m\times\mathds{Z}_2^n$, we have
\begin{equation}
\begin{split}
\mathcal D_\lambda^{(s)}&=\bigoplus_{t=0}^{\min\{s,n-|\nu|\}}\bigoplus_{s_1=k_\alpha}^{s-t}\mathcal D_\lambda^{\langle s_1|s-t-s_1\rangle}\odot\mathcal D_{\mathbf 0}^{\langle t\rangle}
\\
&=\Bigl\{ \,x(\lambda{-}\mu|\tau)\xi(\mu|\tau)\,\Big|\, \mu\in\mathds Z_2^{m+n}, \ |\mu|+|\tau|=s, \ \mu\preceq\lambda, \,\tau\subset\mathds Z_2^n\setminus \text{supp}\,(\nu) \Bigr\}.
\end{split}
\end{equation}
where $D_\lambda^{\langle s_1|s-t-s_1\rangle}=\text{Span}_\mathbb K\bigl\{\bigl(x\bigl(\lambda{-}\mu|\mathbf0\bigr)\,\xi\bigl(\mu|\mathbf0\bigr)\in \mathcal D_\lambda^{(s)}\,\,\bigl|\,\,\, |\,\mu'\,|=s_1\ge k_\lambda, \, \nu\succeq\nu'\,\bigr\}$, $\mathcal D_{\mathbf 0}^{\langle t\rangle}=\text{Span}_{\mathbb K}\{\,x^{\tau}\otimes d\xi^{\tau}\mid \tau\subseteq\mathds Z_2^n\setminus\text{supp}\,(\nu) , \ |\,\tau\,|=t\,\}\subset\text{Ker}\,d\,^t$, with a convention: $\mathcal D_{\mathbf 0}^{\langle 0\rangle}=\mathbb K$, and 
$$
\mathcal D_\lambda^{\langle s_1|s-t-s_1\rangle}\odot\mathcal D_{\mathbf 0}^{\langle t\rangle}=\Bigl\{ x(\lambda{-}\mu|\tau)\xi(\mu|\tau)\,\Big|\,  |\mu'|=s_1,\, |\nu'|=s-t-s_1,\, |\tau|=t, \, \mu\preceq\lambda, \,\tau\subset\mathds Z_2^n\setminus \text{supp}\,(\nu) \Bigr\}.
$$
So,
\begin{equation}
\begin{split}
\dim\,\mathcal{D}_\lambda^{(s)}&= \sum\limits_{0\le t\le \min\{s,n-|\nu|\}}\,\sum\limits_{
 k_\alpha\le s_1\le s-t}\binom{m-k_\alpha-h_\alpha}{s_1-k_\alpha}\binom{|\nu|}{s-t-s_1}\binom{n{-}|\nu|}t\\
 &=\sum_{t=0}^{\min\{s,n{-}|\nu|\}}\binom{m-k_\alpha-h_\alpha+|\nu|}{s-t-k_\alpha}\binom{n{-}|\nu|}t=\binom{m+n-k_\lambda-h^0_\lambda}{s-k_\lambda}.
 \end{split}
\end{equation}  

$\mathcal{D}_\lambda^{(s)}\neq 0$ if and only if $k_\lambda\le s\,(\le m+n-h^0_\lambda)$.
 \end{lemma}
\begin{proof}
For a given super-weight $\lambda=\langle\alpha,\nu\rangle$, 
 by definition, we see that as vector spaces, 
$$
 \mathcal D_\lambda^{(s)}=\bigoplus_{t=0}^{\min\{s,n-|\nu|\}}\bigoplus_{s_1=k_\alpha}^{s-t}\mathcal D_\lambda^{\langle\, s_1\,|\,s-t-s_1\,\rangle}\odot\mathcal D_{\mathbf 0}^{\langle t\rangle},
 $$ 
$\dim \mathcal D_{\mathbf 0}^{\langle t\rangle}=\binom{\min\{s,n{-}|\,\nu\,|\}}{t}$,
 $\dim\, D_\lambda^{\langle\, s_1\,|\,s-t-s_1\,\rangle}=\binom{m-k_\alpha-h_\alpha}{s_1-k_\alpha}\binom{|\nu|}{s-t-s_1}$. According to constituents of differential forms in $\mathcal D_\lambda^{(s)}$, we have
\begin{equation*}
\begin{split}
\dim\,\mathcal{D}_\lambda^{(s)}&= \sum\limits_{0\le t\le \min\{s,n-|\nu|\}}\,\sum\limits_{
 k_\alpha\le s_1\le s-t}\binom{m-k_\alpha-h_\alpha}{s_1-k_\alpha}\binom{|\nu|}{s-t-s_1}\binom{n{-}|\nu|}t\\
&=\sum_{t=0}^{\min\{s,n{-}|\nu|\}}\binom{m-k_\alpha-h_\alpha+|\nu|}{s-t-k_\alpha}\binom{n{-}|\nu|}t=\binom{m+n-k_\lambda-h^0_\lambda}{s-k_\lambda}.
 \end{split}
\end{equation*}  
So, $\mathcal{D}_\lambda^{(s)}\neq 0$ if and only if $k_\lambda\le s\le m+n-h^0_\lambda$. 

By definition, it is clear that $x^\tau\otimes d\xi^\tau\in \text{Ker}\,d\,^t$ with $t=|\tau|$. 

This completes the proof.
\end{proof}

With the notation of Definition 33 in hand, formulas of (5.4) \& (5.5), for any $x(\lambda{-}\mu|\tau)\xi(\mu|\tau)\in \mathcal D_{\lambda}^{(s)}$ with super-weight $\lambda$ and differential degree $s=|\mu|+|\tau|$, can be briefly written as 
\begin{equation}
\begin{split}
d^s\bigl(x(\lambda{-}\mu|\tau)\xi(\mu|\tau)\bigr)& = %\sum_{\tau\subset\mathds Z_2^n,\, |\mu|+|\tau|=s}
\sum\limits_{j=1}^{m+n}q^{\varepsilon_j\ast(\lambda-\mu)}
x(\lambda{-}\mu{-}\varepsilon_j|\tau)\,\bigl(\,\xi(\varepsilon_j)\wedge_q \xi(\mu|\tau)\,\bigr)\\
&=
\sum\limits_{j=1}^{m+n}q^{\varepsilon_j\ast(\lambda-\mu)}(-q)^{\varepsilon_j\ast\mu}
x(\lambda{-}\mu{-}\varepsilon_j|\tau)\,\xi(\mu{+}\varepsilon_j|\tau)\\
&=\sum\limits_{j=1}^{m+n}q^{\varepsilon_j\ast\lambda}(-1)^{\varepsilon_j\ast\mu}
x(\lambda{-}\mu{-}\varepsilon_j|\tau)\,\xi(\mu{+}\varepsilon_j|\tau).
\end{split}
\end{equation}

\begin{defi}
Define a super-weight $\lambda=\langle\alpha,\nu\rangle\in \mathbb Z_+^m\times\mathds Z_2^n$ to be critical if it satisfies $k_\lambda+h_\lambda=m+n$. Otherwise, it is called non-critical.  
Denote by $\text{SW}_c$ the set of critical super-weights 
\begin{gather}
SW_c=\{\,\lambda_c=\langle\, \sum_{i=1}^{s-t}r\ell\varepsilon_{j_i},\,\mathbf0\,\rangle\mid 0\le j_1<j_2<\cdots<j_{s-t}\le m,\ 0\le t\le \min\{s, n\} \,\}.
\end{gather}
Denote by $SW_{nc}$ the set of non-critical super-weights. Then $SW_c\cup SW_{nc}=\mathbb Z_+^m\times\mathds Z_2^n$, where $SW_{nc}=\{\lambda=\langle\alpha,\nu\rangle\in \mathbb Z_+^m\times\mathds Z_2^n\mid k_\lambda+h_\lambda<m+n\,\}$.
\end{defi}

\begin{theorem} The quantum de Rham cochain subcomplex $\Bigl(\mathcal{D}_q(m|n,\mathbf{r})^{(\bullet)}, d^\bullet\Bigr)$
is non-acyclic, indeed, the $s$-th quantum de Rham cohomology group
\begin{equation*}
\begin{split}
H^s_{DR}\bigl(\mathcal{D}_q(m|n,\mathbf{r})^{(\bullet)}\bigr)&=\bigoplus_{\lambda\in \mathbb Z_+^m\times\mathds Z_2^n}\text{Ker}\,d^s\,|_{\mathcal D_\lambda^{(s)}}/\text{Im}\,d^{s-1}\,|_{\mathcal D_{\lambda}^{(s-1)}}=\bigoplus_{\lambda\in SW_c}\text{Ker}\,d^s\,|_{\mathcal D_\lambda^{(s)}}/\text{Im}\,d^{s-1}\,|_{\mathcal D_{\lambda}^{(s-1)}}\\
&\cong\bigoplus_{\{\tau\}\subset I_1, \,0\le |\tau|\le \min\{s,n\}}\bigoplus\limits_{1\leq i_1<i_2\cdots<i_{s_1}\leq m}\mathbb{K}\bigl[\,x\bigl(\sum_{j=1}^{s_1}(r\ell{-}1)\varepsilon_{i_j}|\tau\bigr)\xi\bigl(\sum_{j=1}^{s_1}\varepsilon_{i_j}|\tau\bigr)\,
\bigr],
\end{split}
\end{equation*}
where $s_1+|\tau|=s$, is highly nontrivial as $\mathbb{K}$-vector spaces, and $\dim\,H^s_{DR}(\mathcal{D}_q(m|n,\mathbf{r})^{(\bullet)})=\binom{m+n}{s}$, for $0\le s\le m+n$.
\end{theorem}
\begin{proof}
Since $\mathcal{D}_q(m|n,\mathbf{r})^{(s)}=\bigoplus_{\lambda=\langle\alpha,\nu\rangle\in\mathbb{Z}_+^m\times\mathds{Z}_2^n}\mathcal{D}_\lambda^{(s)}$ and each $d^s$ preserves super-weight gradings. It suffices to consider the subcomplex restricted to a super-weight $\lambda=\langle\alpha, \nu\rangle$.
$$
0\longrightarrow \mathcal{D}_\lambda^{(0)}\overset{d^0}\longrightarrow\cdots\overset{d^{s-1}}\longrightarrow\mathcal{D}_\lambda^{(s)}\overset{d^s}\longrightarrow \mathcal{D}_\lambda^{(s+1)}\overset{d^{s+1}}\longrightarrow\cdots\longrightarrow \mathcal{D}_\lambda^{(m+n)}\overset{d^{m+n}}\longrightarrow0.
$$

(I) For any critical super-weight $\lambda_c\in \text{SW}_c$, $\exists$ $t$, satisfying $0\le t\le \min\{s,n\}$, such that $\lambda_c=\langle\, \sum_{i=1}^{s-t}r\ell\varepsilon_{j_i},\,\mathbf0\,\rangle$, which corresponds to the canonical elements $x\bigl(\sum_{i=1}^{s-t}(r\ell{-}1)\varepsilon_{j_i}|\tau\bigr)\xi\bigl(\sum_{i=1}^{s-t}\varepsilon_{j_i}|\tau\bigr)$ of $\text{Ker}\,d^s$ (we call them the {\it critical 
quantum differential forms}), i.e.,
\begin{gather}
    x^{\bigl(\sum_{i=1}^{s-t}(r\ell-1)\varepsilon_{j_i}\bigr)}{\otimes} x_{g_1}\cdots x_{g_t}{\otimes} 
    d\xi_{j_1}\wedge_q \cdots\wedge_q d\xi_{j_{s-t}}\wedge_q d\xi_{g_1}\wedge_q\cdots\wedge_q d\xi_{g_t}\in
\mathcal D_{\lambda_c}^{\langle s-t|0\rangle}\odot\mathcal D_{\mathbf0}^{\langle t\rangle},
\end{gather}
which are not in the image of $d^{s-1}\,|_{\mathcal D_{\lambda_c}^{(s-1)}}$. So, 
$$
\dim\, H_{DR}^s\Bigl(\bigl(\bigoplus_{\lambda_c\in \text{SW}_c}\mathcal D_{\lambda_c})^{(\bullet)}\Bigr)=\sum_{t=0}^{\min\{s,n\}}\binom{m}{s-t}\binom{n}t=\binom{m+n}s. 
$$

(II) It suffices to show that $H_{DR}^s\bigl((\mathcal D_\lambda^{(s)})^{(\bullet)}\bigr)=0$ for any
$\lambda\in SW_{nc}$. Namely, we will prove that  
$$
\text{Im}\,d^{s-1}\,\Big|_{\mathcal{D}_\lambda^{(s-1)}}=\text{Ker}\,d^s\,\Big|_{\mathcal D_\lambda^{(s)}}, \quad \dim\,\text{Im}\,d^s\,|_{\mathcal D_\lambda^{(s)}}=\binom{m{+}n{-}k_\lambda{-}h_\lambda^0{-}1}{s{-}k_\lambda}, \quad \forall\, \lambda\in SW_{nc},\leqno{(\spadesuit)}
$$ 
by induction on $s$.

Note that $\mathbf0\in SW_c$. So, for $\lambda\in SW_{nc}$, $\lambda\ne\mathbf0$ and $\exists\,\, i_0\in I_0\cup I_1$, such that $0<\lambda_{i_0}<r$ for $i_0\le m$, or $\lambda_{i_0}=1$ for $i_0>m$.  

For $s=0=|\mu|+|\tau|$, we have $\mu=\tau=\mathbf0$, and 
 $\dim\,\mathcal D_q(m|n, \mathbf r)_\lambda^{(0)}=\dim\,\Omega_q(m|n, \mathbf r)_\lambda=1$. So by (5.14), we get
$d^0(x(\lambda|\mathbf0))=\sum \limits_{i=1}^{m+n}q^{\varepsilon_i\ast\lambda}x(\lambda{-}\varepsilon_i|\mathbf0)\,\xi(\varepsilon_i|\mathbf0)\ne 0$. So, $\dim\,\text{Im}\,d^0\,|_{\mathcal D_\lambda^{(0)}}=1$, and $\text{Ker}\,d^0\,|_{\mathcal D_\lambda^{(0)}}=0$, that is,
$H_{DR}^0\bigl(\mathcal D_\lambda^{(0)}\bigr)=0$ for any $\lambda\in SW_{nc}$.

For $s=1=|\mu|+|\tau|$, we have $\mu=\mathbf0$, $\tau=\varepsilon_j$ for $j>m$; or $\mu=\varepsilon_i$ for $i\in I$, $\tau=\mathbf0$.

%(II) Consider the action of $d^1$ on $\mathcal D_\lambda^{(1)}=\mathcal D_q(m|n, \mathbf %r)_\lambda^{(1)}$, for any super-weight $ %\lambda=\langle\,\alpha, \nu\,\rangle\ne\mathbf 0$.

$\bullet$ Case $(1|0):$ $s=1$ and $k_\lambda=0$: 

%(A)
%If $s_1=s_2=0$ and $t=1$, then taking $\lambda=\mathbf 0$,  $\mathcal D_{\mathbf 0}^{\langle 1\rangle}=\bigoplus_{j=m+1}^{m+n}\mathbb K\,x_j\otimes d\xi_j=\text{Ker}\,d^1$, which is not contained in the image of $d^0$. So $H_{DR}^1(\mathcal D_{\mathbf 0}^{(\bullet)})=\text{Ker}\,d^1$, and $\dim\,H_{DR}^1(\mathcal D_{\mathbf 0}^{(\bullet)})=n$. $\text{Im}\,d^1\,|_{\mathcal D_{\mathbf0}^{(1)}}=0$.

%(B) If $s_1=t=0$ and $s_2=1$, then taking $\lambda=\langle \mathbf0, \varepsilon_j\rangle$, for %any $j\in I_1$, then $\mathcal D_{\langle \mathbf0, %\varepsilon_j\rangle}^{\langle0|1\rangle}=\bigoplus_{j\in I_1}\mathbb K\, %d\xi_j=\text{Ker}\,d^1\,|_{\mathcal D_{\langle \mathbf0, %\varepsilon_j\rangle}^{\langle0|1\rangle}}=\text{Im}\,d^0\,|_{\mathcal D_{\langle \mathbf0, %\varepsilon_j\rangle}^{\langle0|0\rangle}}$, namely,
%$H_{DR}^1\bigl(\mathcal D_{\langle \mathbf0, \varepsilon_j\rangle}^{(\bullet)}\bigr)=0$. In this %case, 
%$\mathcal D_{\langle \mathbf0, \varepsilon_j\rangle}^{(1)}=\mathcal D_{\langle \mathbf0, %\varepsilon_j\rangle}^{\langle1|0\rangle}\bigoplus\mathcal D_{\langle \mathbf0, %\varepsilon_j\rangle}^{\langle0|1\rangle}
%\bigoplus \mathcal D_{\langle \mathbf0, \varepsilon_j\rangle}^{\langle0|0\rangle}\odot\mathcal %D_{\mathbf0}^{\langle1\rangle}=\mathbb K\, d\xi_j\bigoplus\bigoplus_{i\in I_1, i\ne j}\mathbb %K\,x_jx_i\otimes d\xi_i$, so $\dim\,\text{Im}\,d^1\,|_{\mathcal D^{(1)}_{\langle \mathbf0, %\varepsilon_j\rangle}}=\binom{n-1}1$.

For any $\lambda=\langle\,\alpha, \nu\,\rangle\in SW_{nc}$, i.e., $k_\lambda+h_\lambda<m+n$. 
For $k_\lambda=0<s=1$, we have $\dim\,\mathcal{D}_\lambda^{(0)}=1$ and $\dim\,\mathcal{D}_\lambda^{(1)}\neq 0$, according to Lemma 34.

Assume $0\neq\sum_{j=1}^{m}a_j(x^{(\alpha-\varepsilon_j)}\otimes x^\nu)\otimes d\xi_j+\sum_{j=m+1}^{m+n}b_j(x^{(\alpha)}\otimes x^{\nu-\varepsilon_j})\otimes d\xi_j+\sum_{j=m+1}^{m+n}c_jx^{(\alpha)}\otimes x^{\nu+\varepsilon_j}\otimes d\xi_j\in \text{Ker}\,d^1|_{\mathcal D_\lambda^{(1)}}$, by (5.12), $\mathcal D_\lambda^{(1)}=\Bigl(\bigl(\mathcal D_{\lambda}^{\langle 1|0\rangle}\bigoplus\mathcal D_{\lambda}^{\langle 0|1\rangle}\bigr)\odot\mathcal D_{\mathbf 0}^{\langle 0\rangle}\Bigr)\bigoplus \,\bigl(\mathcal D_\lambda^{\langle 0|0\rangle}\odot\mathcal D_{\mathbf 0}^{\langle 1\rangle}\bigr)$, and using (5.5), we have
\begin{equation*}
\begin{split}
&d^1\Bigl(\sum\limits_{j=1}^{m}a_j(x^{(\alpha-\varepsilon_j)}\otimes x^\nu)\otimes d\xi_j+\sum\limits_{j=m+1}^{m+n}b_j(x^{(\alpha)}\otimes x^{\nu-\varepsilon_j})\otimes d\xi_j+\sum_{j=m+1}^{m+n}c_jx^{(\alpha)}\otimes x^{\nu+\varepsilon_j}\otimes d\xi_j\Bigr)\\
&=\sum\limits_{1\leq i<j\leq m}\Bigl(a_jq^{\varepsilon_i\ast(\alpha-\varepsilon_j)}-q\,a_iq^{\varepsilon_j\ast(\alpha-\varepsilon_i)}\Bigr)
\bigl(x^{(\alpha-\varepsilon_i-\varepsilon_j)}\otimes x^\nu\bigr)\otimes \bigl(d\xi_i\wedge_q d\xi_j\bigr)\\
&\ +\sum\limits_{1\leq i\leq m<j\leq m+n}\Bigl(a_iq^{|\alpha-\varepsilon_i|}q^{\varepsilon_j\ast \nu}\delta_{\nu_j,1}(-q)+b_jq^{\varepsilon_i\ast\alpha}\Bigr)\bigl(x^{(\alpha-\varepsilon_i)}\otimes x^{\nu-\varepsilon_j}\bigr)\otimes \bigl(d\xi_i\wedge_q d\xi_j\bigr)\\
&\ +\sum\limits_{m+1\leq i< j\leq m+n}\Bigl(b_jq^{|\alpha|}q^{\varepsilon_i\ast(\nu-\varepsilon_j)}\delta_{\nu_i,1}+b_i(-q)q^{|\alpha|}q^{\varepsilon_j\ast(\nu-\varepsilon_i)}\delta_{\nu_j,1}\Bigr)\bigl(x^{(\alpha)}{\otimes}\, x^{\nu-\varepsilon_i-\varepsilon_j}\bigr){\otimes}\, \bigl(d\xi_i\wedge_q d\xi_j\bigr)\\
&\ +\sum\limits_{1\leq i\le m< j\leq m+n}c_jq^{\varepsilon_i\ast\alpha} x^{(\alpha-\varepsilon_i)}\otimes x^{\nu+\varepsilon_j}\otimes d\xi_i\wedge_q d\xi_j\\
&\ +\sum\limits_{m<i<j}c_jq^{\varepsilon_i\ast(\alpha+\nu)}x^{(\alpha)}\otimes x^{\nu+\varepsilon_j-\varepsilon_i}\otimes d\xi^{\varepsilon_i+\varepsilon_j} -\sum\limits_{m<i<j}c_iq^{\varepsilon_j\ast(\alpha+\nu)}x^{(\alpha)}\otimes x^{\nu+\varepsilon_i-\varepsilon_j}\otimes d\xi^{\varepsilon_i+\varepsilon_j}\\
&=0.
\end{split}
\end{equation*}
We obtain a system of equations with indeterminates $a_i \ (1\leq i\leq m)$ and $b_i \ (m{+}1\leq i\leq m{+}n)$:
\begin{equation}
\begin{split}
a_jq^{\varepsilon_i\ast(\alpha-\varepsilon_j)}-q\,a_iq^{\varepsilon_j\ast(\alpha-\varepsilon_i)}=0,& \quad \textit{for }\ 1\leq i<j\leq m;\\
a_iq^{|\alpha-\varepsilon_i|}q^{\varepsilon_j\ast\nu}\delta_{\nu_j,1}(-q)
+b_jq^{\varepsilon_i\ast\alpha}=0, &  \quad \textit{for }\ 1\leq i\leq m<j\leq m{+}n;\\
b_jq^{|\alpha|}q^{\varepsilon_i\ast(\nu-\varepsilon_j)}\delta_{\nu_i,1}
+b_i(-q)\,q^{|\alpha|}q^{\varepsilon_j\ast(\nu-\varepsilon_i)}\delta_{\nu_j,1}=0, &  \quad \textit{for }\ m{+}1\leq i<j\leq m{+}n;\\
c_j=0, &  \quad \textit{for }\ m{+}1\leq j\leq m{+}n.
\end{split}
\end{equation}
From (5.17), we get a unique solution: $c_j=0$ for $j\in I_1$ (this implies $\text{Ker}\,d^1\,|_{\mathcal D_\lambda^{\langle0|0\rangle}\odot\mathcal D_{\mathbf0}^{\langle 1\rangle}}=0$), and $a_j=a_1q^{\varepsilon_j\ast\gamma}$ for any $1\leq j\leq m$, $b_j=a_1q^{|\alpha|}q^{\varepsilon_j*\nu}\delta_{\nu_j,1}$ and $b_j\delta_{\nu_i,1}q^{\varepsilon_i\ast\nu}=b_i\delta_{\nu_j,1}q^{\varepsilon_j\ast\nu}$, for any $m{+}1\leq i<j\leq m{+}n$, which forms exactly the image of $d^0$ acting on $\Omega(m|n, \mathbf r)_\lambda$ (see (5.4) in Definition 30). Namely, $ \text{Ker}\,d^1\mid_{\mathcal{D}_\lambda^{(1)}}\subseteq \text{Im}\,d^0\mid_{\mathcal{D}_\lambda^{(0)}}$. But as $\text{Im}\,d^0\mid_{\mathcal{D}_\lambda^{(0)}}\subseteq \text{Ker}\,d^1\mid_{\mathcal{D}_\lambda^{(1)}}$, we get
$\text{Im}\,d^0\mid_{\mathcal{D}_\lambda^{(0)}}= \text{Ker}\,d^1\mid_{\mathcal{D}_\lambda^{(1)}}$, namely, $H_{DR}^1(\mathcal D^{(\bullet)}_\lambda)=0$. Moreover, $\dim\text{Ker}\,d^1\mid_{\mathcal{D}_\lambda^{(1)}}=1$, so
\begin{equation*}
\begin{split}
\dim\,\text{Im}\,d^1(\mathcal{D}_\lambda^{(1)})&=\dim\,\mathcal{D}_\lambda^{(1)}
-\dim\,\text{Ker}\,d^1\mid_{\mathcal{D}_\lambda^{(1)}}\\
&=\dim\,\mathcal D_{\lambda}^{\langle1|0\rangle}+\dim\,\mathcal D_{\lambda}^{\langle 0|1\rangle}+\dim\,\bigl(\mathcal D_\lambda^{\langle0|0\rangle}\odot\mathcal D_{\mathbf 0}^{\langle 1\rangle}\bigr)-1\\
&=\binom{m-h_\alpha}1+\binom{|\nu|}1+\binom{n{-}|\nu|}1-1=\binom{m+n{-}h_\lambda^0{-}1}{1}.
\end{split}
\end{equation*}

$\bullet$ Case $(1|1):$ $s=1$ and $k_\lambda=1$: $\dim\mathcal D_\lambda^{(1)}=1$, by Lemma 34.

%(A) $s_1=s=1$, $s_2=t=0$. Taking $\lambda=\langle r\ell\varepsilon_i, \mathbf0\rangle$, for any %$i\in I_0$, we have 
%$$
%\mathcal D_{\langle r\ell\varepsilon_i, \mathbf0\rangle}^{\langle1|0\rangle}=\mathbb K\, %x^{((r\ell-1)\varepsilon_i)}\otimes d\xi_i=\text{Ker}\,d^1\,\Big|_{\mathcal D_{\langle %r\ell\varepsilon_i, \mathbf0\rangle}^{\langle1|0\rangle}}, 
%$$ 
%which is not contained in the image of $d^0$. So $H_{DR}^1\bigl(\bigoplus_{i\in I_0}(\mathcal %D_{\langle r\ell\varepsilon_i, %\mathbf0\rangle}^{\langle1|0\rangle})^{(\bullet)}\bigr)=\bigoplus_{i\in I_0}\mathbb K\,x^{((r\ell-%1)\varepsilon_i)}\otimes d\xi_i$,
%$\dim\,H_{DR}^1\bigl((\bigoplus_{i\in I_0}\mathcal D_{\langle r\ell\varepsilon_i, %\mathbf0\rangle}^{\langle1|0\rangle})^{(\bullet)}\bigr)=m$. In this case, %$\text{Im}\,d^1\,|_{\mathcal D_{\langle r\ell\varepsilon_i, \mathbf0\rangle}^{(1)}}=0$.

$s=s_1=1$, $t=s_2=0$. $\lambda=\langle a_{i_1}\varepsilon_{i_1}+\cdots+a_{i_j}\varepsilon_{i_2}+\cdots+a_{i_k}\varepsilon_{i_k}, \varepsilon_{i_{k+1}}+\cdots+\varepsilon_{i_q}\rangle\in SW_{nc}$, for a unique $i_j\in I_0$, such that $a_{i_j}=r\ell$ and the others less than $r$, we have 
$$
\mathcal D_\lambda^{(1)}=\mathcal D_\lambda^{\langle1|0\rangle}=\mathbb K\, x^{(a_{i_1}\varepsilon_{i_1}+\cdots+(r\ell-1)\varepsilon_{i_j}+\cdots+a_{i_k}\varepsilon_{i_k})}\otimes x^{\varepsilon_{i_{k+1}}+\cdots+\varepsilon_{i_q}}\otimes d\xi_{i_j}, 
$$ 
which is contained in the image of $d^0$. Clearly, $\dim\,\text{Im}\,d^1\,|_{\mathcal D_\lambda^{(1)}}=1$, and $\dim\,H_{DR}^1\bigl((\mathcal D_{\mathcal D_\lambda^{(1)}})^{(\bullet)}\bigr)=0$.

Based on the above two cases, we check our claim $(\spadesuit)$ for $s=1$.

Now for $s>1$, suppose that the assertions $(\spadesuit)$ are true for $s'<s$, let us consider the case $s$.

Consider the behavior of $d^{s-1}$, $d^s$ at $\mathcal D_q(m|n, \mathbf r)_\lambda^{(\bullet)}$.

For $\lambda=\langle\alpha, \nu\rangle\in SW_{nc}$, 
by the Criterion of Lemma 34, $k_\lambda<s$ if and only if $\mathcal D_\lambda^{(s-1)}\ne 0$; $k_\lambda=s$ if and only if $\mathcal D_\lambda^{(s-1)}=0$
and $\mathcal D_\lambda^{(s)}\ne 0$; $k_\lambda>s$ if and only if $\mathcal D_\lambda^{(s)}=0\,(\,=\mathcal D_\lambda^{(s-1)}\,)$, and in this case, $H^s_{DR}\bigl(\mathcal{D}^{(\bullet)}_\lambda\bigr)=0$; 
if $s> m+n-h^0_\lambda\,(\ge k_\lambda)$,  Lemma 34 still gives $\mathcal{D}_\lambda^{(s)}=0$, thus $H^s_{DR}\bigl(\mathcal{D}^{(\bullet)}_\lambda\bigr)=0$.

So, it suffices to consider the case $k_\lambda\le s\le  m+n-h^0_\lambda$.

%Without loss of generality, we can assume that the non-critical super-weight $\lambda$ is of %the form: $0<\lambda_i<r$, for $1\le i\le p$, $\lambda_i=0$, for $p+1\le i\le p+e$, %$\lambda_i=r$, for $p+e<i\le m$, and $\lambda_i=1$ for $m+1\le i\le m+f$, $\lambda_i=0$ for %$m+f<i\le m+n$. 

For $0\le t\le\min\{s,n-|\nu|\}$, consider the kernel of $d^{s}\,\Big|_{\mathcal D_{\lambda}^{(s)}}$, 
for any $x(\lambda{-}\mu|\tau)\,\xi(\mu|\tau)\in \mathcal D_{\lambda}^{(s)}$ with $|\mu|+t=s, \,t=|\tau|$ ($\tau\subseteq\mathds Z_2^n\setminus \{\,\nu\,\}$), from 
$$
   d^s.\left(\sum_{\tau\subset\mathds Z_2^n\setminus\{\nu\}}\sum\limits_{\mu\preceq\lambda,\, |\mu|
=s{-}t}A_{\mu|\tau}\,x(\lambda{-}\mu|\tau)\,\xi(\mu|\tau)\right)=0,
$$
and using (5.16), we can obtain 
$$
\sum_{\tau\subset\mathds Z_2^n\setminus\{\nu\}}\sum_{|\mu|=s-t}\sum_{j=1}^{m+n}q^{\varepsilon_j\ast\lambda}(-1)^{\varepsilon_j\ast\mu}A_{\mu|\tau}\,x(\lambda{-}\mu{-}\varepsilon_j|\tau)\,\xi(\mu{+}\varepsilon_j|\tau)=0.\leqno{(\heartsuit)}
$$

Set $\mu(j)=\mu{+}\varepsilon_j$. 
Denote by $[\lambda]$ the sequence of length $m+n-h_\lambda$ composing of the increasing indexes $i$ with $\lambda_i\ne0$. Then the indexes $(\mu|\tau)$ of variables $A_{\mu|\tau}$ written as $(i_1\cdots i_s)$ are the subsequences of length $s$ of $[\,\lambda\cup (\mathds Z_2^n\setminus\text{supp}\,(\nu)\,]$ with $1\le i_1< i_2<\cdots<i_s\le m+n$. 

Let us collect those indexes subsequences together, namely we 
set 
\begin{gather}
P^s_{\lambda}=\{(i_1\cdots i_s)=(\mu|\tau)\subset [\,\lambda\cup (\mathds Z_2^n\setminus\text{supp}\,(\nu)\,]\mid x\bigl(\lambda{-}\mu|\tau\bigr)\,\xi\bigl(\mu|\tau\bigr)\ne0\,\}, \quad \text{\it and}\\
Q^{s+1}_\lambda=\{(i_1\cdots\, i_j\cdots\, i_s)=(\mu(i_j)|\tau)\subset [\,\lambda\cup (\mathds Z_2^n\setminus\text{supp}\,(\nu)\,]\mid x\bigl(\lambda{-}\mu(i_j)|\tau\bigr)\,\xi\bigl(\mu(i_j)|\tau\bigr)\ne0\,\}.
\end{gather}

Write $p=|P_\lambda^s|,\ q=|Q_\lambda^{s+1}|$. Order lexicographically the words in $P_\lambda^s$ and $Q^{s+1}_\lambda$ respectively in column to get two column vectors $P^s_\lambda,\ Q^{s+1}_\lambda$.
Thus from $(\heartsuit)$, we get
a system of $q$ linear equations with $p$ variables $A_{i_1\cdots i_s}$ $(1 \leq i_1<\cdots<i_s\leq m+n$): 
\begin{gather}
\sum\limits_{j=1}^{s+1}(-1)^{j-1}q^{\varepsilon_{i_j}\ast\lambda}A_{i_1\cdots\widehat{i_j}\cdots i_{s+1}}=0, \quad \text{\it for } \ (i_1\cdots i_{s+1})\in Q^{s+1}_\lambda \ \text{\it with } 
(i_1\cdots\widehat{i_j}\cdots i_{s+1})\in P^s_\lambda. 
\end{gather}
Write $X=(A_{i_1\cdots i_s})_{(i_1\cdots i_s)\in P^s_\lambda}$. Then we express the system above of $q$ linear equations with $p$ variables $A_{i_1\cdots i_s}$ as a matrix equation $CX=0$, where the coefficients matrix $C$ is of size $q\times p$.

When $k_\lambda\le s$, there exists a unique longest word $\imath_1\cdots \imath_{k_\lambda}$ such that $\lambda_{\imath_j}=r\ell$, for $1\leq j\leq k_\lambda$. By Lemma 34, $\imath_1\cdots \imath_{k_\lambda}$ must be a subword of any word $i_1\cdots i_s$ in $P^s_\lambda$, and each $\lambda_{i_j}\neq0$, for $1\leq j\leq s$. We set $b=\text{min}\{j\mid \lambda_j\neq 0\}$. Owing to the lexicographic order adopted in $X$, it is easy to see that there is a right-upper corner diagonal submatrix $\text{diag}\{q^{\varepsilon_b\ast\lambda},\cdots, q^{\varepsilon_b\ast\lambda}\}$ with order $\binom{m{+}n{-}k_\lambda{-}h^0_\lambda{-}1}{s{-}k_\lambda}$ in the right-upper corner of $C$, which is provided by the front $\binom{m{+}n{-}k_\lambda{-}h^0_\lambda{-}1} {s{-}k_\lambda}$ equations corresponding to those words in $Q^{s+1}_\lambda$ with the beginning letter $b$. Otherwise, the diagonal submatrix is $\text{diag}\,\{q^{\varepsilon_b\ast\lambda},\cdots, q^{\varepsilon_b\ast\lambda}\}$. Thus, $\text{rank}\,A\geq \binom {m{+}n{-}k_\lambda{-}h^0_\lambda{-}1}{s{-}k_\lambda}$, and 
\begin{equation*}
\begin{split}
\dim\,\text{Ker}\,d^{s}\,\left|_{\mathcal{D}_\lambda^{(s)}}\right.&=\dim\mathcal{D}_\lambda^{(s)}-\text{rank}\,A\\
&\leq
\binom{m{+}n{-}k_\lambda{-}h^0_\lambda}{s{-}k_\lambda}-\binom
{m{+}n{-}k_\lambda{-}h^0_\lambda{-}1}{s{-}k_\lambda}
=\binom {m{+}n{-}k_\lambda{-}h^0_\lambda{-}1}{s{-}k_\lambda{-}1}.
\end{split}
\end{equation*}
 By the inductive hypothesis, $\dim\,\text{Im}\,d^{s-1}\,\left|_{\mathcal{D}_\lambda^{(s-1)}}\right.=\binom {m{+}n{-}k_\lambda{-}h^0_\lambda{-}1}{s{-}k_\lambda{-}1}$, and $\text{Im}\,d^{s-1}\subseteq \text{Ker}\,d^s$, we get $\dim\,\text{Ker}\,d^s\,\left|_{\mathcal D_\lambda^{(s)}}\right.\geq \binom{m{+}n{-}k_\lambda{-}h^0_\lambda{-}1}{s{-}k_\lambda{-}1}$.
 Hence, we get $\dim\,\text{Ker}\,d^{s}\,\left|_{\mathcal{D}_\lambda^{(s)}}=\binom {m{+}n{-}k_\lambda{-}h^0_\lambda{-}1}{s{-}k_\lambda{-}1}\right.$ 
 $=\text{dim}\,\text{Im}\,d^{s-1}\,\left|_{\mathcal{D}_\lambda^{(s-1)}}\right.$, that is,
$$
\text{Ker}\,d^{s}\,\left|_{\mathcal{D}_\lambda^{(s)}}=\text{Im}\,d^{s-1}\,\right|_{\mathcal{D}_\lambda^{(s-1)}}.
$$
So, the cases $k_\lambda\le s$ with $m+n>k_\lambda+h_\lambda$ have no contribution to $H^s_{DR}(\mathcal{D}_q(m|n,\mathbf{r})^{(\bullet)})$.

This completes the proof.
\end{proof}

\subsection{Quantum de Rham cohomologies $H_{DR}^s\bigl(\mathcal{D}_q(m|n)^{(\bullet)}\bigr)$} 
In this final subsection, we give a description of the quantum super de Rham cohomologies for the quantum de Rham cochain complex $\bigl(\mathcal{D}_q(m|n)^{(\bullet)},d^\bullet\bigr)$. Actually, Lemma 34 and the result of Case (II) in the proof of Theorem 35 are still available to the quantum super de Rham cochain complex $\bigl(\mathcal D_q(m|n)^{(\bullet)}, d^\bullet\bigr)$. So, we can obtain the quantized version of the Poincar\'e Lemma on the de Rham complex as follows.

\begin{theorem} $($Poincar\'e Lemma$)$
Suppose $q$ is an $\ell$-th primitive root of unity. Then the quantum de Rham cochain complex $(\mathcal{D}_q(m|n)^\bullet,d^\bullet)$ over $\Omega_q(m|n):$
$$
0\longrightarrow \mathcal{D}_q(m|n)^{(0)}\overset{d^0}\longrightarrow \cdots \overset{d^{s-1}}\longrightarrow \mathcal{D}_q(m|n)^{(s)}\overset{d^s}\longrightarrow \mathcal D_q(m|n)^{(s+1)}\overset{d^{s+1}}\longrightarrow \cdots \longrightarrow \mathcal D_q(m|n)^{(m+n)}\overset{d^{m+n}}\longrightarrow 0
$$
is acyclic, that is, the quantum de Rham cohomology groups 
$$
H^0_{DR}\bigl(\mathcal{D}_q(m|n)^{(\bullet)}\bigr)=\text{$\mathbb K$}, \qquad H^s_{DR}\bigl(\mathcal{D}_q(m|n)^{(\bullet)}\bigr)=0, \qquad  {\it for } \  1\le s\le m+n.
$$
\end{theorem}
\begin{proof}
Obviously, $H^0_{DR}\bigl(\mathcal{D}_q(m|n)^{(\bullet)}\bigr)=\mathbb{K}$.

For any given $\lambda=\langle\,\alpha, \nu\,\rangle\in \mathbb{Z}^m\times\mathds Z_2^n$, since each $d^s$ preserves super-weight gradings, we have
$$
0\overset{d^{-1}}\longrightarrow\mathcal{D}_q(m|n)_\lambda^{(0)}\overset{d^0}\longrightarrow \cdots \overset{d^{s-1}}\longrightarrow \mathcal{D}_q(m|n)_\lambda^{(s)}\overset{d^s}\longrightarrow \mathcal{D}_q(m|n)_\lambda^{(s+1)}\overset{d^{s+1}}\longrightarrow \cdots\longrightarrow \mathcal{D}_q(m|n)_\lambda^{(m+n)}\overset{d^{m+n}}\longrightarrow  0.
$$

Since $H^s_{DR}(\mathcal{D}_q(m|n)^{(\bullet)})=\bigoplus_{\lambda\in \mathbb{Z}^m\times\mathds Z_2^n}\text{Ker}\,d^s\mid_{\mathcal{D}_q{(m|n)}_\lambda^{(s)}}/ \text{Im}\,d^{s-1}\mid_{\mathcal{D}_q(m|n)_\lambda^{(s-1)}}$, and for any given super-weight $\textbf{0}\neq\lambda=\langle\,\alpha, \nu\,\rangle\in \mathbb{Z}^m\times\mathds Z_2^n$, there exists  $r\in \mathbb{N}$, such that $r\ell>|\,\alpha\,|+m+n$, we see $\mathcal{D}_q(m|n)_\lambda^{(s)}
=\mathcal{D}_q(m|n,\textbf{r})_\lambda^{(s)}$, for $1\le s\le m+n$ and $k_\alpha=0$. This means that any nonzero super-weight $\lambda$ is non-critical in our sense. So, the criterion in Lemma 34 is adapted to our case, namely, $\mathcal{D}_q(m|n)_\lambda^{(s-1)}
=\mathcal{D}_q(m|n,\textbf{r})_\lambda^{(s-1)}\ne 0$, for $1\le s\le m+n$. According to the proof of Theorem 35, the result of Case (II) works, that is, $\text{Ker}\, d^s|_{\mathcal D_\lambda^{(s)}}=\text{Im}\,d^{(s-1)}|_{\mathcal D_\lambda^{(s-1)}}$ for any given non-critical super-weight $\lambda$. Notice that according to our choice of $r$, every super-weight $\lambda$ is always non-critical. 
This implies that $H_{DR}^{(s)}\bigl(\mathcal D_q(m|n)^{(\bullet)}_\lambda\bigr)=0$ for $1\le s\le m+n$.
\end{proof}


\bigskip
\begin{thebibliography}{99}

%\bibitem{BKK}
%G. Benkart, S.-J. Kang and M. Kashiwara, \textit{Crystal bases for the quantum superalgebra $U_q(\mathfrak{gl}(m,n))$}, J. Amer. Math. Soc. \textbf{13} (2) (2000), 295--331.

%\bibitem{BR}
%A. Berele and A. Regev, \textit{Hook Young diagrams with applications to combinatorics and representations of Lie superalgebras},
%Adv. in Math. \textbf{64} (1987), 118---175.

%\bibitem{FRT} L.D. Faddeev, N.Yu. Reshetikhin and L.A. Takhtajan, \textit{Quantization of Lie %groups and Lie algebras}, Leningrad Math. J. \textbf{1} (1990), 193---225.

\bibitem{DM}
P. Deligne; J.~W. Morgan, \textit{Notes on supersymmetry (following Joseph Bernstein)}, Quantum Fields and Strings: a Course for Mathematicians, Vol. \textbf{1, 2} (Princeton, NJ, 1996/1997), 41---97, Amer. Math. Soc., 1999.

\bibitem{F}
G. Feng, \textit{Representations of quantum groups and constructions of some admissible quantum affine algebras}, Ph.~D. thesis, East China Normal University, March 2021.

\bibitem{FHRZ}
G. Feng, N.~H. Hu, M. Rosso, X.~T. Zhang, \textit{Quantum super-symmetries (I): quantum Manin superspace, quantum Grassmann superalgebras and beyond}, 28 pages, Preprint 2023.

%\bibitem{G}
%J.A. Green, \textit{Polynomial Representations of $\textrm{GL}_n$}, Lecture Notes in Math. %\textbf{830}, 2nd corrected and augmented edition, with an Appendix on Schensted Correspondence %and Littelmann Paths, by K. Erdmann, J. A. Green and M. Schocker, Springer-Verlag Berlin %Heidelberg 2007.

\bibitem{GH}
H.~X. Gu and N.~H. Hu, \textit{Loewy filtration and quantum de Rham cohomology over quantum divided power algebra}, J. Algebra \textbf{435} (2015), 1---32.

\bibitem{Hu}
N.~~H. Hu, \textit{Quantum divided power algebra, $q$-derivatives, and some new quantum groups},
J. Algebra \textbf{232} (2000), 507---540.

\bibitem{Hum1}
J.E. Humphreys, \textit{Representations of Semisimple Lie Algebras in the BGG Category $\mathcal{O}$}, Graduate Studies in Math., vol. \textbf{94}, Amer. Math. Soc., Providence, RI, 2008.

\bibitem{Ir1}
R.~S. Irving, \textit{Projective modules in the category $\mathcal{O}_{S}$: Loewy series}, Trans. Amer. Math. Soc. \textbf{291}(2) (1985), 733---754.

\bibitem{Ir2}
---, \textit{The socle filtration of a Verma module}, Ann. Sci. {\'E}c. Norm. Sup{\'e}r. \textbf{21} (1) (1988), 47---65. 

%\bibitem{Kac1}
%V.~G. Kac, \textit{Lie superalgebras}, Adv. Math. \textbf{26} (1977), 8---96.

%\bibitem{KT}
%S.~M. Khoroshkin and V.~N. Tolstoy, \textit{Universal $R$-matrices for quantized %(super)algebras}, Comm. Math. Phys.
%\textbf{141} (1991), 599---617.

\bibitem{Lus}
G. Lusztig, \textit{Modular representations and quantum groups}, Contemp. Math. \textbf{82} pp. 59---77, 1989.

%\bibitem{Lus2}
%---, \textit{Introduction to Quantum Groups}, Birkhauser, Boston, Progress in Math. %\textbf{110}, 1993.

%\bibitem{Ma}
%S. Majid, \textit{Foundations of Quantum Groups Theory}, Cambridge University Press, 1995.

%\bibitem{M}
%\textit{Quantum Groups and Noncommutative Geometry}, Universit\'e de Montr\'eal, CRM, Montreal, %1988.
\bibitem{M}
Y.~I. Manin, \textit{Quantum Groups and Noncommutative Geometry}, Universit\'e de Montr\'eal, CRM, Montreal, 1988.

\bibitem{M1}
---, \textit{Notes on quantum groups and quantum de Rham complexes}, Teoret. Mat. Fiz. \textbf{92} (1992), 425---450.

\bibitem{M2}
---, \textit{Gauge Field Theory and Complex Geometry}, 2nd Edition, Springer, Berlin, 1997.

%\bibitem{Mon}
%S. Montgomery, \textit{Hopf Algebras and their Actions on Rings}, CBMS Regional Conf. Series in %Math. \textbf{82},
%Amer. Math. Soc. Providence, RI, 1993.

\bibitem{Mu}
I.~M. Musson, \textit{Lie Superalgebras and Envelopoing Algebras}, Graduate Studies in Math. \textbf{131}, Amer. Math. Soc. Providence, RI, 2012.

\bibitem{N}
S. Noja, \textit{On the geometry of forms on supermanifolds}, Diff. Geom. Appl. \textbf{88} (2023), No. 101999, 71 pages. 

\bibitem{N1}
---, \textit{On BV supermanifolds and the super Atiyah class}, Eur. J. Math. \textbf{9} (1) (2023), Paper No. 19, 36 pp.

\bibitem{SZ}
Y.~C. Su; R.~B. Zhang, \textit{Mixed cohomology of Lie superalgebras}, 
J. Algebra \textbf{549} (2020), 1---29.

%\bibitem{WZ}
%J. Wess and B. Zumino, \textit{Covariant differential calculus on the quantum hyperplane}, %Nuclear Phys. B Proc. Suppl. \textbf{18} (1990), 302---312.

%\bibitem{Wo}
%S.~L. Woronowicz, \textit{Differential calculus on compact matrix pseudogroups (quantum %groups)}, Comm. Math. Phys.
%\textbf{122} (1989), 125---170.

\bibitem{Y}
H. Yamane, \textit{Quantized enveloping algebras associated to simple Lie superalgebras and universal $R$-matrices},
Publ. RIMS, Kyoto Uni. \textbf{30} (1994), 15---84.

%\bibitem{Yet} D.~N. Yetter, \textit{Quantum groups and representations of monoidal categories}, %Math. Proc. Cambridge Philos. Soc. {\bf 108} (1990), 261---290.

%\bibitem{ZHn}
%J. Zhang and N.~H. Hu, \textit{Realization of $U_q(\mathfrak{sp}_{2n})$ within the differential %algebra on quantum symplectic
%space}, Symmetry, Integrability and Geometry: Methods and Applications, SIGMA \textbf{13} %(2017), 084, 21 pages.

\bibitem{Zh}
R.~B. Zhang, \textit{Quantum superalgebra representations on cohomology groups of non-commutative bundles}, J. Pure and Applied Algebra, \textbf{191} (2004), 285---314. 
\end{thebibliography}
\end{document}